\documentclass[reqno,11pt]{amsart}
\usepackage{amssymb, amsmath,latexsym,amsfonts,amsbsy, amsthm}
\usepackage{color}

\usepackage{tikz}
\usetikzlibrary{arrows,chains,matrix,positioning,scopes}
\makeatletter
\tikzset{join/.code=\tikzset{after node path={%
\ifx\tikzchainprevious\pgfutil@empty\else(\tikzchainprevious)%
edge[every join]#1(\tikzchaincurrent)\fi}}}
\makeatother
\tikzset{>=stealth',every on chain/.append style={join},
         every join/.style={->}}
\tikzstyle{labeled}=[execute at begin node=$\scriptstyle, execute at end node=$]

\allowdisplaybreaks[4]

\let\al=\alpha

\let\e=\varepsilon

\let\s=\sigma
\let\f=\frac
\let\p=\psi

\def\na{\nabla}

\def\p{\partial}

\newcommand{\beq}{\begin{equation}}
\newcommand{\eeq}{\end{equation}}
\newcommand{\ben}{\begin{eqnarray}}
\newcommand{\een}{\end{eqnarray}}
\newcommand{\beno}{\begin{eqnarray*}}
\newcommand{\eeno}{\end{eqnarray*}}

\newtheorem{theorem}{Theorem}[section]

\newtheorem{lemma}[theorem]{Lemma}
\newtheorem{proposition}[theorem]{Proposition}

\newtheorem{remark}[theorem]{Remark}

\newcommand{\ee}{\mathbf{e}}

\newcommand{\nn}{\mathbf{n}}

\def\bR{\mathbb R}
\def\bN{\mathbb N}

\def\t{\tilde}
\def\q{\quad}
\def\qq{\qquad}

\def\dl{\delta}
\def\Dl{\Delta}
\def\lt{\left}

\def\rt{\right}

\def\i{\infty}
\def\e{\epsilon}

\def \ls{\lesssim}
\def\p{\partial}
\def\f{\frac}
\def\na{\nabla}
\def\al{\alpha}

\def\O{\Omega}
\def\o{\omega}

\def\s{\sqrt}

\def\blue{\color{blue}}
\def\nn{\nonumber}
\def\be{\begin{equation}}
\def\ee{\end{equation}}
\def\bes{\begin{equation*}}
\def\ees{\end{equation*}}
\def\bali{\begin{aligned}}
\def\eali{\end{aligned}}

\def\pf{\noindent {\bf Proof. \hspace{2mm}}}

\def\bT{\mathbb T}

\begin{document}
\title[Vanishing viscosity limit in a strip]
{Vanishing viscosity limit of the Navier-Stokes equations in a horizontally periodic strip}

\author{Mingwen Fei}
\address[M. Fei]{School of  Mathematics and Statistics, Anhui Normal University, Wuhu 241002, China}
\email{mwfei@ahnu.edu.cn}

\author{Xinghong Pan}
\address[X. Pan]{School of Mathematics and Key Laboratory of MIIT, Nanjing University of Aeronautics and Astronautics, Nanjing 211106, China}
\email{xinghong\_87@nuaa.edu.cn}

\author{Jianfeng Zhao}
\address[J. Zhao]{Center for Nonlinear Studies, School of Mathematics, Northwest University, Xi'an 710069, China}
\email{zhaojianfeng@amss.ac.cn}

\begin{abstract}
  In this paper, we establish vanishing viscosity limit of the 2D Navier-Stokes equations in a horizontally periodic strip. On the vertical direction, the horizontal component of the velocity is subjected to two different types of boundary conditions: at the lower boundary, we give the degenerate zero boundary condition, while at the upper boundary, a small smooth perturbation of non-zero constant is prescribed. Due to this different boundary condition setting, the boundary layer effects are different near the lower and upper boundaries, which result in different thickness of boundary layer and different leading order boundary layer equations. We will construct an approximate solution to this 2D Navier-Stokes equations by using higher order asymptotic approximation and show the validity of the boundary layer expansion. The leading order of the Euler solution is the Couette flow $(Ay,0)$ for some suitable constant $A$, which is determined by using the principle of the Prandtl-Batchelor theory.
\end{abstract}



\date{\today}
\maketitle

\setcounter{equation}{0}
%%%%%%%%%%%%%%%%%%%%%%%%%%%%%%%%%%%%%%%%%%%%%%
%%%%%%%%%%%%%%%%%%%%%%%%%%%%%%%%%%%%%%%%%%

\numberwithin{equation}{section}
\section{Introduction}

\indent

We consider the following steady Navier-Stokes equations in a horizontally periodic strip $\O=\bT\times[0,1]$
\begin{equation}\label{ns}
\left \{
\begin {array}{ll}
(u^\e\p_x+v^\e\p_y)u^\e+\p_x p^\e-\e^2\Dl u^\e=0,\\ [5pt]
(u^\e\p_x+v^\e\p_y)v^\e+\p_y p^\e-\e^2\Dl v^\e=0,\\ [5pt]
\p_xu^\e+\p_y v^\e=0
\end{array}
\right.
\end{equation}
with boundary condition
\begin{align*}
&(u^\e,v^\e)\big|_{y=1}=(\alpha+\dl f(x),0),\q\ \text{for}\q x\in [0,2\pi),\\
&(u^\e,v^\e)\big|_{y=0}=(0,0),\qq\qq\q \text{for}\q x\in [0,2\pi),
\end{align*}
where  $\epsilon^2>0$ is reciprocal to Reynolds number, $\alpha>0$ is a constant, $(u^\e,v^\e)$ is the velocity, $p^\epsilon$ is the pressure, $\dl$ is a  small number, and $f(x)$ is a $2\pi$-periodic smooth function. Formally, as $\epsilon\rightarrow 0$, we obtain the following 2D steady Euler equations for $(u_e,v_e)$

\begin{equation}\label{equeuler}
\left \{
\begin {array}{ll}
(u_e\p_x+v_e\p_y)u_e+\p_x p_e=0,\\ [5pt]
(u_e\p_x+v_e\p_y)v_e+\p_y p_e=0,\\ [5pt]
\p_x u_e+\p_yv_e=0,
\end{array}
\right.
\end{equation}
with the boundary condition
\bes
v_e|_{y=0,1}=0, \q \text{for}\q x\in [0,2\pi).
\ees
We will show the existence of solution $(u^\epsilon, v^\epsilon)$ to (\ref{ns}) which converges to a solution of the steady Euler equations (\ref{equeuler}).

In the current work, we take the leading order Euler flow $(u_e(x,y),v_e(x,y))$ to be a shear flow.  Moreover, the Prandtl-Batchelor theory without a force implies that the only possible Euler flow is the following Couette flow
\begin{align}
u_e(x,y)=u_e(y), \ v_e(x,y)=0,\nn
\end{align}
where $u_e(y)=Ay+B$ with
\be
A=\Big(\alpha^2+\al\f{\dl}{2\pi}\int^{2\pi}_0f(x)dx+\frac{\dl^2}{2\pi}\int_0^{2\pi}f^2(x)dx\Big)^{\frac12},\q B=0. \label{Adef}
\ee
The constants $A,B$ are deduced by the Batchelor-Wood formula. One can refer to \cite{FGLT,DrivasIN:2023ARXIV} for its detailed derivation. We will also give an explanation later on. See Lemma \ref{lembw} and Section \eqref{sec431}.

Next, we introduce the leading order steady boundary layer equations near the upper boundary $y=1$. Denote the leading order term of the upper boundary layer expansions by $(u_p^{(0)},v_p^{(1)})$, then
\be\label{boundarylayeruppermain}
\left\{
\begin {array}{ll}
\big(A+u_p^{(0)}\big)\partial_x u_p^{(0)}+\big( v_p^{(1)}- v_p^{(1)}(x,0)\big)\partial_\zeta u_p^{(0)}-\partial^2_{\zeta}u_p^{(0)}=0,\\[5pt]
\partial_x u_p^{(0)}+\partial_\zeta v_p^{(1)}=0,\\[5pt]
u_p^{(0)}(x,\zeta)=u_p^{(0)}(x+2\pi,\zeta),\ v_p^{(1)}(x,\zeta)=v_p^{(1)}(x+2\pi,\zeta),\\[5pt]
u_p^{(0)}\big|_{\zeta=0}=\alpha+\dl f(x)-A,\\[5pt]
\lim\limits_{\zeta\rightarrow -\infty}u_p^{(0)}=\lim\limits_{\zeta\rightarrow -\infty}v_p^{(1)}=0,
\end{array}
\right.
\ee
where $\zeta:=\f{y-1}{\e}$ is the scaled variable. This boundary layer equations \eqref{boundarylayeruppermain} is derived by the procedure of matched asymptotic expansion and its well-posedness is stated in Proposition \ref{propdcu0}.

The following is the main result of our paper.
\begin{theorem}\label{thmain}
Assume that $f(x)$ is a smooth $2\pi$-periodic function. Then there exist two constants $\e_0$ and $\dl_0$ such that for any $\e\in(0,\e_0]$ and $\dl\in(0,\dl_0]$, the system \eqref{ns} has a solution $(u^\e,v^\e)$ satisfying
\bes
\bali
\lt\|u^\e(x,y)-Ay-u_p^{(0)}(x,\zeta)\rt\|_{L^\i(\O)}+\lt\|v^\e(x,y)\rt\|_{L^\i(\O)}\leq C\e,
\eali
\ees
where $C$ is independent of $\e$ and $\dl$, $A$ is the constant defined in \eqref{Adef} and $u_p^{(0)}(x,\zeta)$ is the solution of \eqref{boundarylayeruppermain}. \qed
\end{theorem}

 \begin{remark}
  In the literatures \cite{IyerZ:2019JDE, IyerZ:2019JDE1, FGLT, FeiGLT:2021ARXIV, GaoZ:2023SCM, GuoI:2023CPAM} and references therein, the authors considered the vanishing viscosity limits of the 2D Navier-Stokes equations in the box domain $(0,L)\times(0,1)$, in the strip domain $(0,L)\times(0,+\i)$, in  a disk, in an annulus and so on. Comparing our current work with theirs, the main novelty of the present paper lies in the below.
  \begin{itemize}
  \item \textbf{Setting of the thickness of the lower boundary layer.} Near the lower boundary $y=0$, where the the degenerate boundary condition is given, the thickness of the lower boundary layer is set to be size $\e^{2/3}$, which is different from the upper boundary layer, where the thickness of the boundary layer is of size $\e$, in the sense of the usual Prandtl boundary expansion. The consequence of this milder scaling is that the boundary layer equations at the lower boundary will become uniform with respect to the viscosity for lower and higher expansions. See Remark \ref{rem4.1} for a detailed explanation.
  \item \textbf{More subtle construction of the asymptotic expansion.} Since we have different scaling asymptotic expansions at the lower and upper boundary, it will be more subtle to construct a combining asymptotic expansion to deal with the coupling effect. Especially for the Euler part, we need to construct more detailed Euler expansions to balance the boundary conditions produced by the upper and lower boundary layer expansions.
  \item \textbf{Weighted $\e$-independence horizontal derivatives' estimates.} since the leading order Euler term is degenerate at the lower boundary, we will perform linear stability estimate in the weighted space with the Couette flow as a weight. Due to the weighted Sobolev embedding and Poincar\'{e} inequality, the $L^\i$ norm of the error term can also be obtained.
  \item \textbf{More physical boundary condition in $\boldsymbol{x}$ direction.} The boundary condition in the $x$ direction in our setting is periodic, which avoids the vertical boundary condition and hence the setting of the problem is more clear than that considered in the previous literature in the domain $(0,L)\times(0,1)$, where most of the boundary conditions on $x\in\{0,L\}$ need to be artificially set. However the periodicity in $x$ direction also brings some new difficulties, e.g., since we do not have the boundary condition on the vertical boundary, the  Poincar\'{e} inequality is invalid. To overcome this difficult we decompose the error of horizontal velocity into  the zero frequency part and the nonzero frequency part, and estimate them in different ways.
  \end{itemize}

\end{remark}

%\begin{remark}
% In the case of magnetohydrodynamics, the boundary layer theory is also very interesting and richer due to the choices of the magnetic physical parameters. We will address the vanishing viscosity limit and some related issues for the steady MHD system in a forthcoming paper.
%\end{remark}

 Due to the boundary condition's mismatch, rigorous justification of vanishing viscosity limit for the Navier-Stokes equations with non-slip boundary condition has always been an important and challenging problem. {See for example \cite{OleinikS:1999,SchlichtingG:2017}.}  In the case of the unsteady Navier-Stokes equations, to the best of our knowledge, the rigorous analysis of the Prandtl boundary
layer theory was given in the analytic framework
\cite{FTZ2018, KNVW2022, KVW2020, M2014, NN2018, SC1998-1, SC1998-2, WWZ2017, ZZ2016} and then  in the function spaces of Gevrey class \cite {CWZ2022, GMM2018, GM2015}. However, we can note that the validity for the Prandtl layer theory in Sobolev framework still remains unclear in this case. One can see  \cite{G2000,GGN2016,GN2019} for some recent progresses on the instability of Prandtl expansion of shear flow type in the Sobolev space. We also remark that the nonlinear stability analysis near Couette flow for the Euler and Navier-Stokes equations is also an important and difficult topic. See \cite{BedrossianM:2015PMI,BedrossianMV:2016ARMA,BedrossianVW:2018JNS,BedrossianH:2020CMP} and references therein for enhanced dissipation, inviscid damping, stability threshold and related results in the inviscid limit of the Navier-Stokes equations near the two dimensional Couette flow in domains $\bT\times\bR$ or $\bT\times[-1,1]$, respectively.

  The problem of inviscid limit of the steady Navier-Stokes equations seem to be  also very involved since there are infinitely many solutions for the steady Euler equations and what is the principle to find the suitable Euler flow is the first problem need to be settled. Actually, for general domain, it is still unknown how to choose a suitable Euler flow. For a simply-connected regions, Prandtl \cite{Prandtl:1904} considered the steady motion viscous incompressible fluid and found that the vorticity of steady Euler flow must be a constant in a region of nested closed streamlines.  This important property was rediscovered later by Batchelor and Wood in \cite{Batchelor:1956, Wood:1957JFM} and the exact value of this vorticity constant is also determined. This kind of result is now referred to as Prandtl-Batchelor theory in the literature. See some related researches on the Prandtl-Batchelor theory in \cite{Kim-thesis, Kim:1998SIAM, KimC:2001SIAM, WV:2007} and references therein.

The Prandtl-Batchelor theory can also be applied to our choice of the suitable Euler flow in our setting domain $\O=\bT\times[0,1]$. From the Prandtl-Batchelor theory in $\O$,  the shear Euler flow $(u_e(y),0)$ must satisfying the following ODE
\be\label{equationofleadingeuler}
\lt\{
\bali
&u_e^{''}(y)=0,\\
&u_e(1)=\Big(\alpha^2+\al\f{\dl}{2\pi}\int^{2\pi}_0f(x)dx+\frac{\dl^2}{2\pi}\int_0^{2\pi}f^2(x)dx\Big)^{\frac12},\q u_e(0)=0,
\eali
\rt.
\ee
which indicate that $u_e(y)=Ay$. The constants $A$ are given by solving the steady the nonlinear Prandtl equations in the Von Mises transformation. See Lemma \ref{lembw}.

In recent years, some important progresses are made for the asymptotic expansion and inviscid limit of the steady Navier-Stokes equations.  Guo-Nguyen \cite{GuoN:2017ANNPDE} considered the Prandtl boundary expansion of the steady Navier-Stokes equations in a moving plate in $(0,L)\times \bR_+$, where the boundary condition on $y=0$ for the horizontal component of the velocity to the Navier-Stokes equations is equal to some nonzero constant and there is a mismatch between this constant and the horizontal component of the shear Euler flow. Such a result was extended to the case of a rotating disk by Iyer \cite{Iyer:2017ARMA}. The case that the Euler flow is a perturbation of a shear flow can be found in Iyer \cite{Iyer:2019SIAM}. The Prandtl boundary layer expansion for motionless non-slip boundary condition, i.e. the velocity of the viscous flow equals to $(0,0)$ on $y=0$, was considered in Guo-Iyer \cite{GuoI:2023CPAM} with the Euler flow being a shear flow, which was extended to the non-shear flow case by Gao-Zhang in \cite{GaoZ:2021ARXIV}. For $x$-periodic domain, Gerard-Varet and Maekawa in \cite{GerardM:2019ARMA} showed a stability result of
the forced steady Navier-Stokes equations around some shear type profile. The above mentioned results are all considered in a narrow domain $L<<1$ or the period in $x$ direction is small.

For any domain with large width in $x$ direction, Iyer \cite{Iyer:2019Peking} showed the global Prandtl boundary expansion over a moving plate in $(0,+\i)\times(0,+\i)$, where the horizontal velocity of the viscous flow
on the boundary is set to be $1-\dl$ with small $\dl$ and the inviscid flow is the shear flow $(1,0)$. In the case of motionless no-slip boundary condition, Gao-Zhang \cite{GaoZ:2023SCM} showed boundary expansion in the domain $(0,L)\times(0,+\i)$ for any positive constant $L$, where the Euler flow is a shear flow and the Prandtl profile is concave in vertical direction.  Iyer-Masmoudi \cite{IyerM:2021AIA, IyerM:2020ARXIV} studied the stability of the shear Euler flow $(1,0)$ and the self-similar Blasius flow. The smallness assumption of $L$ in \cite{GerardM:2019ARMA} was recently improved to some spectral condition by Chen-Wu-Zhang in \cite{ChenWZ:2023SCM}. Iyer-Zhou \cite{IyerZ:2019JDE} established a density result for certain steady shear flows that vanishes at the vertical boundary of a box $(0,L)\times[0,2]$ for any large $L$.  Fei-Gao-Li-Tao in \cite{FGLT, FeiGLT:2021ARXIV} considered Prandtl boundary expansion in a disk and an annulus, where the boundary conditions for the swirl component of the velocity is a small perturbation of the rigid-rotation. Very recently, Gao-Xin in \cite{GX2023} considered the Prandtl boundary expansion in an infinitely long convergent channel, where they have proved the structure stability relies only the the assumption of a curvature-decreasing condition on boundary curves. Here we also mentions some results in \cite{DM:2019, WangZ:2021AIHP,ShenWZ:2021, GuoWZ:2023ANNPDE,WangZ:2023MATHANN} on the dynamic stability, asymptotic behavior, global $C^\i$ regularity and boundary layer separation of the steady Prandtl equations. For the steady viscous incompressible magnetohydrodynamics system,  there are several interesting studies for related problems, see for examples, \cite{DLX,DJL,LYZ2023}.

In this work, in order to prove Theorem \ref{thmain}, we firstly construct an approximate solution to the system \eqref{ns} by the method of asymptotic expansion and then estimate the error which satisfies the linearized Navier-Stokes equations around the constructed approximate solution. A brief view of the strategy of the proof will be given as follows.

{\bf\noindent Construction of an approximation solution}.

We first build an approximate solution $(u^a,v^a)$ by the asymptotic expansion, which contain the Euler part $(u^a_e, v^a_e)$, the upper boundary layer part $(u^a_p,v^a_p)$ and the lower boundary layer part $(\hat{u}^a_p,\hat{v}^a_p)$. The Euler part of the approximate solution satisfies the following estimates
\begin{align}
&\|\p^j_x\p^k_y(u^a_e-Ay)\|_\infty\leq C_{j,k}\epsilon(\dl+\epsilon),\q \|\p^j_x\p^k_yv^a_e\|_\infty\leq C_{j,k} \epsilon(\dl+\epsilon). \label{eulerapprox}
\end{align}
While for the upper boundary layer part, we have for any $j,k,\ell\in \{0\}\cup\bN$,
\begin{align}
\|\zeta^\ell\p^{j}_x\partial_\zeta^k u_p^a\|_\infty\leq C_{j,k,\ell} (\dl+\epsilon), \ \|\zeta^\ell\p^{j}_x\partial_\zeta^k v_p^a\|_\infty\leq C_{j,k,\ell} \epsilon(\dl+\epsilon), \label{prandtlapprox1}
\end{align}
and for the lower boundary layer part, we have for any $j,k,\ell\in \{0\}\cup\bN$,
\begin{align}
\|\eta^\ell\p^{j}_x\partial_\eta^k \hat{u}_p^a\|_\infty\leq C_{j,k,\ell}\e(\dl+\epsilon^{2/3}),  \ \|\eta^\ell\p^{j}_x\partial_\eta^k \hat{v}_p^a\|_\infty\leq C_{j,k,\ell} \epsilon^{1+2/3}(\dl+\epsilon^{2/3}),\label{prandtlapprox2}
\end{align}
where $\zeta:=\f{y-1}{\e}$ is the scaled variable near the upper boundary and $\eta:=\f{y}{\e^{2/3}}$ is the scaled variable near the lower boundary. Estimates \eqref{eulerapprox}, \eqref{prandtlapprox1} and \eqref{prandtlapprox2} will be frequently used in our next error estimates. The construction and the asymptotic behavior of the approximate solution are rather long and very technical. In order to keep fluent and show clearly the proof line of our paper, we postpone this construction to Section \ref{secappro}.

{\bf\noindent Weighted $\dot{H}^1$ error estimate of the linearized Navier-Stokes equations}

Denote the error function by
\beno
u:=u^\epsilon-u^a,\ v:=v^\epsilon-v^a,\ p:=p^\epsilon-p^a,
\eeno
where $p^a$ is the constructed pressure. Then the error function will satisfy the following linear forced Navier-Stokes equations
 \be\label{errorequation}
\left\{
\begin{array}{lll}
u^a \p_xu +v^a\p_yu+u \p_xu^a+v\p_yu^a+\p_xp-\epsilon^2\Dl u=R_u,\\[5pt]
u^a \p_xv +v^a\p_yv+u \p_x v^a+v\p_yv^a+\p_yp-\epsilon^2\Dl v=R_v,\\[5pt]
\p_xu+\p_yv=0,  \\[5pt]
u(x+2\pi,y)=u(x,y), \ v(x+2\pi,y)=v(x,y), \\[5pt]
u(x,0)=v(x,0)=u(x,1)=v(x,1)=0,
\end{array}
\right.
\ee
where $R_u$ and $R_v$ are the remainders.

We will give a weighted $\dot{H}^1$ estimate for the above linearized Navier-Stokes equations, which contain an $\e$ independently weighted estimate(also called positive estimate) and a basic energy estimate. The $\e$ independently weighted estimate comes from the the first term $u^a\p_x u$ and $u^a\p_x v$ in \eqref{errorequation}$_1$ and \eqref{errorequation}$_2$ respectively since we will later see that $u^a\thickapprox y$.

By multiplying  \eqref{errorequation}$_1$ with $u_x$ and  \eqref{errorequation}$_2$ with $v_x$ respectively and then integrating the resulted equations on $\O$ and using the incompressibility, we can achieve the following estimate
\be\label{linear1}
\bali
&\int_{\O}y\lt(u^2_x+ v^2_x\rt)dxdy\\
&\ls \int_{\O}\lt(v^a\p_y u u_x+u \p_xu^au_x+u \p_x v^a v_x\rt) dxdy+\cdots,
\eali\ee
where $``\cdots"$ represents the remainders terms and the higher $\e$ order terms, which can be easily handled. Noting that on the left hand of \eqref{linear1}, there is not weighted  $u_y$ estimate, which results in that we must pay careful attention on the estimate of the right hand term containing $u$ and $u_y$.

The control of the first term on the right hand of \eqref{linear1} comes from the basic energy estimates. Here we give a more subtle explanation. Denote the zero frequency of $u$ in $x$ by
$$u_0(y):=\f{1}{2\pi}\int^{2\pi}_{0}u(x,y)dx,$$
and then decompose $u$ by
$$
u=u_0+\bar{u},
$$
where  $\f{1}{2\pi}\int^{2\pi}_{0}\bar{u}(x,y)dx=0$. Then by integration by parts, we have

\begin{align*}
\int_{\O}u \p_xu^au_x dxdy=&-\f{1}{2}\int_\O u^a_{xx} u^2dxdy\\
                          =&-\f{1}{2}\int_\O u^a_{xx} (u^2_0+2u_0\bar{u}+\bar{u}^2)dxdy\\
                          =&-\f{1}{2}\int_\O u^a_{xx} (2u_0\bar{u}+\bar{u}^2)dxdy.
\end{align*}
In the last line of the above equality, we use the fact that $u_0$ is independent of $y$ and periodicity of $u$ and $u^a$ in $x$.

 Using the boundary condition $u^a|_{y=0}=0$, the differential mean value theorem, estimates \eqref{eulerapprox}, \eqref{prandtlapprox1} and \eqref{prandtlapprox2}, we have for function $f$

\begin{align*}
&|u^a_{xx}f|\ls \e\dl \lt({|f|}+y\f{|f|}{1-y}\rt), \text{ or } |u^a_{xx}|\leq \dl y,\\
& |v^a|+|v^a_x|\ls \e\dl y.
\end{align*}
 See Lemma \ref{lem2.3}. Then by the Cauchy inequality, we can obtain that
\begin{align*}
&\int_{\O}u \p_xu^au_x dxdy\\
&\leq \dl\int_{\O}(y\bar{u})^2dxdy+C\e^2\dl\int_{\O}\lt(\f{u_0^2}{y^2}+\f{u_0^2}{(1-y)^2}\rt)dxdy+C\dl\int_{\O}\bar{u}^2ydxdy\\
&\leq C\dl \int_{\O}u^2_x ydxdy+C\e^2\dl\int_{\O}u^2_{0,y}dxdy, \text{\q by the Poincar\'{e} inequality and Hardy inequality.}
\end{align*}
Also by using the Cauchy inequality, we have
\begin{align*}
&\int_{\O}v^a\p_y u u_x+u \p_x v^a v_xdxdy\\
\leq &C\dl \int_{\O}(u^2_x+v^2_x) ydxdy+C\e^2\dl\int_{\O}(u^2_{0,y}+\bar{u}^2)ydxdy.
\end{align*}
Then \eqref{linear1} is improved to
\bes
\bali
&\int_{\O}y\lt(u^2_x+ v^2_x+v^2_y\rt)dxdy\ls C\e^2\dl\int_{\O}(u^2_{0,y}+{u}^2_y y)dxdy+\cdots,
\eali
\ees
One see more details in Lemma \ref{lemlinearenergy}.

It is not easy to obtain the weighted $L^2$ estimate of $u_y$ by just multiplying \eqref{errorequation}$_1$ by $u_y$ and \eqref{errorequation}$_2$ by $v_y$ and performing the basic energy estimates. The reason is that such weighted energy estimates will cause non-cancellation of the pressure term. In order to avoid such a problem, we will reformulate our system \eqref{errorequation} in the form of stream function.  Existence of the stream function is a typical feature for the 2D incompressible flow. We will perform the $\e$ independently weighted energy estimate(also called positive estimate) for horizontal derivatives $u_x,v_x$ in the the framework of the stream function. One see more details in Lemma \ref{lempositive}.

At last,  by combining the positive estimate and basic energy estimate, we can obtain the following linear stability estimate
\bes
\bali
&\e^2\int_{\O}(u^2_{0,y}+{u}^2_y y)dxdy+\int_{\O}y\lt(u^2_x+ v^2_x+v^2_y\rt)dxdy\ls \cdots.
\eali
\ees
One see more details in Proposition \ref{proplinearstability}.

{\bf\noindent Partial $\dot{H}^2$ and $L^\i$ error estimate and existence}.

In order to obtain $L^\i$ estimate of the error function $(u,v)$, we still need partial $\dot{H}^2$ derivatives. Actually by performing  second order derivatives and integration by parts in \eqref{errorequation}, we can obtain all the second order derivatives estimates except for $u_{yy}$, which is enough for us to deduce the $L^\i$ estimate and close our energy.

 By using the deduced a priori estimates up to partial second order derivatives, the standard contraction mapping implies the existence of $L^\i$ solution to the system \eqref{errorequation}. One see more details in Proposition \ref{existence and error estimate of error equation}.

Throughout the paper, $C_{a,b,c,...}$ denotes a positive constant depending on $a,\,b,\, c,\,...$ which may be different from line to line. We also apply $A\lesssim_{a,b,c,\cdots} B$ to denote $A\leq C_{a,b,c,...}B$, while $A \approx_{a,b,c,...}B$ means $A\leq C_{a,b,c,...}B$ and $B\leq C_{a,b,c,...}A$. For a norm $\|\cdot\|$, we use $\|(f,g,\cdots)\|$ to denote $\|f\|+\|g\|+\cdots$.  For a function $f(x,y)$ and $1\leq p,q \leq +\i$,  define
\bes
\|f\|_{L^p_xL^q_y}:=\lt(\int^{1}_0 \lt(\int_{\bT} |f|^p dx\rt)^{q/p}dy \rt)^{1/q}.
\ees
If $p=q$, we simply write it as $\|f\|_{L^p}$ or $\|f\|_{p}$. If $p=q=+\infty$, we use $\|f\|_{\infty}$ to denote the essential supremum of $f$ in $\Omega$.

Our paper is organized as follows. In section \ref{sec2}, we first present the equations satisfied by the approximate solution and its asymptotic behavior. Then based on the approximate solution's property, we derive the error equations and establish the linear stability $H^1$ estimates of the error solution. In section \ref{sec3}, we obtain the partial $H^2$ and $L^\i$ estimates of the error solution and prove its existence by contraction mapping theorem. At last, in Section \ref{secappro}, we come back to give the construction of the approximate solution.
\section{$H^1$ linear stability estimates of the error equation}\label{sec2}

\indent
In this section, we derive the error equation and establish the $H^1$ linear stability estimate of the error equation.

First we present the following approximate solution $(u^a,v^a,p^a)$  which will be constructed in Section \ref{secappro}:

\begin{eqnarray}\label{app equationd}
\left\{
\begin{array}{ll}
u^a\p_xu^a+v^a\p_yu^a+\p_xp^a-\epsilon^2\Dl u^a=R_u^a, &(x,y)\in \bT\times[0,1],\\[5pt]
u^a\p_xv^a+v^a\p_yv^a+\p_y p^a-\epsilon^2\Dl v^a=R_v^a, &(x,y)\in \bT\times[0,1],\\[5pt]
\p_xu^a+\p_yv^a=0,  &(x,y)\in \bT\times[0,1], \\[5pt]
 u^a(x+2\pi,y)=u^a(x,y), \ v^a(x+2\pi,y)=v^a(x,y), &(x,y)\in \bT\times[0,1],\\[5pt]
u^a(x,1)=\alpha+\dl f(x),\ v^a(x,1)=0, &x\in [0,2\pi], \\[5pt]
 u^a(x,0)=v^a(x,0)=0, & x\in [0,2\pi].
\end{array}
\right.
\end{eqnarray}
where the forced terms $(R_u,R_v)$ satisfy the following estimates
 \begin{align*}
 \|R_u^a\|_2+\|\partial_x R_u^a\|_2\leq C\epsilon^9 , \  \|R_v^a\|_2+\|\partial_x R_v^a\|_2\leq C\epsilon^9.\label{remainderesti}
 \end{align*}

Besides we have the following lemma which will be shown in Section \ref{secappro} to state the asymptotic behavior of the approximate solution $(u^a,v^a)$ near and away from the boundary of the strip $\O$.
\begin{lemma}\label{lemdetail}
The constructed approximate solution $(u^a,v^a)$ is decomposed as
\be\label{uaconstruct}
u^a:=u^a_e+(1-\chi(y))^2u^a_p+\chi^2(y)\hat{u}^a_p+\e^9 h(x,y),\q v^a:=v^a_e+(1-\chi(y))^2v^a_p+\chi^2(y)\hat{v}^a_p,
\ee
where $\chi(y)\in C^\i_c([0,+\i))$ is a function satisfying
\bes
\chi(y)=\lt\{
\bali
&1,\q y\in [0,1/4],\\
&0,\q y\geq 3/4.
\eali
\rt.
\ees
The Euler part $(u^a_e,v^a_e)$ satisfies,
\begin{align}
&\|\p^j_x\p^k_y(u^a_e-Ay)\|_\infty\leq C_{j,k}\epsilon(\dl+\epsilon),\q \|\p^j_x\p^k_yv^a_e\|_\infty\leq C_{j,k} \epsilon(\dl+\epsilon). \label{eulerapproxd}
\end{align}
The upper boundary layer part $(u^a_p,v^a_p)$ satisfies
\begin{align}
\|\zeta^\ell\p^{j}_x\partial_\zeta^k u_p^a\|_\infty\leq C_{j,k,\ell} (\dl+\epsilon), \ \|\zeta^\ell\p^{j}_x\partial_\zeta^k v_p^a\|_\infty\leq C_{j,k,\ell} \epsilon(\dl+\epsilon), \label{prandtlapprox1d}
\end{align}
and the lower boundary layer part $(\hat{u}^a_p,\hat{v}^a_p)$ satisfies
\begin{align}
\|\eta^\ell\p^{j}_x\partial_\eta^k \hat{u}_p^a\|_\infty\leq C_{j,k,\ell} \e(\dl+\epsilon^{2/3}),  \ \|\eta^\ell\p^{j}_x\partial_\eta^k \hat{v}_p^a\|_\infty\leq C_{j,k,\ell} \epsilon^{1+2/3}(\dl+\epsilon^{2/3}).\label{prandtlapprox2d}
\end{align}
While $h(x,y)$ satisfies
\bes
h(x,0)=h(x,1)=0,\q \|\p^j_x\p^k_y h(x,y)\|_{L^\i}\leq C_{j,k}\e^{-k}.
\ees
\qed

\end{lemma}
Set the error function by
\beno
u:=u^\epsilon-u^a,\ v:=v^\epsilon-v^a,\ p:=p^\epsilon-p^a,
\eeno
Then there holds
\begin{align}\label{e:error equation}
\left\{
\begin{array}{ll}
-\epsilon^2\Dl u+\p_xp+S_u=R_u, &(x,y)\in \bT\times[0,1],\\[5pt]
-\epsilon^2\Dl v+\p_yp+S_v=R_v,&(x,y)\in \bT\times[0,1],\\[5pt]
\p_xu+\p_yv=0,  &(x,y)\in \bT\times[0,1],\\[5pt]
u(x+2\pi,y)=u(x,y), \ v(x+2\pi,y)=v(x,y),&x\in \bT, \\[5pt]
u(x,0)=v(x,0)=u(x,1)=v(x,1)=0, &x\in \bT,
 \end{array}
\right.
\end{align}
where
\begin{align*}
S_u:&=u^a \p_xu +v^a\p_yu+u \p_xu^a+v\p_yu^a,\\[5pt]
S_v:&=u^a \p_xv +v^a\p_yv+u \p_x v^a+v\p_yv^a,
\end{align*}
and the remainders
\beno
R_u:=R_u^a-(u\p_x+v\p_y)u,\ R_v:=R_v^a-(u\p_x+v\p_y)v.
\eeno

The main estimate for the linear Stokes system \eqref{e:error equation} is the following.
\begin{proposition}\label{proplinearstability}
 Let $(u,v)$ be a smooth solution of (\ref{e:error equation}).
 Define the $x-$mean of $u$ by
 \bes
 u_0=\f{1}{2\pi}\int^{2\pi}_0 u(x,y)dx.
 \ees
Then there exist $\epsilon_0>0, \dl_0>0$ such that for any $\epsilon\in (0,\epsilon_0), \dl\in(0,\dl_0)$, there holds
\begin{align}\label{linearstability}
&\e^2\int_{\bT\times[0,1]}\lt(u^2_{0,y}+yu^2_y\rt)dxdy+\int_{\bT\times[0,1]}\lt(u^2_{x}+v^2_x\rt)ydxdy\nn\\
&\leq C\lt|\int_{\bT\times[0,1]} \lt( R_u u_{x}+R_v v_{x}\rt)dxdy\rt|+C\int_{\bT\times[0,1]} \lt( yR^2_v+y^{\kappa} R^2_u\rt)dxdy+\lt|\int_{\bT\times[0,1]} R_u u_0dxdy\rt|.
\end{align}
Here $\kappa\in[0,1)$.
\end{proposition}
Proof of Proposition \ref{proplinearstability} contains a positive estimate in Section \ref{subspositive} to obtain the $\e$ independently horizontal derivative estimate for the second term on the left hand of \eqref{linearstability} and the energy estimate in Section \ref{subsenergy} to  obtain the vertical derivative estimate for the first term on the left hand of \eqref{linearstability}.

Before giving the proof of Proposition \ref{proplinearstability}, we give two useful lemmas which will be frequently used in the following proof. One is some detailed estimates for the approximate solution that we actually need  and the other is the Hardy inequality in one dimensional space.
{\blue
\begin{lemma} \label{lem2.3}
By using the boundary conditions  in \eqref{app equationd} and the asymptotic behaviors in \eqref{eulerapproxd}, \eqref{prandtlapprox1d}, \eqref{prandtlapprox2d} satisfied by $(u^a,v^a)$, we have the following estimates.
\be\label{approxdetailed}
u^a\thickapprox y, \q |\p^j_xu^a|\ls_j \dl y, \q {j\geq 1},
\ee
and
\be\label{approxdetailed1}
|\p^j_x v^a|\ls_j \e \dl y,\q |\p^j_xv^a_y|\ls_j \dl y, \q {j\geq 0}.
\ee
\end{lemma}
\begin{proof}
We first show that when $y\leq \f{1}{4},$ $|(u^a-Ay)_y|\leq C\e^{1/3}\delta$. From \eqref{uaconstruct}, $u^a=u^a_e+\hat{u}^a_p$ By using \eqref{eulerapproxd} and \eqref{prandtlapprox2d}, we have
\begin{align}
|(u^a-Ay)_y|=&|(u^a_e-Ay)_y|+|\hat{u}^a_{p,y}|+\e^9|h_y|\nn\\
             \ls&  \e\dl+\e^{-2/3}|\hat{u}^a_{p,\eta}|+\e^8\ls \e\dl+\e^{1/3}\dl +\e^8\ls \e^{1/3}\delta. \label{uathird}
\end{align}
So using the boundary condition that $(u^a-Ay)|_{y=0}$,
\be\label{uafirst}
|u^a-Ay|\leq\|(u^a-Ay)_y\|_{L^\i}y\leq C\e^{1/3}\dl y,\text{ for } y\in[0,1/4].
\ee
Also, from  \eqref{eulerapproxd}, \eqref{prandtlapprox1d} and \eqref{prandtlapprox2d}, it is easy to see that
\be\label{uasecond}
|u^a-Ay|\leq C\dl \leq C\dl y,\text{ for } y\in[1/4, 1].
\ee
Combining \eqref{uafirst} and \eqref{uasecond}, we see that $u^a\thickapprox y$.

For $1 \leq j\in\bN$, first by using \eqref{eulerapproxd} and \eqref{prandtlapprox2d}, we have
\begin{align}
|\p^j_xu^a|\leq|\p^j_x(u^a_e-Ay)|+(1-\chi(y))^2|\p^j_x {u}^a_p|+\chi^2(y)|\p^j_x \hat{u}^a_p|+|\p^j_xh|\ls \dl, \text{ for } y\in[1/4,1]. \label{uafour}
\end{align}
The same as \eqref{uathird}, we have $\|\p^j_x u^a_y\|_{L^\i}\leq  \e^{1/3}\delta$ for $y\in[0,1/4]$. Then by using the boundary condition $\p^j_xu^a|_{y=0}=0$ and mean value formula, we have
\begin{align}
|\p^j_xu^a|\leq \|\p^j_x u^a_y\|_{L^\i}y\leq \e^{1/3}\delta y, \text{ for }  y\in[0, 1/4]. \label{uafive}
\end{align}
The above two inequalities \eqref{uafour} and \eqref{uafive} indicate the second one of \eqref{approxdetailed}.

Estimates for $v^a$ in \eqref{approxdetailed1} is much simpler. first when $y\in [0,1/4]$
\bes
|\p^j_x v^a_y|=|\p^j_x \p_yv^a_e|+|\p^j_x \p_y\hat{v}^a_p|\leq \e\dl.
\ees
Using the boudary condition $\p^j_xv^a|_{y=0}$, mean value formula and incompressibility, we can see that
\begin{align}
|\p^j_x v^a|\leq\|\p^j_x\p_y v^a\|_{L^\i}y=\e\dl y, \text{ for } y\in[0,1/4].  \label{vafirst}
\end{align}
For $0 \leq j\in\bN$, by using \eqref{eulerapproxd} and \eqref{prandtlapprox2d}, we have
\begin{align}
|\p^j_x v^a|\leq |\p^j_x v^a_e|+(1-\chi(y))^2|\p^j_x {v}^a_p|+\chi^2(y)|\p^j_x \hat{v}^a_p|\ls \e\dl\ls \e\dl y, \text{ for } y\in[1/4,1]. \label{vasecond}
\end{align}
The above two inequalities \eqref{vafirst} and \eqref{vasecond} indicate the first one of \eqref{approxdetailed1}.

The second one of \eqref{approxdetailed1} is a direct consequence of the second one of \eqref{approxdetailed} and the incompressibility.
\end{proof}
}
Next, we present two types of the Hardy inequalities as follows.
\begin{lemma}
For function $f(y)\in C^1([0,1])$ with $f(1)=0$, we have
\be\label{hardy1}
\int^1_0 y^{\kappa} f^2 dy\leq \f{4}{(\kappa+1)^2}\int^1_0 y^{\kappa+2} (f_y)^2dy\q \text{for } \kappa>-1.
\ee
If additionally $f|_{y=0}=0$, we have,
\be\label{hardy2}
\int^1_0 \f{f^2}{y^{2}} dy \leq 4\int^1_0 (f_y)^2dy,\q \int^1_0 \f{f^2}{(1-y)^2} dy \leq 4\int^1_0 (f_y)^2dy.
\ee
\end{lemma}
\begin{proof}
Using integration by parts, the Cauchy inequality and H\"{o}lder inequality, we have
\begin{align}
\int^1_0 y^{\kappa} f^2 dy=&\f{1}{1+\kappa}\int^1_0  f^2 dy^{\kappa+1}=-\f{1}{1+\kappa}\int^1_0 2y^{\kappa+1} f f_y dy\nn\\
                       \leq & \f{2}{1+\kappa}\lt(\int^1_0 y^{\kappa} f^2 dy\rt)^{1/2}\lt(\int^1_0 y^{\kappa+2} (f_y)^2 dy\rt)^{1/2},\label{hardyproof1}
\end{align}
which indicates \eqref{hardy1}. For the first one in \eqref{hardy2}, also by using integration by part, we have
\begin{align}
\int^1_0 y^{-2} f^2 dy=&-\int^1_0  f^2 dy^{-1}=- f^2(y)y^{-1}\Big|^{1}_0+2\int^1_0 y^{-1} f f_y dy. \label{hardyproof2}
\end{align}

Using the facts $f(0)=0$ and $f(y)\in C^1([0,1])$ and the mean value formula, we see that
\bes
|f(y)|\leq\|\p_y f\|_{L^\i}y\ls y,
\ees
which indicate that $f^2(y)y^{-1}\big|_{y=0}=0$, then the same as \eqref{hardyproof1} by using Cauchy inequality, we have the first one in \eqref{hardy2} from \eqref{hardyproof2}. The second one of \eqref{hardy2} is similar with the derivation of the first one, we omit the details.
\end{proof}

The positive estimate and the weighted energy estimate will be carried out in the framework of the stream function. Here we give a reformulation of system \eqref{e:error equation} in the stream function.

From the incompressibility and $\int^{2\pi}_0 vdx=0$, there exists a $x$-periodic stream function $\phi$ such that
\bes
\na^{\perp}\phi=(\p_y,-\p_x)\phi=(u,v),
\ees
Actually we can take
\bes
\phi(x,y):=\int^y_1 u(x,\bar{y})d\bar{y},
\ees
which satisfies $\phi(x,1)=0$. The boundary condition of $v|_{y=0,1}=0$ indicates that
\bes
\p_x\phi(x,0)=-v(x,0)=0,
\ees
which indicates that $\phi(x,0)=constant$.

Define the vorticity $\o=\p_xv-\p_y {u}$, then $-\Dl \phi=\o$. From the first two equations of \eqref{e:error equation}, we have
\begin{align}\label{e:error equation2}
\left\{
\begin{array}{lll}
\epsilon^2\Dl^2\phi+(u^a\p_x+v^a\p_y)\o+({u}\p_x+v\p_y)\o^a=\p_xR_v-\p_yR_u,\\[5pt]
{u}=\p_y\phi,\q v=-\p_x\phi,\q \o=\p_x v-\p_y{u},  \\[5pt]
\phi\big|_{y=1}=0,\q \phi(x,0)=\text{constant}, \q\p_y\phi\big|_{y=0,1}=0, \q \phi(x+2\pi,y)=\phi(x+2\pi,y),
 \end{array}
\right.
\end{align}
where $\o^a:=\p_x v^a-\p_y u^a$ and we have used the following fact
\bes
\p_xS_v-\p_yS_u=(u^a\p_x+v^a\p_y)\o+(u\p_x+v\p_y)\o^a.
\ees

\subsection{The $\boldsymbol{\e}$ independently positivity estimate}\label{subspositive}

\indent

In this subsection, we give the weighted $\e$ independently horizontal derivatives' estimate for the solution $(u,v)$. We have the following lemma.

\begin{lemma}\label{lempositive}
 Let $(u,v)$ be a smooth solution of (\ref{e:error equation}), then there exist $\epsilon_0>0, \dl_0>0$ such that for any $\epsilon\in (0,\epsilon_0), \dl\in(0,\dl_0)$, there holds
\begin{align}\label{estipositive}
&\int_{\bT\times[0,1]}\lt(yu^2_{x}+yv^2_x\rt)dxdy\nn\\
&\leq C\e^2\dl\int_{\bT\times[0,1]}\lt(u^2_{0,y}+yu^2_y\rt)dxdy+C\lt|\int_{\bT\times[0,1]} (R_u u_{x}+R_v v_{x})dxdy\rt|.
\end{align}
\end{lemma}

\pf Using the divergence-free condition of the approximation solution, we have
\bes
\p_x \o^a=\Dl v^a,\q \p_y \o^a=-\Dl u^a.
\ees
Inserting this into \eqref{e:error equation2}, we have
\be\label{e:error equation3}
\left\{
\bali
&\epsilon^2\Dl^2\phi-u^a\p_x\Dl\phi-v^a\p_y\Dl\phi+\p_y\phi\Dl v^a+\p_x\phi\Dl u^a=\p_xR_v-\p_yR_u,\\[5pt]
&\phi\big|_{y=0}=\text{const.},\q \phi\big|_{y=1}=0,\q \p_y\phi\big|_{y=0,1}=0, \q \phi(x+2\pi,y)=\phi(x+2\pi,y).
 \eali
\right.
\ee
Now we multiply \eqref{e:error equation3}$_1$ by $\p_x\phi$ and integrate on $\bT\times[0,1]$ to obtain that
\begin{align}
&-\int_{\bT\times[0,1]}u^a\Dl\p_x\phi \p_x\phi \nn\\
=&\underbrace{-\epsilon^2\int_{\bT\times[0,1]}\Dl^2\phi \p_x\phi}_{I_1}+\underbrace{\int_{\bT\times[0,1]}\lt(v^a\p_y\Dl\phi-\p_y\phi\Dl v^a\rt)\p_x\phi}_{I_2}\label{errorpositive1}\\
 &\underbrace{-\int_{\bT\times[0,1]}\p_x\phi\Dl u^a \p_x\phi}_{I_3}+\underbrace{\int_{\bT\times[0,1]} (\p_xR_v-\p_yR_u)\p_x\phi}_{I_4}.\nn
\end{align}
Here and later, for simplicity of notations, we omit the integrating variables $dxdy$ if no ambiguity is caused.  For the left-hand of \eqref{errorpositive1}, by using integration by parts, we can have
\begin{align}
-\int_{\bT\times[0,1]}u^a\Dl\p_x\phi \p_x\phi=&\int_{\bT\times[0,1]} u^a(\phi_{xx})^2+ u^a_x \phi_{xx}\phi_x+\int_{\bT\times[0,1]} u^a(\phi_{xy})^2+ u^a_y \phi_{xy}\phi_x\nn\\
       =&\int_{\bT\times[0,1]} u^a\lt[(\phi_{xx})^2+(\phi_{xy})^2\rt]-\f{1}{2}\int_{\bT\times[0,1]} \lt(u^a_{xx}+ u^a_{yy}\rt) (\phi_{x})^2. \label{errorpositive2}
\end{align}

From the second estimate in \eqref{approxdetailed}, we have
\begin{align}
\lt|\int_{\bT\times[0,1]} u^a_{xx} (\phi_{x})^2\rt|\ls &\dl\int_{\bT\times[0,1]} y(\phi_{x})^2\nn\\
\ls & \dl\int_{\bT\times[0,1]} y^{3} (\phi_{xy})^2\ls  \dl\int_{\bT\times[0,1]} y (\phi_{xy})^2.\label{errorpositive4}
\end{align}
Here in the last line, we use Hardy inequality \eqref{hardy1} since $\phi_x|_{y=1}=0$.

Also from Lemma \ref{lemdetail}, we have
\begin{align}
&\lt|\int_{\bT\times[0,1]}  u^a_{yy} (\phi_{x})^2\rt|\nn\\
\ls & \int_{\bT\times[0,1]} (|(u^a_e-Ay)_{yy}|+|[(1-\chi(y))^2u^a_p]_{yy}|+|[\chi^2(y)\hat{u}^a_p]_{yy}|+h_{yy}) (\phi_{x})^2\nn\\
\ls & \int_{\bT\times[0,1]} \lt\{\e\dl  (\phi_{x})^2 +(1-y)^2|[(1-\chi(y))^2u^a_p]_{yy}| \f{(\phi_{x})^2}{(1-y)^2}+y^{1/2}|[\chi^2(y)\hat{u}^a_p]_{yy}|y^{-1/2}(\phi_{x})^2\rt\}\nn\\
\ls & \int_{\bT\times[0,1]} \lt\{\e\dl  (\phi_{x})^2 +\zeta^2|[(1-\chi(y))^2u^a_p]_{\zeta\zeta}| \f{(\phi_{x})^2}{(1-y)^2}+\eta^{1/2}\e^{-1}|[\chi^2(y)\hat{u}^a_p]_{\eta\eta}|y^{-1/2}(\phi_{x})^2\rt\}\nn\\
\ls & \int_{\bT\times[0,1]} \lt\{\e\dl  (\phi_{x})^2 +\dl (\phi_{x})^2+  \dl \f{(\phi_{x})^2(1-\chi(y))^2}{(1-y)^2}+\dl y^{-1/2}(\phi_{x})^2\rt\}\nn\\
 \ls &  \dl\int_{\bT\times[0,1]} y (\phi_{xy})^2, \label{errorpositive4p}
\end{align}
 In the last line, we used Hardy inequality \eqref{hardy1} and \eqref{hardy2} since $\phi_x|_{y=0,1}=0$.

Inserting \eqref{errorpositive4} and \eqref{errorpositive4p} into \eqref{errorpositive2}, we can obtain that

\begin{align}
&-\int_{\bT\times[0,1]}u^a\Dl\p_x\phi \p_x\phi \nn\\
&\geq \int_{\bT\times[0,1]} u^a\lt[(\phi_{xx})^2+(\phi_{xy})^2\rt]-C\dl\int_{\bT\times[0,1]}y (\phi_{xy})^2. \label{errorpositive4f}
\end{align}

Now we estimate the righthand of \eqref{errorpositive1} term by term.

{\noindent\bf Term $I_1$}: By using integration by parts and the boundary condition $(\phi_x,\phi_{xy})|_{y=0,1}=0$, it is easy to see that
\be\label{errorpositive4s}
I_1=\f{\e^2}{2}\int_{\bT\times[0,1]}\lt((\phi_{xx})^2+(\phi_{yy})^2\rt)_x =0.
\ee

{\noindent\bf Term $I_2$}:  By integration by parts and the boundary condition $v^a|_{y=0,1}=0$ and $\p_x\phi|_{y=0,1}=0$, we have
\begin{align}
I_2=&-\int_{\bT\times[0,1]}\phi_{xy}\p_x(v^a\phi_x)+\phi_{yy}\p_y(v^a\phi_x)\nn\\
    &+\int_{\bT\times[0,1]}v^a_{x}\p_x(\phi_y\phi_x)+v^a_{y}\p_y(\phi_y\phi_x).\nn
\end{align}
After cancellation, using integration by parts and incompressibility of $u^a_x+v^a_y=0$, we can obtain that
\begin{align}
I_2=&-\int_{\bT\times[0,1]}\phi_{xy}v^a\phi_{xx}+\phi_{yy}v^a\phi_{xy}+\int_{\bT\times[0,1]}v^a_{x}\phi_y\phi_{xx}+v^a_{y}\phi_y\phi_{xy}\nn\\
    \leq& \int_{\bT\times[0,1]} |(v^a,v^a_x)| |(\phi_{y},\phi_{xy})||(\phi_{xx},\phi_{yy})|-\f{1}{2}\int_{\bT\times[0,1]}v^a_{xy} (\phi_y)^2\nn\\
    \leq& \int_{\bT\times[0,1]} |(v^a,v^a_x)| |(\phi_{y},\phi_{xy})||(\phi_{xx},\phi_{yy})|+\f{1}{2}\int_{\bT\times[0,1]}u^a_{xx} (\phi_y)^2.\nn
\end{align}

Then by using the first one of \eqref{approxdetailed1} and the Cauchy inequality, we can achieve that
\begin{align}
I_2\leq& \dl\int_{\bT\times[0,1]} y \lt[(\phi_{xx})^2+(\phi_{xy})^2\rt]+\e^2\dl\int_{\bT\times[0,1]}y(\phi_{yy})^2+\f{1}{2}\int_{\bT\times[0,1]}u^a_{xx} (\phi_y)^2.\label{errorpositive6}
\end{align}

Now we come to estimate the third term on the right hand of \eqref{errorpositive6}.  Denote
\bes
\phi_0=\f{1}{2\pi}\int^{2\pi}_0\phi dx,\q \bar{\phi}=\phi-\phi_0.
\ees
Then we have
\begin{align}
\f{1}{2}\int_{\bT\times[0,1]}u^a_{xx} (\phi_y)^2=&\f{1}{2}\int_{\bT\times[0,1]}u^a_{xx} (\phi^2_{0,y}+2\phi_{0,y}\bar{\phi}_y+\bar{\phi}^2_y)\nn\\
   =&\f{1}{2}\int_{\bT\times[0,1]}u^a_{xx} (2\phi_{0,y}\bar{\phi}_y+\bar{\phi}^2_y),\label{errorpositive7}
\end{align}
where in the last line, by using the fact that $\phi_0$ is independent of $x$ and periodicity of $x$, we have
\bes
\int_{\bT\times[0,1]} u^a_{xx}\phi^2_{0,y}=0.
\ees

Using the estimates in Lemma \ref{lemdetail}, Cauchy inequality, Hardy inequality \eqref{hardy2} and Poincar\'{e} inequality in $x$ direction, we have
\begin{align}
&\int_{\bT\times[0,1]}u^a_{xx} \phi_{0,y}\bar{\phi}_y\nn\\
=& \int_{\bT\times[0,1]}\lt\{(u^a_e-Ay)_{xx}+(1-\chi(y))^2(u^a_p)_{xx}+\chi^2(y)(\hat{u}^a_p)_{xx}+h_{xx}\rt\} \phi_{0,y}\bar{\phi}_y\nn\\
\ls& \int_{\bT\times[0,1]}\lt\{\e\dl+\e^8 \rt\}|\phi_{0,y}\bar{\phi}_y|+(1-y)(1-\chi(y))^2(u^a_p)_{xx} \f{\phi_{0,y}}{1-y}\bar{\phi}_y\nn\\
\ls& \e\dl \int_{\bT\times[0,1]}y^{-1}|\phi_{0,y}|y|\bar{\phi}_y| +y\f{|\phi_{0,y}|}{1-y}|\bar{\phi}_y|\nn\\
   \ls & \e^2\dl \int_{\bT\times[0,1]}(y^{-1}\phi_{0,y})^2+((1-y)^{-1}\phi_{0,y})^2+\dl\int_{\bT\times[0,1]}y^2(\bar{\phi}_y)^2\label{errorpositive8}\\
   \ls &\e^2\dl \int_{\bT\times[0,1]}(\phi_{0,yy})^2+\dl\int_{\bT\times[0,1]}y({\phi}_{xy})^2. \nn
\end{align}

Using the second one of \eqref{approxdetailed} and the Poincar\'{e} inequality in $x$ direction, we have
\begin{align}
\f{1}{2}\int_{\bT\times[0,1]}u^a_{xx} (\bar{\phi}_y)^2\ls& \dl \int_{\bT\times[0,1]}y\bar{\phi}^2_y\ls \dl \int_{\bT\times[0,1]}y{\phi}^2_{xy}. \label{errorpositive9}
\end{align}
Inserting the above estimates \eqref{errorpositive8}, \eqref{errorpositive9} into \eqref{errorpositive7}, and then into \eqref{errorpositive6}, we can obtain that
\begin{align}
I_2\leq& \dl\int_{\bT\times[0,1]} y \lt[(\phi_{xx})^2+(\phi_{xy})^2\rt]+\e^2\dl\int_{\bT\times[0,1]}y(\phi_{yy})^2\nn\\
       &+\e^2\dl \int_{\bT\times[0,1]}(\phi_{0,yy})^2.\label{errorpositive10}
\end{align}

{\noindent\bf Term $I_3$}: This one has been done in \eqref{errorpositive4} and \eqref{errorpositive4p}.

{\noindent\bf Term $I_4$}: Using integration by parts, we have
\begin{align}
I_4=\int_{\bT\times[0,1]} (\p_xR_v-\p_yR_u)\p_x\phi=&\int_{\bT\times[0,1]} R_u\phi_{xy}-R_v\phi_{xx}\nn\\
                                    =&\int_{\bT\times[0,1]} R_u u_x+R_vv_{x}.\label{errorpositive12}
\end{align}

Now combining all the estimates in \eqref{errorpositive1}, \eqref{errorpositive4f}, \eqref{errorpositive4s}, \eqref{errorpositive10}, and \eqref{errorpositive12}, we can obtain that

\begin{align}
&\int_{\bT\times[0,1]} (u^a-C\dl y)\lt[(\phi_{xx})^2+(\phi_{xy})^2\rt] \nn\\
&\leq C\e^2\dl\int_{\bT\times[0,1]}y(\bar{\phi}_{yy})^2+\e^2\dl \int_{\bT\times[0,1]}(\phi_{0,yy})^2+C\int_{\bT\times[0,1]} R_uu_{x}+R_v v_{x}. \label{errorpositive13}
\end{align}
Noting that
\bes
\phi_{xx}=-v_x,\q \phi_{xy}=u_x,\q \phi_{yy}=u_y,
\ees
and the first one of \eqref{approxdetailed}, the above inequality \eqref{errorpositive13} is actually \eqref{estipositive} by choosing suitably small $\dl_0$. \qed

In order to close the linear estimate, we still need to give the control of the first term on the right hand of \eqref{estipositive}, which will be achieved by the following lemma through basic energy estimates.
\subsection{Energy estimates}\label{subsenergy}
\begin{lemma}\label{lemlinearenergy}
 Let $(u,v)$ be a smooth solution of (\ref{e:error equation}), then there exist $\epsilon_0>0, \dl_0>0$ such that for any $\epsilon\in (0,\epsilon_0), \dl\in(0,\dl_0)$, there holds
\begin{align}\label{estilinearenergy}
&\e^2\int_{\bT\times[0,1]}\lt(u^2_{0,y}+yu^2_y\rt)dxdy\nn\\
&\leq C\int_{\bT\times[0,1]}\lt(yu^2_{x}+yv^2_x\rt)dxdy+C\int_{\bT\times[0,1]} \lt(y^\kappa R^2_u+yR^2_v\rt)+C\lt|\int_{\bT\times[0,1]} R_u u_0\rt|.
\end{align}
Here $\kappa\in[0,1)$.
\end{lemma}

\pf The proof is divided into two parts. The first part deals with  the weighted estimate of the non-zero frequency of $u$, namely, $\bar{u}$, and the other part takes care of the non weighted estimate of zero frequency of $u$, namely, $u_0$.

\subsubsection{Energy estimate I: $\e^2\int_{\bT\times[0,1]}y\bar{u}^2_{y}dxdy$}\label{subsenergy1}
\indent

Now we multiply \eqref{e:error equation3}$_1$ by $y\bar{\phi}$ and integrate the resulted equation on $\bT\times[0,1]$ to obtain that
\begin{align}
&\epsilon^2\int_{\bT\times[0,1]}\Dl^2\phi y\bar{\phi} \nn\\
=&\underbrace{\int_{\bT\times[0,1]}\lt(u^a\Dl\phi_x-\phi_x\Dl u^a\rt) y\bar{\phi}}_{J_1}+\underbrace{\int_{\bT\times[0,1]}\lt(v^a\p_y\Dl\phi-\p_y\phi\Dl v^a\rt) y\bar{\phi}}_{J_2}\label{errorenergy1}\\
 &+\underbrace{\int_{\bT\times[0,1]} (\p_xR_v-\p_yR_u)y\bar{\phi}}_{J_3}.\nn
\end{align}
For the left-hand of \eqref{errorenergy1}, by decomposing $\phi=\phi_0+\bar{\phi}$, we can have
\begin{align}
\e^2\int_{\bT\times[0,1]}\Dl^2\phi y\bar{\phi}=&\e^2\int_{\bT\times[0,1]}(\phi_{xxxx}+2\phi_{xxyy}+\phi_{yyyy}) y\bar{\phi}\nn\\
=&\e^2\int_{\bT\times[0,1]}(\bar{\phi}_{xxxx}+2\bar{\phi}_{xxyy}+\bar{\phi}_{yyyy}) y\bar{\phi}.\nn
\end{align}
Here we have used the following facts
\begin{align*}
\phi_{0,xxxx}=\phi_{0,xxyy}=0,\q \int_{\bT\times[0,1]}\phi_{0,yyyy}y\bar{\phi}=\int^1_0\phi_{0,yyyy}ydy\int_{\bT}\bar{\phi}dx=0.
\end{align*}
Then integration by parts indicates that
\begin{align}
&\e^2\int_{\bT\times[0,1]}\Dl^2\phi y\bar{\phi}\nn\\
&=\e^2\int_{\bT\times[0,1]}y\bar{\phi}^2_{xx}-2\e^2\int_{\bT\times[0,1]}\bar{\phi}_{xxy}(\bar{\phi}+y\bar{\phi}_y)-\e^2\int_{\bT\times[0,1]}\bar{\phi}_{yyy}(\bar{\phi}+y\bar{\phi}_y)\nn\\
&=\e^2\int_{\bT\times[0,1]}y\bar{\phi}^2_{xx}+2\e^2\int_{\bT\times[0,1]}\bar{\phi}^2_{xy}y+\e^2\int_{\bT\times[0,1]}\bar{\phi}_{yy}\bar{\phi}_y-\e^2\int_{\bT\times[0,1]}\bar{\phi}_{yyy}y\bar{\phi}_y\nn\\
&=\e^2\int_{\bT\times[0,1]}y\bar{\phi}^2_{xx}+2\e^2\int_{\bT\times[0,1]}\bar{\phi}^2_{xy}y+\e^2\int_{\bT\times[0,1]}\bar{\phi}^2_{yy}y\nn\\
&=\e^2\int_{\bT\times[0,1]}y{\phi}^2_{xx}+2\e^2\int_{\bT\times[0,1]}{\phi}^2_{xy}y+\e^2\int_{\bT\times[0,1]}\bar{\phi}^2_{yy}y.\label{errorenergy2}
\end{align}
Here we have used the following facts
\bes
\bar{\phi}|_{y=0,1}=\bar{\phi}_x|_{y=0,1}=\bar{\phi}_y|_{y=0,1}=0.
\ees

Now we give estimates of the right hand terms of \eqref{errorenergy1}.

{\noindent\bf Term $J_1$}:  By integration by parts, we have
\begin{align}
J_1=&\int_{\bT\times[0,1]}\phi_{x}\lt[(u^a y\bar{\phi})_{xx}-u^a_{xx}y\bar{\phi}\rt]+\int_{\bT\times[0,1]}\phi_{x}\lt[(u^a y\bar{\phi})_{yy}-u^a_{yy}y\bar{\phi}\rt]\nn\\
   =&\int_{\bT\times[0,1]}\bar{\phi}_{x}\lt[u^a y\bar{\phi}_{xx}+2yu^a_x\bar{\phi}_x\rt]-\int_{\bT\times[0,1]}u^a \lt[ (y\bar{\phi})_{yy}\bar{\phi}_x+2(y\bar{\phi})_y\bar{\phi}_{xy}\rt].\nn
\end{align}
Using the basic fact that in \eqref{approxdetailed} and the Cauchy inequality, we can obtain that

\begin{align}
J_1
   \ls &\int_{\bT\times[0,1]} |\bar{\phi}_{x}| \lt|(y\bar{\phi}_{xx},y\bar{\phi}_{x})\rt|
   +\int_{\bT\times[0,1]}\lt|(y\bar{\phi}_y,\bar{\phi})\rt|\lt|(\bar{\phi}_x,y\bar{\phi}_{xy})\rt| - \int_{\bT\times[0,1]} u^a y\bar{\phi}_{yy}\bar{\phi}_x \nn\\
   \ls& \int_{\bT\times[0,1]} \bar{\phi}^2_{x}+\int_{\bT\times[0,1]}\lt[y \lt(\bar{\phi}^2_{xx}+\bar{\phi}^2_{x}+\bar{\phi}^2_y\rt)+\bar{\phi}^2\rt]\nn\\
   &+\int_{\bT\times[0,1]} \lt[y\bar{\phi}^2_{y}+\bar{\phi}^2\rt]+\int_{\bT\times[0,1]} y\bar{\phi}^2_{xy} - \int_{\bT\times[0,1]} u^a y\bar{\phi}_{yy}\bar{\phi}_x . \nn
\end{align}

Using integration by parts for $\bar{\phi}_{yy}$, the estimate in \eqref{approxdetailed}  and  Cauchy inequality, we have
\begin{align*}
\lt|\int_{\bT\times[0,1]} u^a y\bar{\phi}_{yy}\bar{\phi}_x\rt|=&\lt|\int_{\bT\times[0,1]} (u^a_y y +u^a ) \bar{\phi}_{y}\bar{\phi}_x-\f{1}{2}u^a_x y \bar{\phi}^2_{y}\rt|\\
\ls& \lt|\int_{\bT\times[0,1]}u^a_y y\bar{\phi}_{y}\bar{\phi}_x\rt|+ \int_{\bT\times[0,1]} y \lt(\bar{\phi}^2_{y}+\bar{\phi}^2_x\rt).
\end{align*}
Using estimates in Lemma \ref{lemdetail}, Cauchy inequality, Hardy inequality \eqref{hardy2} and Poincar\'{e} inequality in $x$ direction, we have
\begin{align}
&\lt|\int_{\bT\times[0,1]}u^a_y y\bar{\phi}_{y}\bar{\phi}_x\rt|\nn\\
= & \lt|\int_{\bT\times[0,1]}\lt\{u^a_e+(1-\chi(y))^2u^a_p+\chi^2(y)\hat{u}^a_p+h\rt\}_y y\bar{\phi}_{y}\bar{\phi}_x\rt|\nn\\
\ls& \int_{\bT\times[0,1]}\lt\{A+\dl \rt\}y|\bar{\phi}_{y}\bar{\phi}_x|+(1-y)[(1-\chi(y))^2u^a_p]_y y |\bar{\phi}_{y}|\f{|\bar{\phi}_x|}{1-y}\nn\\
\ls& \dl \int_{\bT\times[0,1]}\f{(1-\chi(y))^2\bar{\phi}^2_{x}}{(1-y)^2}+\int_{\bT\times[0,1]}y(\bar{\phi}^2_{y}+\bar{\phi}^2_x)\nn\\
   \ls & \int_{\bT\times[0,1]}y(\bar{\phi}_{x}^2+\bar{\phi}_{y}^2)+\dl\int_{\bT\times[0,1]}y(\bar{\phi}_{xy})^2.\nn
\end{align}

Using Hardy inequality in \eqref{hardy1} and Poincar\'{e} inequality in $x$ direction, we can obtain that
\begin{align}
J_1
   \ls& \int_{\bT\times[0,1]} y\lt(\bar{\phi}^2_{xy}+\bar{\phi}^2_{xx}\rt)\ls \int_{\bT\times[0,1]} y\lt({\phi}^2_{xy}+{\phi}^2_{xx}\rt). \label{errorenergyj1}
\end{align}

{\noindent\bf Term $J_2$}:  By integration by parts, we have

\begin{align}
J_2=&\int_{\bT\times[0,1]}\phi_{y}\lt[(v^a y\bar{\phi})_{xx}-v^a_{xx}y\bar{\phi}\rt]+\int_{\bT\times[0,1]}\phi_{y}\lt[(v^a y\bar{\phi})_{yy}-v^a_{yy}y\bar{\phi}\rt]\nn\\
   =&\int_{\bT\times[0,1]}{\phi}_{y}\lt[v^a y\bar{\phi}_{xx}+2yv^a_x\bar{\phi}_x\rt]+\int_{\bT\times[0,1]}(y\bar{\phi})_y\lt[v^a_y{\phi}_y-v^a{\phi}_{yy}\rt].\nn
\end{align}
Using the basic fact in \eqref{approxdetailed1}, divergence-free condition, and the Cauchy inequality, we can obtain that
\begin{align}
J_2
   \ls &\e\dl\int_{\bT\times[0,1]} |{\phi}_{y}| \lt|(y\bar{\phi}_{xx},y\bar{\phi}_{x})\rt|
   +\e\dl\int_{\bT\times[0,1]}\lt|(y\bar{\phi}_y,\bar{\phi})\rt|\lt| y{\phi}_{yy}\rt|\nn\\
   &-\int_{\bT\times[0,1]} u^a_x\phi_y (y\bar{\phi})_{y}\nn\\
   \ls& \e^2\dl\int_{\bT\times[0,1]}y\lt({\phi}^2_{y}+{\phi}^2_{yy}\rt)+\dl\int_{\bT\times[0,1]}\lt[y \lt(\bar{\phi}^2_{xx}+\bar{\phi}^2_{x}+\bar{\phi}^2_y\rt)+\bar{\phi}^2\rt]\nn\\
   &-\int_{\bT\times[0,1]} u^a_x\phi_y (y\bar{\phi})_{y}.\nn
\end{align}
Using the Hardy inequality and Poincar\'{e} inequality, we can obtain that
\begin{align}
J_2
   \ls&\int_{\bT\times[0,1]} y\lt({\phi}^2_{xy}+{\phi}^2_{xx}\rt)+\e^2\dl\int_{\bT\times[0,1]}y\lt({\phi}^2_{y}+{\phi}^2_{yy}\rt)-\int_{\bT\times[0,1]} u^a_x\phi_y (y\bar{\phi})_{y}.\nn
\end{align}
Using the second one of \eqref{approxdetailed} and Poincar\'{e} inequality, we can obtain that
\begin{align}
&-\int_{\bT\times[0,1]} u^a_x\phi_y (y\bar{\phi})_{y}\nn\\
&=-\int_{\bT\times[0,1]} u^a_x(\phi_{0,y}+\bar{\phi}_y) (y\bar{\phi}_y+\bar{\phi})\nn\\
&\leq \int_{\bT\times[0,1]}u^a_x \phi_{0,y} (y\bar{\phi}_y+\bar{\phi})+\dl\int_{\bT\times[0,1]}y\bar{\phi}^2_y
+\dl\int_{\bT\times[0,1]}y|\bar{\phi}_y \bar{\phi}|\nn\\
&\ls\int_{\bT\times[0,1]}u^a_x \phi_{0,y} (y\bar{\phi}_y+\bar{\phi})+\dl\int_{\bT\times[0,1]}y({\phi}^2_{xy}+{\phi}^2_{xx}).\nn
\end{align}
The same as before by using Lemma \ref{lemdetail}, we can obtain that
\bes
|u^a_x \phi_{0,y}|\ls \e\dl \lt(\f{|\phi_{0,y}|}{1-y}+{|\phi_{0,y}|}\rt).
\ees
Then using Cauchy inequality, Hardy inequality, the above estimate and Poincar\'{e} inequality, we can obtain that
\begin{align}
\int_{\bT\times[0,1]}u^a_x \phi_{0,y} (y\bar{\phi}_y+\bar{\phi})\ls \e^2\dl \int_{\bT\times[0,1]}\phi^2_{0,yy} +\dl \int_{\bT\times[0,1]} y({\phi}^2_{xy}+{\phi}^2_{xx}).\nn
\end{align}

Then at last by using the Hardy inequality \eqref{hardy1}, we obtain
\begin{align}
J_2
   \ls&\int_{\bT\times[0,1]} y\lt({\phi}^2_{xy}+{\phi}^2_{xx}\rt)+\e^2\dl\int_{\bT\times[0,1]}\lt(y{\phi}^2_{yy}+\phi^2_{0,yy}\rt). \label{errorenergyj2}
\end{align}

{\noindent\bf Term $J_3$}: Using integration by parts, the Cauchy inequality and Hardy inequality \eqref{hardy1}, we have
\begin{align}
J_3=&\int_{\bT\times[0,1]} (\p_xR_v-\p_yR_u)y\bar{\phi}\nn\\
=&\int_{\bT\times[0,1]} R_u (y\bar{\phi})_{y}-R_vy\bar{\phi}_{x}\nn\\
                \ls &\int_{\bT\times[0,1]} |R_u| |(y\bar{\phi}_{y},\bar{\phi})|+|R_v|y|\bar{\phi}_{x}|\nn\\
                \ls & \int_{\bT\times[0,1]} \lt(y\bar{\phi}^2_{x}+y\bar{\phi}^2_{y}+y^{-\kappa}\bar{\phi}^2\rt)+\int_{\bT\times[0,1]} \lt(yR^2_v+y^{\kappa}R^2_u\rt)\q \kappa\in [0,1)\nn\\
                \ls & \int_{\bT\times[0,1]} \lt(y{\phi}^2_{xx}+y^{\kappa+2}{\phi}^2_{xy}\rt)+\int_{\bT\times[0,1]} \lt(yR^2_v+y^{\kappa}R^2_u\rt)\nn\\
                \ls & \int_{\bT\times[0,1]} \lt(y{\phi}^2_{xx}+y{\phi}^2_{xy}\rt)+\int_{\bT\times[0,1]} \lt(yR^2_v+y^{\kappa}R^2_u\rt). \label{errorenergyj4}
\end{align}

Combining \eqref{errorenergy1}, \eqref{errorenergy2}, \eqref{errorenergyj1}, \eqref{errorenergyj2}, and \eqref{errorenergyj4}, we can at last obtain that
\begin{align}
&\e^2\int_{\bT\times[0,1]}y{\phi}^2_{xx}+2\e^2\int_{\bT\times[0,1]}{\phi}^2_{xy}y+\e^2\int_{\bT\times[0,1]}\bar{\phi}^2_{yy}y\nn\\
&\leq C\int_{\bT\times[0,1]} y\lt({\phi}^2_{xy}+{\phi}^2_{xx}\rt)+\e^2\dl\int_{\bT\times[0,1]}\lt(y{\phi}^2_{yy}+\phi^2_{0,yy}\rt)\nn\\
  &\quad+\int_{\bT\times[0,1]} \lt(yR^2_v+y^{\kappa}R^2_u\rt),\nn
\end{align}
which indicates that
\begin{align}
&\e^2\int_{\bT\times[0,1]}\bar{\phi}^2_{yy}y\nn\\
&\leq C\int_{\bT\times[0,1]} y\lt({\phi}^2_{xy}+{\phi}^2_{xx}\rt)+C\e^2\dl\int_{\bT\times[0,1]}\lt(y{\phi}^2_{yy}+u^2_{0,y}\rt)+C_\kappa\int_{\bT\times[0,1]} \lt(yR^2_v+y^{\kappa}R^2_u\rt). \label{errorenergy3}
\end{align}

Next we give the zero frequency estimate of $u$.
\subsubsection{Energy estimate II: $\int_{\bT\times[0,1]}u^2_{0,y}dxdy$}
\indent

Now we go back to the first equation of \eqref{e:error equation}.
\be\label{errorenergys1}
-\epsilon^2\Dl u+u^au_x+uu^a_x+v^au_y+vu^a_y+\p_xp=R_u.
\ee
Now multiplying \eqref{errorenergys1} by $u_0$ and then integrating on $\bT\times[0,1]$, we can obtain that
\begin{align}
&-\int_{\bT\times[0,1]}\epsilon^2\Dl u u_0\nn\\
&=\underbrace{-\int_{\bT\times[0,1]} \lt(u^au_x+uu^a_x+p_x\rt)u_0}_{K_1}\underbrace{-\int_{\bT\times[0,1]} \lt(v^au_y+vu^a_y\rt)u_0}_{K_2}+\underbrace{\int_{\bT\times[0,1]} R_u u_0}_{K_3}.\nn
\end{align}
Using the boundary condition $u|_{y=0,1}=u_0|_{y=0,1}=0$, we can have
\begin{align}
&-\int_{\bT\times[0,1]}\epsilon^2\Dl u u_0\nn\\
&=-\e^2 \int_{\bT\times[0,1]}(u_{xx}+u_{yy}) u_0=-\e^2 \int_{\bT\times[0,1]}u_{0,yy} u_0=\e^2 \int_{\bT\times[0,1]}u^2_{0,y}.\label{errorenergys2}
\end{align}
{\noindent\bf Term $K_1$}: By using integration by parts in $x$ and $u_0$ being independent of $x$, it is easy to see that
\be\label{errorenergysk1}
K_1=0.
\ee
{\noindent\bf Term $K_2$}: By using the facts that
\bes
\int^{2\pi}_0 v^adx=\int^{2\pi}_0 vdx=0,
\ees
and $u_0$ being independent on $x$, we have
\begin{align}
K_2=&\int_{\bT\times[0,1]} v^a (\bar{u}+u_0 )_yu_0+v (u^a-u_e(y))_y u_0+v (u_e(y))_y u_0\nn\\
   =&\int_{\bT\times[0,1]} v^a \bar{u}_y u_0+v (u^a-Ay)_y u_0+v A u_0,\q u_e=Ay\nn\\
   =&\int_{\bT\times[0,1]} v^a \bar{u}_y u_0+v (u^a-Ay)_y u_0\nn\\
   =&-\int_{\bT\times[0,1]} \bar{u}\lt( v^a_yu_0+v^a u_{0,y}\rt)+\int_{\bT\times[0,1]}v (u^a-Ay)_y u_0.\label{errorenergys3}
\end{align}
Then using the first one of \eqref{approxdetailed1}, the Cauchy inequality, Hardy inequality and Poincar\'{e} inequality, we have

\begin{align}
-\int_{\bT\times[0,1]} \bar{u}v^a u_{0,y}\leq &\e\dl\int_{\bT\times[0,1]}y|\bar{u}||u_{0,y}|\nn\\
\ls & \dl\int_{\bT\times[0,1]}y\bar{u}^2+\e^2\dl\int_{\bT\times[0,1]}u^2_{0,y}\nn\\
\ls & \dl\int_{\bT\times[0,1]}y{u}^2_x+\e^2\dl\int_{\bT\times[0,1]}u^2_{0,y}. \label{errorenergys4p}
\end{align}
Using estimates in Lemma \ref{lemdetail}, we see that
\bes
|v^a_y u_0|\ls \e\dl |u_0| \lt(1+\f{1-\chi(y)}{1-y}\rt).
\ees
Then
\begin{align}
-\int_{\bT\times[0,1]} \bar{u} v^a_yu_0\leq &\e\dl\int_{\bT\times[0,1]}|\bar{u}|\lt(1+\f{1-\chi(y)}{1-y}\rt)|u_0|\nn\\
\ls & \dl\int_{\bT\times[0,1]}y\bar{u}^2+\e^2\dl\int_{\bT\times[0,1]} \f{u^2_0}{y^2}+\f{u^2_0}{(1-y)^2}\nn\\
\ls & \dl\int_{\bT\times[0,1]}y{u}^2_x+\e^2\dl\int_{\bT\times[0,1]}u^2_{0,y}.\label{errorenergys4}
\end{align}

Besides, by using \eqref{eulerapproxd}, \eqref{prandtlapprox1d} and \eqref{prandtlapprox2d}, we have
\begin{align}
|(u^a-Ay)_y|vu_0\leq& \e\dl |v u_0| +\e\dl |v|\f{u_0}{y}+\e\dl \f{|u_0|}{1-y}\f{|v|}{1-y}.\nn
\end{align}

So, by the Cauchy inequality and Hardy inequality
\begin{align}
\int_{\bT\times[0,1]}v (u^a-Ay)_y u_0\ls &\e\dl\int_{\bT\times[0,1]}|v|\lt( |y^{-1}u_0|+|(1-y)^{-1}u_0|\rt)\nn\\
\ls &\e^2\dl\int_{\bT\times[0,1]}u^2_{0,y} +\dl\int_{\bT\times[0,1]} v^2\label{errorenergys5}\\
\ls &\e^2\dl\int_{\bT\times[0,1]}u^2_{0,y} +\dl\int_{\bT\times[0,1]} yv^2_y.\nn
\end{align}
Inserting \eqref{errorenergys4p}, \eqref{errorenergys4} and \eqref{errorenergys5} into \eqref{errorenergys3}, we obtain that
\be\label{errorenergysk2}
K_2\ls \e^2\dl\int_{\bT\times[0,1]}u^2_{0,y} +\dl\int_{\bT\times[0,1]} y u^2_x.
\ee
Combining \eqref{errorenergys2}, \eqref{errorenergysk1} and \eqref{errorenergysk2}, we at last obtain
\bes
\e^2 \int_{\bT\times[0,1]}u^2_{0,y}\leq C\e^2\dl\int_{\bT\times[0,1]}u^2_{0,y} +C\dl\int_{\bT\times[0,1]} y u^2_x+\lt|\int_{\bT\times[0,1]} R_u u_0\rt|.
\ees
This and \eqref{errorenergy3} imply \eqref{estilinearenergy} by choosing small $\dl_0$. \qed
\\

{\bf\noindent Proof of Proposition \ref{proplinearstability}}: At last, by combining the positivite estimate \eqref{estipositive} in Lemma \ref{lempositive} and \eqref{estilinearenergy} in Lemma \ref{lemlinearenergy}, we arrive at \eqref{linearstability} in Proposition \ref{proplinearstability}. \qed

\section{Partial $\dot{H}^2$ and existence of the error equation}\label{sec3}

In order to prove the main result, we still need give the $L^\i$ estimate of $(u,v)$ and existence of solution to system \eqref{e:error equation}. To this end, we consider the following ``$H^2$ estimate" for a Stokes system.
\subsection{Partial ``$\dot{H}^2$ estimate" for Stokes system}
\indent

We consider the Stokes equations
\begin{align}\label{e:stoke error equation}
\left\{
\begin{array}{lll}
-\epsilon^2\Dl u+\p_xp=f_u,\\[5pt]
-\epsilon^2\Dl v+\p_y p=g_v,\\[5pt]
u_x+v_y=0,  \\[5pt]
 u(x,y)=u(x+2\pi,y),\ v(x,y)=v(x+2\pi,y),\\[5pt]
u(x,1)=v(x,1)=0, \q u(x,0)=v(x,0)=0.
\end{array}
\right.
\end{align}
\begin{lemma}
Let $(u,v)$ be a smooth solution of (\ref{e:stoke error equation}), then we have
\begin{align}\label{s:Stokes estimate}
&\int_{0}^{1}\int_0^{2\pi} [u^2_{xx}+v^2_{xx}]dx dy+\int_{0}^{1}\int_0^{2\pi}\big(u_{xy}^2+v_{xy}^2\big)dxdy\nonumber\\
&\leq C\e^2\Big|\int_{0}^1\int_0^{2\pi}\big[f_u u_{xx}+g_v v_{xx}\big]dxdy\Big|.
\end{align}
\end{lemma}
\begin{proof}
Multiplying $u_{xx}$ to (\ref{e:stoke error equation})$_1$ and $v_{xx}$ to(\ref{e:stoke error equation})$_2$, integrating the resulted equation in $\Omega$ and summing them together, we arrive at
\begin{align*}
&\int_{0}^1\int_0^{2\pi}\Big[-\epsilon^2\Dl u+p_x\Big]u_{xx}dxdy+\int_{0}^1\int_0^{2\pi}\Big[-\epsilon^2\Dl v+p_y\Big]v_{xx}dxdy\\
&=\int_{0}^1\int_0^{2\pi}\big[f_u u_{xx}+g_v v_{xx}\big]dxdy.
\end{align*}
Integrating by parts and using the divergence-free condition, we deduce
\beno
\int_{0}^1\int_0^{2\pi}\big[p_x u_{xx}+p_y v_{xx}\big]dxdy=0.
\eeno
Then, also by integrating by parts, we obtain
\beno
&&\int_{0}^1\int_0^{2\pi}-\epsilon^2 \Dl u u_{xx}dxdy=-\int_{0}^1\int_0^{2\pi}\epsilon^2 (u^2_{xx}+u^2_{xy})dxdy,\\
&&\int_{0}^1\int_0^{2\pi}-\epsilon^2 \Dl v v_{xx}dxdy=-\int_{0}^1\int_0^{2\pi}\epsilon^2 (v^2_{xx}+v^2_{xy})dxdy.
\eeno
Thus we obtain (\ref{s:Stokes estimate}) and complete the proof of this lemma.
\end{proof}

To obtain the $L^\infty$ estimate, we need the following anisotropic Sobolev embedding.
\begin{lemma}
Let $(u,v)$ be smooth periodic function in $x$ and satisfies $(u,v)|_{y=0,1}=(0,0)$, then
\begin{align}\label{s:Sobolev embedding}
\|(u,v)\|_\infty \leq C(\|(u_x,v_x)\|_2+\|(u_y,v_y)\|_2+\|(u_{xy},v_{xy})\|_2).
\end{align}
\end{lemma}
\begin{proof}
Actually this result has been  proven in \cite[Lemma 4.2]{FeiGLT:2021ARXIV}. For completeness, we repeat the proof again here. In fact, we make Fourier series expansion of $u(x,y)$ to obtain
\beno
u(x,y)=\sum_{k\in Z}u_k(y)e^{ikx}, \ \forall y\in [0,1],
\eeno
where $u_k(y):=\f{1}{2\pi}\int^{2\pi}_0 u(x,y)dx$.  Thus
\beno
|u(x,y)|\leq\sum_{k\in Z}|u_k(y)|, \ \forall (x,y)\in [0,2\pi]\times[0,1].
\eeno
Notice that $u(x,0)=0, \ \forall\ x \in [0,2\pi]$, which gives $u_k(x,0)=0,\ \forall \ x\in [0,2\pi], k\in Z$, hence by one dimensional Sobolev embedding,
\beno
\|u_k\|_{\infty}\leq \sqrt{2}\|u_k\|^{\frac12}_2\|u'_k\|^{\frac12}_2.
\eeno
Thus,
\beno
\sum_{k\neq 0}\|u_k\|_\infty\leq \sqrt{2}\sum_{k\neq 0}\|u_k\|^{\frac12}_2\|u'_k\|^{\frac12}_2\leq C\Big(\sum_{k\neq 0}|k|^2\|u_k\|^2_2\Big)^{\frac14}\Big(\sum_{k\neq 0}|k|^2\|u'_k\|^2_2\Big)^{\frac14}.
\eeno

Moreover, it's easy to get
\beno
u_x(x,y)=\sum_{k\in Z}iku_k(y)e^{ikx}, \quad u_{xy}(x,y)=\sum_{k\in Z}iku'_k(y)e^{ikx}, \ \forall y\in [0,1].
\eeno
Thus,
\beno
\|u_x\|_2^2=\sum_{k\in Z}|k|^2\|u_k\|_2^2, \quad \|u_{xy}\|_2^2=\sum_{k\in Z}|k|^2\|u'_k\|_2^2.
\eeno
Hence, we obtain
\beno
\sum_{k\neq 0}\|u_k\|_\infty\leq C\|u_x\|^{\frac12}_2\|u_{xy}\|_2^{\frac12}.
\eeno
Moreover, due to $u_0(y)=\frac{1}{2\pi}\int_0^{2\pi}u(x,y)dx$ and $u(x, 0)=0,\ \forall x\in [0,2\pi]$, it's easy to get
\beno
|u_0(y)|\leq C\|u_y\|_2, \quad \forall y\in [0,1].
\eeno

Finally, we obtain
\beno
\|u\|_{\infty}\leq C(\|u_x\|_2+\|u_y\|_2+\|u_{xy}\|_2).
\eeno

Similarly we can obtain the same result for $v$ and hence  we complete the proof of this lemma.
\end{proof}

\subsection{Existence for the error equations}
\indent

We apply the contraction mapping theorem to prove the existence for the error equations (\ref{e:error equation}).

\begin{proposition}\label{existence and error estimate of error equation}
There exist $\epsilon_0>0,\dl_0>0$ such that for any $\epsilon\in (0,\epsilon_0), \dl\in (0,\dl_0)$, the error equations (\ref{e:error equation}) have a unique solution $(u,v)$ which satisfies
$$\|(u,v)\|_\infty\leq C\epsilon.$$
\end{proposition}
\begin{proof}
For a smooth function $(u,v)$ which satisfies
\begin{align}\label{iterative conditon}
\left\{
\begin{array}{ll}
 u_x+v_y=0,\\[5pt]
u(x,y)=u(x+2\pi,y),\ v(x,y)=v(x+2\pi,y),\\[5pt]
u(x,1)=u(x,0)=v(x,1)=v(x,0)=0,
\end{array}
\right.
\end{align}
we consider the following linear problem:
\begin{align}
\left\{
\begin{array}{lll}
-\epsilon^2 \Dl \t{u} +\t{p}_x+S_{\t{u}}=R_u,\\[7pt]
-\epsilon^2 \Dl \t{v}+\t{p}_y+S_{\t{v}}=R_v,\\[7pt]
\partial_x \t{u}+\partial_y\t{v}=0,  \\[7pt]
 \bar{u}(x,y)=\t{u}(x+2\pi,y),\  \t{v}(x,y)=\t{v}(x+2\pi,y), \\[7pt]
\bar{u}(x,1)=\t{v}(x,1)=\t{u}(x,0)=\t{v}(x,0)=0,
\end{array}
\right.\nonumber
\end{align}
where
\begin{align*}
S_{\t{u}}:=&u^a\t{u}_x+v^a \t{u}_y+\t{u}u^a_x+\t{v}u^a_y,\\[5pt]
S_{\t{v}}:=&u^a\t{v}_x+v^a\t{v}_y+\t{u}v^a_x+\t{v}v^a_y,\\[5pt]
R_u:=& R_u^a-uu_x-vu_y,\quad R_v:=R_v^a-uv_x-vv_y.
\end{align*}

By the linear stability estimate (\ref{linearstability}), we deduce that there exist $\epsilon_0>0, \dl_0>0$ such that for any $\epsilon\in (0,\epsilon_0), \dl\in(0,\dl_0)$, there holds
\begin{align}
&\|(\s{y}\t{u}_x,\s{y}\t{v}_x,\e \s{y}\t{u}_{\neq0,y},\e\t{u}_{0,y})\|_2^2\nn\\
&\leq  C\lt|\int_{\bT\times[0,1]} R_u \t{u}_{x}+R_v \t{v}_{x}\rt|+C_{\kappa}\int_{\bT\times[0,1]} \lt(y^{\kappa} R^2_u+yR^2_v\rt)+\lt|\int_{\bT\times[0,1]} R_u \t{u}_0\rt|.\label{erroresti1}
\end{align}
Now we give some estimates for the forced terms on the right hand of \eqref{erroresti1}.

Using integration by parts, the Cauchy inequality, and Hardy inequality \eqref{hardy1}, we have
\begin{align*}
&\lt|\int_{\bT\times[0,1]} R_u \t{u}_{x}+R_v \t{v}_{x}\rt|\\
&=\lt|\int_{\bT\times[0,1]}\p_xR_u \t{u}+\p_xR_v \bar{v}\rt|\nn\\
&\leq \nu\e^2\|(\t{u},\t{v})\|^2_{L^2}+C_\nu\e^{-2}\|(\p_xR_u,\p_xR_v)\|^2_{L^2}\nn\\
&\leq  \nu\e^2\|y(\t{u}_y,\t{v}_y)\|^2_{L^2}+C_\nu\e^{-2}\|(\p_xR_u,\p_xR_v)\|^2_{L^2},
\end{align*}
for small $\nu$. And also from the Cauchy inequality and Hardy inequality \eqref{hardy2}, we have
\begin{align}
&\lt|\int_{\bT\times[0,1]} R_u \t{u}_0\rt|\leq \nu\e^{2}\|\t{u}_{0,y}\|^2_{L^2}+C_\nu\e^{-2}\|yR_u\|^2_{L^2}.\nn
\end{align}
Inserting the above two inequalities into \eqref{erroresti1}, we can obtain that
\begin{align*}
\|(\s{y}\t{u}_x,\s{y}\t{v}_x,\e \s{y}\t{u}_{y},\e\t{u}_{0,y})\|_2^2
\leq C\e^{-2}\|(R_u,R_v, \p_xR_u,\p_xR_v )\|^2_{L^2}.
\end{align*}

Direct computation gives that
\beno
\|(R_u, R_v)\|_2^2\leq C\|(R^a_u, R^a_v)\|_2^2+C\|(u,v)\|^2_\infty \|(u_y,u_x, v_x)\|_2^2
\eeno
and
\begin{align}
\|(\p_xR_u, \p_xR_v)\|_2^2\leq &C\|(\p_xR^a_u, \p_xR^a_v)\|_2^2+C\|(u,v)\|^2_\infty \|(u_{xx},v_{xx},v_{xy},u_{xy})\|_2^2.\nn
\end{align}
Using the periodicity and Poincar\'{e} inequality in $x$ direction, we have
\bes
\|(u_y,u_x, v_x)\|_2^2\ls \|(u_{0,y},u_{xy},u_{xx}, v_{xx})\|_2^2.
\ees

Thus, there holds
\begin{align}
&\|(\s{y}\t{u}_x,\s{y}\t{v}_x,\e \s{y}\t{u}_{y},\e\t{u}_{0,y})\|_2^2\nn\\
\leq& C\e^{-2}\|(R^a_u,R^a_v,\p_xR^a_u,\p_xR^a_v)\|^2_{L^2} \nn\\
   &+C\e^{-2}\|(u,v)\|^2_\infty \lt(\|u_{0,y}\|_2^2+\|(u_{xx},v_{xx},v_{xy},u_{xy})\|_2^2\rt).\label{erroresti3}
\end{align}

Next, we consider the following Stokes problem:
\begin{align}
\left\{
\begin{array}{lll}
-\epsilon^2\Dl \t{u}+\t{p}_x=R_u- S_{\t{u}},\\[7pt]
-\epsilon^2\Dl \t{v}+\t{p}_y=R_v-S_{\t{v}},\\[7pt]
\partial_x \t{u}+\partial_y\t{v}=0,  \\[7pt]
 \t{u}(x,y)=\t{u}(x+2\pi,y),\  \t{v}(x,y)=\t{v}(x+2\pi,y),\\[7pt]
\t{u}(x,1)=\t{v}(x,1)=\t{u}(x,0)=\t{v}(x,0)=0.
\end{array}
\right.\nonumber
\end{align}
By the Stokes estimate (\ref{s:Stokes estimate}), we deduce
\begin{align}
&\|(\t{u}_{xx},\t{v}_{xx},\t{u}_{xy},\t{v}_{xy})\|_2^2\nn\\[5pt]
 &\leq C\e^{-2}\Big|\int_{0}^1\int_0^{2\pi}\Big[(R_u-S_{\t{u}}) \t{u}_{xx}+(R_v-S_{\t{v}}) \t{v}_{xx}\Big]dxdy\Big|.\label{erroresti3p}
\end{align}

We compute the right hand term by term. \\

{\bf 1)The term $\int_{0}^{1}\int_0^{2\pi}R_u \t{u}_{xx}dxdy$ and $\int_{0}^{1}\int_0^{2\pi}R_v \t{v}_{xx}dxdy$:} First, by Cauchy inequality, there holds
\begin{align*}
\Big|\int_{0}^{1}\int_0^{2\pi}R_u \t{u}_{xx}dxdy \Big|&=\Big|\int_{0}^{1}\int_0^{2\pi}[R^a_u \t{u}_{xx}-uu_x \t{u}_{xx}-vu_y\t{u}_{xx}]dxdy\Big|\\[5pt]
&\leq \nu\e^{2} \|\t{u}_{xx}\|^2_2+C_\nu \e^{-2}\| R^a_u\|^2_2+C_\nu\e^{-2}\|u\|^2_{L^\i}\|u_{x}\|^2_{L^2}\nn\\
&\quad+\Big|\int_{0}^{1}\int_0^{2\pi}vu_y\t{u}_{xx} dxdy\Big|.
\end{align*}

Using the divergence-free condition, integrating by part and Cauchy inequality, we deduce that
\begin{align*}
&\Big|\int_{0}^{1}\int_0^{2\pi}vu_y\t{u}_{xx}dxdy\Big|\\
&=\Big|\int_{0}^{1}\int_0^{2\pi}(\t{u}_{xy}u_x v+\t{u}_{xy}uv_x+\t{u}_{xx}uu_x)dx dy\Big|\\
&\leq\nu\e^2\|(\t{u}_{xx},\t{u}_{xy})\|_2^2+C_\nu\e^{-2}\|(u,v)\|^2_{L^\i}\|(u_{x},v_x)\|^2_{L^2}.
\end{align*}
Thus, we obtain
\begin{align}
&\Big|\int_{0}^{1}\int_0^{2\pi}R_u \t{u}_{xx}dxdy\Big|\nn\\
&\leq 3\nu\e^2\|(\t{u}_{xx},\t{u}_{xy})\|_2^2+C_\nu\e^{-2}\|(u,v)\|^2_{L^\i}\|(u_{x},v_x)\|^2_{L^2}+C_\nu \e^{-2}\| R^a_u\|^2_2. \label{erroresti4}
\end{align}
Similarly,
\begin{align}
&\Big|\int_{0}^{1}\int_0^{2\pi}R_v \t{v}_{xx}dxdy\Big|\nn\\
&\leq 3\nu\e^2\|(\t{v}_{xx},\t{u}_{xy})\|_2^2+C_\nu\e^{-2}\|(u,v)\|^2_{L^\i}\|(u_{x},v_x)\|^2_{L^2}+C_\nu \e^{-2}\| R^a_v\|^2_2.\label{erroresti5}
\end{align}

Combining \eqref{erroresti4} and \eqref{erroresti5} and using Poincar\'{e} inequality, we have

\begin{align}
&\Big|\int_{0}^{1}\int_0^{2\pi}R_u \t{u}_{xx}dxdy\Big|+\Big|\int_{0}^{1}\int_0^{2\pi}R_v \t{v}_{xx}dxdy\Big|\nn\\
\ls &\nu\e^2\|(\t{u}_{xx},\t{v}_{xx},\t{u}_{xy},\t{v}_{xy})\|_2^2+C_\nu\e^{-2}\lt(\|(u,v)\|^2_{L^\i}\|(u_{xx},v_{xx})\|^2_{L^2}+\|(R^a_u, R^a_v)\|^2_2\rt).\label{erroresti8}
\end{align}

{\bf 2) The term $\int_{0}^{1}\int_0^{2\pi}S_{\t{u}}\t{u}_{xx}dxdy$ and $\int_{0}^{1}\int_0^{2\pi}S_{\t{v}}\t{v}_{xx}dxdy $:}\\

First, there holds
 \beno
\Big| \int_{0}^{1}\int_0^{2\pi}S_{\t{u}}\t{u}_{xx}dxdy\Big|=\Big| \int_{0}^{1}\int_0^{2\pi}[u^a\t{u}_x+v^a\t{u}_y+u^a_x\t{u}+u^a_y\t{v}]\t{u}_{xx}dxdy\Big|.
 \eeno
Integrating by parts, using \eqref{approxdetailed} and Poincar\'{e} inequality, we deduce that
\begin{align*}
\Big|\int_{0}^{1}\int_0^{2\pi}u^a\t{u}_x\t{u}_{xx}dxdy\Big|=&\frac12\Big|\int_{0}^{1}\int_0^{2\pi}u^a_x\t{u}^2_x dxdy \Big|\leq C\dl \|\s{y}\t{u}_x\|_2^2.
\end{align*}
Using integration by parts, \eqref{approxdetailed1} and H\"{o}lder inequality, we deduce that
\begin{align*}
\Big|\int_{0}^{1}\int_0^{2\pi}v^a\t{u}_y\t{u}_{xx}dxdy\Big|=&\Big|\int_{0}^{1}\int_0^{2\pi}[v^a_x \t{u}_{y}+v^a\t{u}_{xy}]\t{u}_x dxdy\Big|\\[5pt]
\leq& C\epsilon\dl \|\s{y}\t{u}_x\|_2\|(\t{u}_{0,y},\t{u}_{xy})\|_2.
\end{align*}
Also by using integration by parts, \eqref{approxdetailed}, H\"{o}lder inequality and Hardy inequality \eqref{hardy1}, we deduce that
\begin{align*}
\Big|\int_{0}^{1}\int_0^{2\pi}u^a_x\t{u}\t{u}_{xx}dxdy\Big|=&\Big|\int_{0}^{1}\int_0^{2\pi}[u^a_x \t{u}^2_x+u^a_{xx}\t{u}_{x}\t{u}] dxdy\Big|\\[5pt]
\leq &C\e\dl\|\s{y}\t{u}_x\|^2_2+C\e\dl\|\s{y}\t{u}_x\|_2\|y\t{u}_y\|_2.
\end{align*}
Almost the same as the above three, we have
\begin{align*}
\Big|\int_{0}^{1}\int_0^{2\pi}u^a_y\t{v}\t{u}_{xx}dxdy\Big|=&\Big|\int_{0}^{1}\int_0^{2\pi}[(u^a_{xy} \t{v} \t{u}_x+u^a_y\t{v}_{x}\t{u}_x] dxdy\Big|\\[5pt]
=&\Big|\int_{0}^{1}\int_0^{2\pi}[(-v^a_{yy} \t{v} \t{u}_x+u^a_y y\t{v}_{x}y^{-1}\t{u}_x] dxdy\Big|\\
\leq& C\dl \|\t{v}_y\|\|y\t{u}_x\|_2+C\|y\t{v}_x\|_2\|\t{u}_{xy}\|_2,\\[5pt]
\leq&C\dl \|\t{v}_{xy}\|\|\s{y}\t{u}_x\|_2+C\|\s{y}\t{v}_x\|_2\|\t{u}_{xy}\|_2.
\end{align*}
Thus, there exist $\epsilon_0>0, \dl_0>0$ such that for any $\epsilon\in (0,\epsilon_0), \dl\in(0,\dl_0)$, there holds
\begin{align}
&\Big| \int_{0}^{1}\int_0^{2\pi}S_{\t{u}}\t{u}_{xx}dxdy\Big|\nn\\
&\leq\nu \e^2\|(\t{u}_{xx},\t{v}_{xx},\t{u}_{xy},\t{v}_{xy})\|_2^2+C_\nu\epsilon^{-2}\dl\|(\s{y}\t{u}_x,\s{y}\t{v}_x,\e \s{y}\t{u}_{y},\e\t{u}_{0,y})\|_2^2.\label{erroresti6}
\end{align}
Similarly, we can obtain
\begin{align}
&\Big| \int_{0}^{1}\int_0^{2\pi}S_{\t{v}}\t{v}_{xx}dxdy\Big|\nn\\
&\leq\nu \e^2\|(\t{u}_{xx},\t{v}_{xx},\t{u}_{xy},\t{v}_{xy})\|_2^2+C_\nu\epsilon^{-2}\dl\|(\s{y}\t{u}_x,\s{y}\t{v}_x,\e \s{y}\t{u}_{y},\e\t{u}_{0,y})\|_2^2.\label{erroresti7}
\end{align}

Combining the above estimates in \eqref{erroresti3p} \eqref{erroresti8}, \eqref{erroresti6} and \eqref{erroresti7}, we deduce that

\begin{align}\label{2orderestimate}
&\|(\t{u}_{xx},\t{v}_{xx},\t{u}_{xy},\t{v}_{xy})\|_2^2\nn\\
\leq& C\epsilon^{-4}\|( R_u^a, R_v^a)\|_2^2+C\epsilon^{-4}\dl\|(\s{y}\t{u}_x,\s{y}\t{v}_x,\e \s{y}\t{u}_{y},\e\t{u}_{0,y})\|_2^2\nn\\
   &+C\epsilon^{-4}\|(u,v)\|_\infty^2\|({u}_{xx},{v}_{xx},{u}_{xy},{v}_{xy})\|_2^2.
\end{align}

Set the total energy to be
\begin{align}
\|(u,v)\|_E^2:=&\|(\s{y}{u}_x,\s{y}{v}_x,\e \s{y}{u}_{y},\e{u}_{0,y})\|_2^2+\e^{4}\|(u_{xx},v_{xx},v_{xy},u_{xy})\|_2^2.\nn
\end{align}

Combining \eqref{erroresti3} and \eqref{2orderestimate}, we can obtain that
\begin{align}
\|(\t{u},\t{v})\|_E^2\ls C\e^{-2}\|(R^a_u,R^a_v,\p_xR^a_u,\p_xR^a_v)\|^2_{L^2}+C\e^{-6}\|(u,v)\|^2_\infty \|(u,v)\|_E^2. \label{energyfinal}
\end{align}
Then by using the Sobolev embedding \eqref{s:Sobolev embedding} and Poincar\'{e} inequality, we have
\begin{align}\label{linfinity}
&\|({u},{v})\|_\infty^2\ls \|u_{0,y}\|^2+\|(u_{xx},u_{xy},v_{xy}, v_{xy})\|^2_{L^2} \ls \e^{-4}\|({u},{v})\|^2_E.
\end{align}

With the help of  \eqref{energyfinal} and \eqref{linfinity} , we arrive that
\begin{align}
\|(\t{u},\t{v})\|_E^2\leq& C\e^{-2}\|(R^a_u,R^a_v,\p_xR^a_u,\p_xR^a_v)\|^2_{L^2}+\e^{-10}\|({u},{v})\|_E^4. \label{energyfinal1}
\end{align}
Let $E=\{(u,v)\in C^\infty: (u,v) \ satisfes \ (\ref{iterative conditon})\ and \ \|(u,v)\|_E<+\infty\}$.
Thus, due to
\beno
\|(R^a_u,R^a_v,\p_xR^a_u,\p_xR^a_v)\|_{L^2}\leq C\varepsilon^9,
\eeno
there exist $\epsilon_0>0, \dl_0>0$ such that for any $\epsilon\in (0,\epsilon_0), \dl\in(0,\dl_0)$, the operator
\beno
(u,v)\mapsto (\t{u},\t{v})
\eeno
maps the ball $B:=\{(u,v): \|(u,v)\|^2_E\leq \epsilon^{11}\}$  in $B$ into itself.

Moreover, for every two pairs
$(u_1, v_1)$ and $(u_2, v_2)$ in the ball, we have
\begin{align}\label{e:contraction estimate}
\|(\t{u}_1-\t{u}_2, \t{v}_1-\t{v}_2)\|^2_E\leq C\epsilon^{-10}(\|(u_1,v_1)\|_E^2+\|(u_2,v_2)\|_E^2)\|(u_1-u_2, v_1-v_2)\|_E^2.
\end{align}
In fact, set
\beno
\t{U}:=\t{u}_1-\t{u}_2,\quad \t{V}:=\t{v}_1-\t{v}_2,\quad \t{P}:=\t{p}_1-\t{p}_2,
\eeno
then we have
\begin{align}
\left\{
\begin{array}{lll}
-\epsilon^2\Dl \t{U}+\t{P}_x+S_{\t{U}}=R_U,\\[7pt]
-\epsilon^2\Dl \t{V}+\t{P}_y+S_{\t{V}}=R_V,\\[7pt]
\partial_x \t{U}+\partial_y\t{V}=0,  \\[7pt]
 \t{U}(x,y)=\t{U}(x+2\pi,y),\  \t{V}(x,y)=\t{V}(x+2\pi,y),\\[7pt]
\t{U}(x,1)=\t{V}(x,1)=\t{U}(x,0)=\t{V}(x,0)=0,
\end{array}
\right.\nonumber
\end{align}
where
\begin{align*}
R_U \triangleq R_{u_1}-R_{u_2}=&u_2u_{2,x}+v_2u_{2,y}-u_1u_{1,x}-v_1u_{1,y}\\[5pt]
=&(u_2-u_1)\partial_x u_2+u_1\partial_x(u_2-u_1)+(v_2-v_1)\partial_y u_2+v_1\partial_y(u_2-u_1),\\[5pt]
R_V \triangleq R_{v_1}-R_{v_2}=&u_2v_{2,x}+v_2v_{2,y}-u_1v_{1,x}-v_1v_{1,y}\\[5pt]
=&(u_2-u_1)\partial_x v_2+u_1\partial_x(v_2-v_1)+(v_2-v_1)\partial_y v_2+v_1\partial_y(v_2-v_1).
\end{align*}

Thus, following the estimate (\ref{energyfinal1}) line by line, we obtain (\ref{e:contraction estimate}).
Hence, there exist $\epsilon_0>0,\dl_0>0$ such that for any $\epsilon\in (0,\epsilon_0), \dl\in (0,\dl_0)$, the operator
\beno
(u,v)\mapsto (\t{u},\t{v})
\eeno
maps the ball $B:=\{(u,v): \|(u,v)\|^2_E\leq \epsilon^{11}\}$ into itself and is a contraction mapping. This completes the proof of Proposition \ref{existence and error estimate of error equation}.
\end{proof}

Now we can give the proof of Theorem \ref{thmain}.
\begin{proof}
Combining Proposition \ref{existence and error estimate of error equation} and the constructed approximate solution  in (\ref{app equationd}), we easily obtain Theorem \ref{thmain}.
\end{proof}

\section{Construction of approximate solutions}\label{secappro}
\indent
In this section, we construct an approximate solution of the Navier-Stokes equations (\ref{ns}) by the procedure of matched asymptotic expansion which can be explained by the following picture.

{\tiny
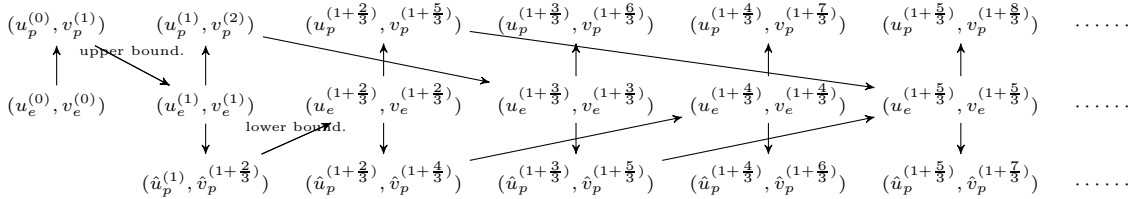
\begin{tikzpicture}
  \matrix (m) [matrix of math nodes, row sep=1.5em, column sep=1em]
    { (u^{(0)}_p,v^{(1)}_p) & (u^{(1)}_p,v^{(2)}_p)  &(u^{(1+\f{2}{3})}_p,v^{(1+\f{5}{3})}_p)  & (u^{(1+\f{3}{3})}_p,v^{(1+\f{6}{3})}_p)  & (u^{(1+\f{4}{3})}_p,v^{(1+\f{7}{3})}_p) & (u^{(1+\f{5}{3})}_p,v^{(1+\f{8}{3})}_p) & \cdots\cdots\\
      (u^{(0)}_e,v^{(0)}_e) & (u^{(1)}_e,v^{(1)}_e)  & (u^{(1+\f{2}{3})}_e,v^{(1+\f{2}{3})}_e)  & (u^{(1+\f{3}{3})}_e,v^{(1+\f{3}{3})}_e)  & (u^{(1+\f{4}{3})}_e,v^{(1+\f{4}{3})}_e) & (u^{(1+\f{5}{3})}_e,v^{(1+\f{5}{3})}_e)  & \cdots\cdots\\
      & (\hat{u}^{(1)}_p,\hat{v}^{(1+\f{2}{3})}_p)  & (\hat{u}^{(1+\f{2}{3})}_p,\hat{v}^{(1+\f{4}{3})}_p)  & (\hat{u}^{(1+\f{3}{3})}_p,\hat{v}^{(1+\f{5}{3})}_p)  & (\hat{u}^{(1+\f{4}{3})}_p,\hat{v}^{(1+\f{6}{3})}_p) & (\hat{u}^{(1+\f{5}{3})}_p,\hat{v}^{(1+\f{7}{3})}_p) & \cdots\cdots\\ };
  { [start chain] \chainin (m-2-1);
    \chainin (m-1-1);
    %{ [start branch] \chainin (m-2-2)
%        [join={node[right,labeled] {\eta_1}}];}
    \chainin (m-2-2) [join={node[above,labeled] {\text{upper bound.}}}];
    { [start branch] \chainin (m-1-2); \chainin (m-2-4);
                        \chainin (m-1-4);}
    \chainin (m-3-2);
    \chainin (m-2-3)[join={node[above,labeled] {\text{lower bound.}}}];
    { [start branch] \chainin (m-1-3);\chainin (m-2-6);\chainin (m-1-6);}
     \chainin (m-3-3);
     \chainin (m-2-5);\chainin (m-3-5);
     }
  { [start chain] \chainin (m-2-4);\chainin (m-3-4);\chainin (m-2-6);\chainin (m-3-6);
 }
 { [start chain] \chainin (m-2-5);\chainin (m-1-5);}
\end{tikzpicture}
}

We give some explanations for this picture.
\begin{itemize}
\item We use $(u_e,v_e)$ to denote the Euler asymptotic expansion away from the boundary, while $(u_p,v_p)$ to denote the asymptotic expansion near the upper boundary and $(\hat{u}_p,\hat{v}_p)$ to denote the asymptotic expansion near the lower boundary.
\item The superscript on the Euler and the boundary layer asymptotic expansion represents the matching order of $\e$.
\item The ``upper" and ``down" arrow mean the matched boundary condition between $u_e$ and $u_p,\ \hat{u}_p$, while the slant arrow means the matched boundary condition matching between $v_e$ and $v_p,\ \hat{v}_p$.
\item The equations satisfied by $(u_e,v_e)$, $(u_p,v_p)$ and $(\hat{u}_p,\hat{v}_p)$ will be obtained by solving \eqref{ns} by matching the order of $\e$. The asymptotic expansion will be solved column by column starting from $(u^{(0)}_e,v^{(0)}_e):=(Ay, 0)$.
\end{itemize}

Define the upper boundary layer variable $\zeta=\f{y-1}{\e}$ and the lower boundary layer variable $\eta:=\f{y}{\e^{\f{2}{3}}}$. Now we are ready to write the formal asymptotic expansion away from the boundary and near the boundary.
\begin{remark}\label{rem4.1}
 Let us now explain the reason why we choose the thickness of the lower boundary layer to be $\e^{2/3}$. One can observe that the leading order of our approximate solution, $u_e(y)=A y$ is vanishing ($u_e(0)=0$) at the lower boundary, which indicates the lower boundary layer expansion doesn't contain a leading zero-order  term compared to the upper boundary layer expansion. Therefore, we can not expect the strong boundary layer scaling of $\e$ as the upper boundary. Instead, we should expect a much weaker boundary layer of thickness $\e^{\beta}$, where $\beta\in (0,1]$, and also form high orders approximation by considering $\e^{k\beta}$-order expansion. The consequence of this milder scaling is that the boundary layer equations at the lower boundary will become now uniform in $\e$. The choosing $\beta$ is as follows.

Assuming that at the lower boundary, the scaling variable is $\eta=\f{y}{\e^\beta}$ for $\beta\in(0,1]$ and the leading expansion term is
\begin{align*}
u=Ay+\e\lt(u^{(1)}_e(x,y)+\hat{u}^{(1)}_p(x,\eta)\rt)\\
v=\e\lt(v^{(1)}_e(x,y)+\e^{\beta}\hat{v}^{(1+\beta)}_p(x,\eta)\rt).
\end{align*}
Inserting this expansion in \eqref{ns}$_1$ and taking the leading $\e$ order, we see that
\bes
\e^{1+\beta}\eta\p_x\hat{u}^{(1)}_p(x,\eta)-\e^{3-2\beta}\p^2_{\eta}\hat{u}^{(1)}_p(x,\eta)=0.
\ees
To match the order, we need $1+\beta=3-2\beta$, which indicates that
\bes
\beta=2/3.
\ees
\end{remark} \qed

Now we write down the Euler asymptotic expansion away from the boundary, the boundary layer asymptotic expansions near the boundary and the boundary conditions satisfied by these expansions.

{\noindent\bf Euler expansions away from the boundary}

Away from the boundary, we make the following formal expansions
\begin{align}
&u^{\e}(x,y)=u_e^{(0)}(x,y)+\e u_e^{(1)}(x,y)+\sum^{k_0}_{k=1}\e^{1+\f{1+k}{3}}u^{(1+\f{1+k}{3})}_e(x,y)+\text{h.o.t.},\nn\\[5pt]
&v^{\e}(x,y)=\e v_e^{(1)}(x,y)+\sum^{k_0}_{k=1}\e^{1+\f{1+k}{3}}v^{(1+\f{1+k}{3})}_e(x,y)+\text{h.o.t.},\label{eulerextension}\\[5pt]
&p^{\e}(x,y)=\e p_e^{(1)}(x,y)+\sum^{k_0}_{k=1}\e^{1+\f{1+k}{3}}p^{(1+\f{1+k}{3})}_e(x,y)+\text{h.o.t.}.\nn
\end{align}
Here and in what follows, ``h.o.t." means higher order terms.

{\noindent\bf Boundary layer expansions near the upper boundary $\{y=1\}$}

Near the upper boundary, we make the following formal expansions
\begin{align}
u^{\e}(x,y)=&u_e^{(0)}(x,y)+u_p^{(0)}(x,\zeta)+\e \lt(u_e^{(1)}(x,y)+u_p^{(1)}(x,\zeta)\rt)\nn\\
             &+\sum^{k_0}_{k=1}\e^{1+\f{1+k}{3}}\lt(u^{(1+\f{1+k}{3})}_e(x,y)+u_p^{(1+\f{1+k}{3})}(x,\zeta)\rt)+\text{h.o.t.},\nn\\[5pt]
v^{\e}(x,y)=&\e \lt(v_e^{(1)}(x,y)+v_p^{(1)}(x,\zeta)\rt)\nn\\
              &+\sum^{k_0}_{k=1}\e^{1+\f{1+k}{3}}\lt(v^{(1+\f{1+k}{3})}_e(x,y)+v^{(1+\f{1+k}{3})}_p(x,\zeta)\rt)+\text{h.o.t.}, \label{boundaryextensionu}\\[5pt]
p^{\e}(x,y)=&\e \lt(p_e^{(1)}(x,y)+p_p^{(1)}(x,\zeta)\rt)\nn\\
            &+\sum^{k_0}_{k=1}\e^{1+\f{1+k}{3}}\lt(p^{(1+\f{1+k}{3})}_e(x,y)+p^{(1+\f{1+k}{3})}_p(x,\zeta)\rt)+\text{h.o.t.},\nn
\end{align}
where we make the following assumptions
\bes
v^{(1+\f{2}{3})}_p=v^{(1+\f{4}{3})}_p=p^{(1+\f{2}{3})}_p=p^{(1+\f{2}{3})}_p\equiv0.
\ees
The boundary condition is matched by
\begin{align}
&u_e^{(0)}\big|_{y=1}+u_p^{(0)}\big|_{\zeta=0}=\al+\dl f(x),\nn\\
&u_e^{(1)}\big|_{y=1}+u_p^{(1)}\big|_{\zeta=0}=0,\q u_e^{(1+\f{1+k}{3})}\big|_{y=1}+u_p^{(1+\f{1+k}{3})}\big|_{\zeta=0}=0,\q 1\leq k\in\bN, \nn\\
&v_e^{(1)}\big|_{y=1}+v_p^{(1)}\big|_{\zeta=0}=0,\q v_e^{(1+\f{1+k}{3})}\big|_{y=1}+v_p^{(1+\f{1+k}{3})}\big|_{\zeta=0}=0,\q 1\leq k\in\bN.\nn
\end{align}

{\noindent\bf Boundary layer expansions near the lower boundary $\{y=0\}$}

Near the lower boundary, we make the following formal expansions
\begin{align}
u^{\e}(x,y)=&u_e(y)+\e\lt(u_e^{(1)}(x,y)+\hat{u}_p^{(1)}(x,\eta)\rt)\nn\\
            &+\sum^{k_0}_{k=1}\e^{1+\f{1+k}{3}}\lt(u^{(1+\f{1+k}{3})}_e(x,y)+\hat{u}_p^{(1+\f{1+k}{3})}(x,\eta)\rt)+\text{h.o.t.},\nn\\[5pt]
v^{\e}(x,y)=&\e v_e^{(1)}(x,y)+\sum^{k_0}_{k=1}\e^{1+\f{1+k}{3}}\lt(v^{(1+\f{1+k}{3})}_e(x,y)+\hat{v}^{(1+\f{1+k}{3})}_p(x,\eta)\rt)+\text{h.o.t.}, \label{boundaryextensionl}\\[5pt]
p^{\e}(x,y)=&\e p_e^{(1)}(x,y)+\sum^{k_0}_{k=1}\e^{1+\f{1+k}{3}}\lt(p^{(1+\f{1+k}{3})}_e(x,y)+\hat{p}^{(1+\f{1+k}{3})}_p(x,\eta)\rt)+\text{h.o.t.},\nn
\end{align}
where we make the following assumptions
\bes
\hat{v}^{(1+\f{3}{3})}_p=p^{(1+\f{3}{3})}_p\equiv0.
\ees
The boundary condition is matched by
\begin{align}
&u_e^{(1)}\big|_{y=0}+\hat{u}_p^{(1)}\big|_{\eta=0}=0,\q u_e^{(1+\f{1+k}{3})}\big|_{y=0}+\hat{u}_p^{(1+\f{1+k}{3})}\big|_{\eta=0}=0,\q 1\leq k\in\bN, \nn\\
&v_e^{(1)}\big|_{y=0}+\hat{v}_p^{(1)}\big|_{\eta=0}=0,\q v_e^{(1+\f{1+k}{3})}\big|_{y=0}+\hat{v}_p^{(1+\f{1+k}{3})}\big|_{\eta=0}=0,\q 1\leq k\in\bN.\nn
\end{align}
Next we deduce the equations satisfied by these expansions.
\subsection{Equations for lower order expansions}\label{lowerexpansion}
\subsubsection{Equations for $(u_p^{(0)},v_p^{(1)},p_p^{(1)})$}
\indent

By substituting the upper boundary layer expansion \eqref{boundaryextensionu} into (\ref{ns}) and collecting the $\epsilon-0$th order terms, we obtain
 the following steady boundary layer equations for $(u_p^{(0)},v_p^{(1)},p_p^{(1)})$
\be\label{boundarylayeru0}
\left\{
\begin {array}{ll}
\big(A+u_p^{(0)}\big)\partial_x u_p^{(0)}+\big( v_p^{(1)}- v_p^{(1)}(x,0)\big)\partial_\zeta u_p^{(0)}-\partial^2_{\zeta}u_p^{(0)}=0,\\[5pt]
\p_\zeta p^{(1)}_p=0,\\[5pt]
\partial_x u_p^{(0)}+\partial_\zeta v_p^{(1)}=0,\\[5pt]
u_p^{(0)}(x,\zeta)=u_p^{(0)}(x+2\pi,\zeta),\ v_p^{(1)}(x,\zeta)=v_p^{(1)}(x+2\pi,\zeta),\\[5pt]
u_p^{(0)}\big|_{\zeta=0}=\alpha+\dl f(x)-A,\\[5pt]
\lim\limits_{\zeta\rightarrow -\infty}u_p^{(0)}=\lim\limits_{\zeta\rightarrow -\infty}v_p^{(1)}=\lim\limits_{\zeta\rightarrow -\infty}p_p^{(1)}=0.
\end{array}
\right.
\ee
From the \eqref{boundarylayeru0}$_{2,6}$ for the requirement of $p^{(1)}_p$, we have $p^{(1)}_p\equiv 0$.
\subsubsection{Equations for $(u_e^{(1)},v_e^{(1)},p_e^{(1)})$}
\indent

Inserting the Euler expansion \eqref{eulerextension} into \eqref{ns} and collecting the $\epsilon-1$th order terms, we deduce that $(u_e^{(1)},v_e^{(1)},p_e^{(1)})$ satisfies the following linearized Euler equations
\be\label{middleeuler1}
\left \{
\begin{array}{ll}
u_e(y) \partial_x u_e^{(1)}+v_e^{(1)}\p_y u_e(y)+\partial_x p_e^{(1)}=0,\\[5pt]
u_e(y) \partial_x v_e^{(1)}+\partial_y p_e^{(1)}=0,\\[5pt]
\partial_xu_e^{(1)}+\partial_y v_e^{(1)}=0,\\
 v_e^{(1)}|_{y=1}=-v_p^{(1)}|_{\zeta=0},\ v_e^{(1)}|_{y=0}=0,\ v_e^{(1)}(x,y)=v_e^{(1)}(x+2\pi,y),
\end{array}
\right.
\ee
where $v_p^{(1)}$ is obtained from \eqref{boundarylayeru0}.

\subsubsection{Equations for $(u_p^{(1)},v_p^{(2)},p_p^{(2)})$}

\indent

By substituting the the upper boundary layer expansion \eqref{boundaryextensionu} into (\ref{ns}) and collecting the $\epsilon-1$th order terms, we obtain
 the following steady boundary layer equations for $(u_p^{(1)},v_p^{(2)},p_p^{(2)})$
 \be\label{boundarylayeru1}
 \lt\{
 \begin{aligned}
 &(A+u^{(0)}_p)\p_xu^{(1)}_p+u^{(1)}_p\p_xu^{(0)}_p+(v^{(1)}_e(x,1)+v^{(1)}_p)\p_\zeta u^{(1)}_p+(v^{(2)}_p-v^{(2)}_p(x,0))\p_\zeta u^{(0)}_p\\
 &-\p^2_\zeta u^{(1)}_p=f_{\text{upper},1}(x,\zeta),\\
 &u^{(0)}_p\p_xv^{(1)}_e(x,1)+(A+u^{(0)}_p)\p_xv^{(1)}_p+(v^{(1)}_e(x,1)+v^{(1)}_p)\p_\zeta v^{(1)}_p+\p_\zeta p^{(2)}_p-\p^2_\zeta v^{(1)}_p=0,\\
 &\p_x u^{(1)}_p+ \p_\zeta v^{(2)}_p=0,
 \end{aligned}
 \rt.
 \ee
 where
\bes
 f_{\text{upper,1}}(x,\zeta)=(\zeta\p_\zeta u^{(0)}_p-u^{(0)}_p)\p_xu^{(1)}_e(x,1)-(u^{(1)}_e(x,1)+A\zeta)\p_xu^{(0)}_p-Av^{(1)}_p.
 \ees

\subsubsection{Equations for $(\hat{u}_p^{(1)},\hat{v}_p^{(1+\f{2}{3})},\hat{p}_p^{(1+\f{2}{3})})$}

\indent

By substituting the the lower boundary layer expansion \eqref{boundaryextensionl} into (\ref{ns}) and collecting the $\epsilon-(1+\f{2}{3})$th order terms, we obtain
 the following steady boundary layer equations for $(\hat{u}_p^{(1)},\hat{v}_p^{(1+\f{2}{3})},\hat{p}_p^{(1+\f{2}{3})})$
 \be\label{boundarylayerl1}
 \lt\{
 \begin{aligned}
 &A\eta \p_x \hat{u}^{(1)}_p+A\lt(\hat{v}^{(1+\f{2}{3})}_p-\hat{v}^{(1+\f{2}{3})}_p(x,0)\rt)-\p^2_{\eta}\hat{u}^{(1)}_p+\p_x\hat{p}^{(1+\f{2}{3})}_{p}=0,\\
 &\p_\eta \hat{p}_p^{(1+\f{2}{3})}=0,\\
 &\p_x \hat{u}^{(1)}_p+\p_\eta\hat{v}^{(1+\f{2}{3})}_p=0,\\
 &\hat{u}^{(1)}_p\big|_{\eta=0}=-u^{(1)}_e(x,0),\q \lim\limits_{\eta\rightarrow +\infty}\hat{v}^{(1+\f{2}{3})}_p=0.
 \end{aligned}
 \rt.
 \ee

\subsubsection{Equations for $(u_e^{(1+\f{2}{3})},v_e^{(1+\f{2}{3})},p_e^{(1+\f{2}{3})})$}
\indent

Inserting the Euler expansion \eqref{eulerextension} into \eqref{ns} and collecting the $\epsilon-1\f{2}{3}$th order terms, we deduce that $(u_e^{(1+\f{2}{3})},v_e^{(1+\f{2}{3})},$ $ p_e^{(1+\f{2}{3})})$ satisfies the following linearized Euler equations
\be\label{middleeuler2}
\left \{
\begin{array}{ll}
u_e(y) \partial_x u_e^{(1+\f{2}{3})}+v_e^{(1+\f{2}{3})}\p_y u_e(y)+\partial_x p_e^{(1+\f{2}{3})}=0,\\[5pt]
u_e(y) \partial_x v_e^{(1+\f{2}{3})}+\partial_y p_e^{(1+\f{2}{3})}=0,\\[5pt]
\partial_xu_e^{(1+\f{2}{3})}+\partial_y v_e^{(1+\f{2}{3})}=0,\\
 v_e^{(1+\f{2}{3})}|_{y=1}=0,\ v_e^{(1+\f{2}{3})}|_{y=0}=-\hat{v}_p^{(1+\f{2}{3})}|_{\eta=0},\ v_e^{(1+\f{2}{3})}(x,y)=v_e^{(1+\f{2}{3})}(x+2\pi,y),
\end{array}
\right.
\ee
where $\hat{v}_p^{(1+\f{2}{3})}$ is obtained from \eqref{boundarylayerl1}.

\subsection{Solvability of the lower order asymptotic expansions}
\indent

Now we give the solvability of the equations satisfied by the lower order asymptotic expansions in Section \ref{lowerexpansion}. The order in which we solve these equations are as follows
\begin{align*}
(u_e(y),0)\rightarrow (u_p^{(0)},v_p^{(1)})\rightarrow (u_e^{(1)},v_e^{(1)})\rightarrow(u_p^{(1)},v_p^{(2)})/(\hat{u}_p^{(1)},\hat{v}_p^{(1+\f{2}{3})})\rightarrow (u_e^{(1+\f{2}{3})},v_e^{(1+\f{2}{3})})\cdots.
\end{align*}

\subsubsection{The boundary layer system \eqref{boundarylayeru0} and its solvability}
\indent

We first derive some necessary condition for the solvability of the system \eqref{boundarylayeru0}, one can also refer to \cite{FGLT}.
\begin{lemma}\label{lembw}
If the system \eqref{boundarylayeru0} has a solution $(u_p^{(0)}, v_p^{(1)})$, which satisfies
\beno
&&A+u_p^{(0)}(x,\zeta)>0, \q \|u_p^{(0)}\|_\infty<+\i,\ \forall\ \zeta\leq 0,
\eeno
then there holds
\begin{align}
A^2=\alpha^2+\frac{\al\dl}{\pi}\int_0^{2\pi}f(x)dx+\frac{\dl^2}{2\pi}\int_0^{2\pi}f^2(x)dx.\nn
\end{align}
\end{lemma}
\begin{proof}
We introduce the von Mises variable
\begin{align*}
\psi=\int_0^\zeta\big(A+{u}_p^{(0)}(x,\bar{\zeta})\big)d\bar{\zeta},\  \mathcal{U}(x,\psi)=A+{u}_p^{(0)}(x,\zeta).
\end{align*}
Then from \eqref{boundarylayeru0}, we deduce that $\mathcal{U}$ satisfies
\begin{eqnarray}
\left \{
\begin {array}{ll}
2\mathcal{U}_x=\lt(\mathcal{U}^2\rt)_{\psi\psi},\\[5pt]
\mathcal{U}(x,\psi)=\mathcal{U}(x+2\pi,\psi),\\[5pt]
\mathcal{U}\big|_{\psi=0}=\alpha+\dl f(x),\ \ \lim_{\psi\rightarrow -\infty}\mathcal{U}=A.\label{modified prandtl equation}
\end{array}
\right.
\end{eqnarray}
Integrating the first equation in (\ref{modified prandtl equation}) from $0$ to $2\pi$ about $x$ leads to
\begin{align*}
\frac{\partial^2}{\partial\psi^2}\int_0^{2\pi}(\mathcal{U})^2(z,\psi)dx=0.
\end{align*}
Notice that $\mathcal{U}$ is bounded at $\psi\rightarrow-\infty$, we deduce that
\begin{align*}
\frac{\partial}{\partial\psi}\int_0^{2\pi}\mathcal{U}^2(x,\psi)dz=0.
\end{align*}
Therefore combining the boundary condition in (\ref{modified prandtl equation}), we deduce that
\begin{align*}
A^2=\frac{1}{2\pi}\int_0^{2\pi}\big(\al+\eta f(x)\big)^2dx&=\al^2+\frac{\al\dl}{\pi}\int_0^{2\pi}f(x)dx+\frac{\al^2}{2\pi}\int_0^{2\pi}f^2(x)dx.
\end{align*}
Thus, we complete the proof of this lemma.
\end{proof}

Next we need to solve the steady boundary layer equations (\ref{boundarylayeru0}), one can refer to \cite[Proposition 2.3 and Corollary 2.4]{FGLT}, which can be proven by using iteration argument and contraction mapping theorem for the modified system (\ref{modified prandtl equation}). Here we omit the details and only present the result.
\begin{proposition}\label{propdcu0}
 There exists $\dl_0>0$ such that for any $\dl\in (0,\dl_0)$ and any $j,k,\ell\in \mathbb{N}\cup \{0\}$, the system (\ref{boundarylayeru0}) has a unique solution $(u^{(0)}_p,v^{(1)}_p)$ which satisfies
\begin{align}\label{decayofblu0}
\begin{aligned}
&\int_{-\infty}^0\int_0^{2\pi}\Big|\partial_x^j\partial_\zeta^k ({u}_p^{(0)},{v}_p^{(1)})\Big|^2\big<\zeta\big>^{2\ell}dx d\zeta\leq C_{j,k,\ell}\dl^2, \\
&\int_0^{2\pi}{v}_p^{(1)}(x,\zeta)d x=0,\ \forall\ \zeta\leq 0.
\end{aligned}
\end{align}
\end{proposition}

\subsubsection{The linearized Euler system for $(u_e^{(1)}, v_e^{(1)}, p_e^{(1)})$ and its solvability}
\begin{proposition}\label{propeulerorder1}
There exists $\dl_0>0$ such that for any $\dl\in(0,\dl_0)$, the linearized Euler equations (\ref{middleeuler1}) have a solution $(u_e^{(1)}, v_e^{(1)}, p_e^{(1)})$ which satisfies
\begin{align}\label{euler1esti}
\|\partial^j_x\partial^k_y(u_e^{(1)},v_e^{(1)})\|_2\leq C_{j,k}\dl, \ \forall\ j,k\geq 0.
\end{align}
\end{proposition}
\begin{proof}
Subtracting $\p_x(\ref{middleeuler1})_2$ from $\p_y(\ref{middleeuler1})_1$ to eliminate the pressure $p_e^{(1)}$, we obtain that
\bes
\Dl v^{(1)}_e=0.
\ees
Then we can obtain $v_e^{(1)}$ by solving the following linear boundary value problem
\be\label{eulerfirstorder1}
\lt\{
\bali
&\Dl v^{(1)}_e=0,\\
&v_e^{(1)}\big|_{y=0}=0,\q v_e^{(1)}\big|_{y=1}=-{v}_p^{(1)}\big|_{\zeta=0}\\
&v_e^{(1)}(x,y)=v_e^{(1)}(x+2\pi,y).
\eali
\rt.
\ee

Noting (\ref{decayofblu0}), it is not hard to deduce that
\begin{align*}
\|\partial^j_x\partial^k_y v_e^{(1)}\|_2\leq C_{j,k}\dl, \ \forall\ k,j\geq 0.
\end{align*}
Also by noting that
\bes
\int^{2\pi}_0 v_e^{(1)} dx\big|_{y=0}=0,\q \int^{2\pi}_0 v_e^{(1)} dx\big|_{y=1}=-\int^{2\pi}_0{v}_p^{(1)}dx\big|_{\zeta=0}=0,
\ees
and integrating \eqref{eulerfirstorder1} with respect to $x$ variable we get that
\bes
\int^{2\pi}_0 v_e^{(1)}(x,y) dx=0 \q \forall\ y\in [0,1].
\ees
After $v^{(1)}_e$ is given, we define
\bes
u^{(1)}_e(x,y)=-\int^x_0 \p_yv^{(1)}_e(\bar{x},y)d\bar{x},
\ees
which satisfies
\begin{eqnarray}
\left \{
\begin {array}{ll}
\partial_x u_e^{(1)}+\partial_yv_e^{(1)}=0,\\[5pt]
u_e^{(1)}(x,y)=u_e^{(1)}(x+2\pi,y).\nonumber
\end{array}
\right.
\end{eqnarray}

After obtaining $(u_e^{(1)}, v_e^{(1)})$, we construct $p_e^{(1)}$ as following
\begin{align*}
p_e^{(1)}(x,y):=\phi(y)- Ayu_e^{(1)}(x,y)-A \int^x_0v_e^{(1)}(\bar{x},y)d\bar{x},
\end{align*}
where $\phi(y)$ is a function satisfying
\begin{align*}
\phi'(y)+Ay\partial_x v_e^{(1)}(0,y)=0.
\end{align*}
Combining the equations of $(u_e^{(1)}, v_e^{(1)})$, it's direct to obtain
\begin{align*}
u_e \partial_x v_e^{(1)}+\partial_yp_e^{(1)}=0.
\end{align*}
Hence, $(u_e^{(1)}, v_e^{(1)},p_e^{(1)})$ solves the equation (\ref{middleeuler1}) and satisfies (\ref{euler1esti}).

\end{proof}

\subsubsection{The linearized Prandtl system for  $(u_p^{(1)},v_p^{(2)})$ and its solvability}
\indent

In this subsection, we consider the solvability of \eqref{boundarylayerl1}.

\begin{proposition}\label{propdcu1}
There exists $\dl_0>0$ such that for any $\dl\in(0,\dl_0)$, the equations (\ref{boundarylayeru1}) have a unique solution $(u_p^{(1)},v_p^{(2)})$ which satisfies
\begin{align*}
\begin{aligned}
&\int_{-\infty}^0\int_0^{2\pi}\big|\partial_x^j\partial_\zeta^k \big({u}_p^{(1)}-A_{1},v_p^{(2)}\big)\big|^2\big<\zeta\big>^{2\ell}dx d\zeta\leq C_{j,k,\ell}\dl^2,\\
&\int_0^{2\pi}v_p^{(2)}(x,\zeta)dx=0, \ \forall\ \zeta\leq 0,
\end{aligned}
\end{align*}
where
$A_{1}:=\lim\limits_{\zeta\rightarrow -\infty}u_p^{(1)}(x,\zeta)$ is a constant which satisfies $|A_{1}|\leq C\dl.$
\end{proposition}

Let $\kappa\in C_c^\infty ((-\infty,0])$ satisfy
\begin{align*}
\kappa(0)=1,\ \int^0_{-\infty}\kappa(\zeta)d\zeta=0.
\end{align*}
For simplicity, we set
\begin{align*}
\bar{u}:&=A+u_p^{(0)}, \ \bar{v}:=v_e^{(1)}(x,1)+v_p^{(1)},\\
u:&=u_p^{(1)}+u_e^{(1)}(x,1)\kappa(\zeta),\\
v:&=v_p^{(2)}-v_p^{(2)}(x,0)-u_e^{(1)}(x,1)\int_0^\zeta \kappa(\bar{\zeta})d\bar{\zeta}.
\end{align*}
Then, the equations (\ref{boundarylayeru1}) reduce to
\begin{eqnarray*}
\left \{
\begin {array}{ll}
\bar{u}\partial_x u+\bar{v}\partial_\zeta u+u\partial_x \bar{u}+v\partial_\zeta\bar{u}-\partial^2_{\zeta}u=\tilde{f},\\[7pt]
\partial_x u+\partial_\zeta v=0,\\[5pt]
u(x,\zeta)=u(x+2\pi,\zeta),\ v(x,\zeta)=v(x+2\pi,\zeta)\\[5pt]
u|_{\zeta=0}=v|_{\zeta=0}=0,\  \lim\limits_{\zeta\rightarrow -\infty}\partial_\zeta u=0,
\end{array}
\right.
\end{eqnarray*}
where $\tilde{f}(x,\zeta)$ is $2\pi$-periodic function and decays fast as $\zeta\rightarrow -\infty$. This system is exactly (2.47) in \cite{FGLT}. Reader can refer there for further detailed proof of Proposition \ref{propdcu1}. Here we omit the details.

Next, we construct the pressure $p_p^{(2)}(x,\zeta)$. Thanks to \eqref{boundarylayeru1}$_2$, we onsider the equation
\begin{align*}
\partial_\zeta p_p^{(2)}(x,\zeta)=g_1(x,\zeta), \quad \lim_{\zeta\rightarrow -\infty}p_p^{(2)}(x,\zeta)=0,
\end{align*}
where $g_1(x,\zeta)$  is defined by
\begin{align*}
g_1(x,\zeta)=-\lt[u^{(0)}_p\p_xv^{(1)}_e(x,1)+(A+u^{(0)}_p)\p_xv^{(1)}_p+(v^{(1)}_e(x,1)+v^{(1)}_p)\p_\zeta v^{(1)}_p-\p^2_\zeta v^{(1)}_p\rt].
\end{align*}
Actually, we take $ p_p^{(2)}(x, \zeta)=\int^x_{-\i}g_1(x,\bar{\zeta})d\bar{\zeta}$.
\subsubsection{The linearized boundary layer system for  $(\hat{u}_p^{(1)},\hat{v}_p^{(1+\f{2}{3})},\hat{p}^{(1+\f{2}{3})}_p)$ and its solvability}
\indent

In this subsection, we consider the solvability of \eqref{boundarylayerl1}.

\begin{proposition}\label{propdcl1}
There exists $\dl_0>0$ and a constant $\hat{A}_1$ such that for any $\dl\in(0,\dl_0)$, the equations (\ref{boundarylayerl1}) have a unique solution $(\hat{u}_p^{(1)},\hat{v}_p^{(1+\f{2}{3})},\hat{p}^{(1+\f{2}{3})}_p)$ which satisfies
\begin{align} \label{decayofbll1}
\begin{aligned}
&\int_{0}^{+\i}\int_0^{2\pi}\big|\partial_x^j\p^k_{\eta}\big(\hat{u}_p^{(1)}-\hat{A}_{1},\hat{v}_p^{(1+\f{2}{3})}\big)\big|^2\langle\eta\rangle^{2\ell} dx d\eta\leq C_{j,k,\ell}\dl^2, \ \forall\ \ell\in \mathbb{N},\\
&\int_0^{2\pi}\hat{v}_p^{(1+\f{2}{3})}(x,\eta)dx=0, \ \forall\ \eta\geq 0,
\end{aligned}
\end{align}
where $\hat{A}_{1}=\lim\limits_{\eta\rightarrow+\i}\hat{u}_p^{(1)}$.
\end{proposition}

\pf

We set
\bes
 \hat{p}_p^{(1+\f{2}{3})}=A\int^x_0\hat{v}^{(1+\f{2}{3})}_p(\bar{x},0)d\bar{x},
\ees
and define
\bes
\hat{u}=\hat{u}^{(1)}_p+\kappa(\eta)u^{(1)}_e(x,0),\q \hat{v}=\hat{v}^{(1+\f{2}{3})}_p-\int^\eta_0\kappa(\bar{\eta})d\bar{\eta}\p_xu^{(1)}_e(x,0).
\ees
Then the system \eqref{boundarylayerl1} can be reformulated into
 \be\label{boundarylayerl1f}
 \lt\{
 \begin{aligned}
 &A\eta \p_x \hat{u}+A\hat{v}-\p^2_{\eta}\hat{u}=f_{\text{b},1},\\
 &\p_x \hat{u}+\p_\eta\hat{v}=0,\\
 &\hat{u}\big|_{\eta=0}=0,\q \p_\eta\hat{u}\big|_{\eta=+\i}=\hat{v}\big|_{\eta=+\i}=0,
 \end{aligned}
 \rt.
 \ee
where
\bes
f_{\text{b},1}:=A\p_xu^{(1)}_e(x,0)\lt(\eta\kappa(\eta)-\int^\eta_0\kappa(\bar{\eta})d\bar{\eta}\rt)-\kappa^{''}(\eta)u^{(1)}_e(x,0),
\ees
has compact support for $\eta$.

The system \eqref{boundarylayerl1f} can be solved by considering the following approximate elliptic system
 \bes
 \lt\{
 \begin{aligned}
 &A\eta \p_x \hat{u}^\nu+A\hat{v}^\nu-\lt(\nu^2\p^2_x+\p^2_{\eta}\rt)\hat{u}^\nu=f_{\text{b},1},\\
 &\p_x \hat{u}^\nu+\p_\eta\hat{v}^\nu=0,\\
 &\hat{u}^\nu\big|_{\eta=0}=0,\q \p_\eta\hat{u}^\nu\big|_{\eta=+\i}=0,
 \end{aligned}
 \rt.
 \ees
and  establishing a uniform a priori estimate of \eqref{decayofbll1} for $(\hat{u}^\nu, \hat{v}^\nu)$. Then letting $\nu\rightarrow0$ indicates the solvability of system \eqref{boundarylayerl1f}. Here we mainly focus on the decay estimates in \eqref{decayofbll1} and omit the standard limit process.

{\noindent\bf Estimates of the average  $\hat{u}_0$}

Denote
\bes
\hat{u}_0(\eta)=\f{1}{2\pi}\int^{2\pi}_0 \hat{u}(x,\eta)dx.
\ees
Integrating \eqref{boundarylayerl1f}$_1$ on $[0,2\pi]$ with respect to $x$ variable to obtain
\be\label{bldecayl1fi}
\lt\{
\bali
&-\p^2_\eta \hat{u}_0(\eta)=f_{0,\text{b},1}(\eta),\\
&\hat{u}_0\big|_{\eta=0}=0,\ \p_\eta\hat{u}_0\big|_{\eta=+\i},
\eali
\rt.
\ee
where
\bes
f_{0,\text{b},1}(\eta)=\f{1}{2\pi}\int^{2\pi}_0f_{\text{b},1}dx.
\ees
Then the solution of \eqref{bldecayl1fi} is obtained by
\bes
\hat{u}_0(\eta):=\int^\eta_0\int^{+\i}_{\bar{\eta}}f_{0,\text{b},1}(\t{\eta})d\t{\eta}d\bar{\eta}.
\ees

Define
\bes
\hat{A}_1:=\hat{u}_0(+\i)=\int^{+\i}_0\int^{+\i}_{\bar{\eta}}f_{0,\text{b},1}(\t{\eta})d\t{\eta}d\bar{\eta}.
\ees
It is easy to see that
\be\label{bldecayl1se}
\int_{0}^{+\i}\int_0^{2\pi}\big|\partial_x^j\partial_\eta^k \big(\hat{u}_0-\hat{A}_{1}\big)\big|^2\big<\eta\big>^{2\ell}dx d\eta\leq C_{j,k,\ell}\dl^2.
\ee

{\noindent\bf Estimates of $\hat{u}-\hat{u}_0$}

Denote $\hat{u}_{\neq 0}=\hat{u}-\hat{u}_0$ and $\hat{f}_{\neq 0}=f_{\text{b},1}-f_{0,\text{b},1}(\eta)$, then we have
 \be\label{boundarylayerl1th}
 \lt\{
 \begin{aligned}
 &A\eta \p_x \hat{u}_{\neq0}+A\hat{v}-\p^2_{\eta}\hat{u}_{\neq0}=f_{\neq0},\\
 &\p_x \hat{u}_{\neq0}+\p_\eta\hat{v}=0,\\
 &\hat{u}_{\neq0}\big|_{\eta=0}=0,\q \p_\eta\hat{u}_{\neq0}\big|_{\eta=+\i}=\hat{v}\big|_{\eta=+\i}=0,\\
 &\int^{2\pi}_0 \hat{u}_{\neq0}dx=\int^{2\pi}_0 f_{\neq0}dx=0.
 \end{aligned}
 \rt.
 \ee

Multiplying \eqref{boundarylayerl1th}$_1$ by $ \eta\hat{u}_{\neq0}$, $\p_x \hat{u}_{\neq0}$ and $ \eta^2 \p_x\hat{u}_{\neq0}$ respectively, then integrating the resulted equations over $\bT\times[0,1]$, and using integration by parts, the Cauchy inequality, Hardy inequality and Poincar\'{e} inequality, we can obtain that
\begin{equation}
\begin{split}
\label{bllone}
\|\s{\eta}		\p_\eta\hat{u}_{\neq0}\|^2_{L^2}=&\int \eta f_{\neq0} \hat{u}_{\neq0}-A\int \eta \hat{v}\hat{u}_{\neq0}\\
 \leq & C_\nu\| \s{\eta}\hat{u}_{\neq0}\|^2+\nu\| \s{\eta}\hat{v}\|^2+\int \eta f_{\neq0} \hat{u}_{\neq0}\\
\leq & C_\nu\| \s{\eta}\hat{u}_{\neq0}\|^2+\nu\| \s{\eta}\hat{v}\|^2+C_\nu\int \eta^2 f^2_{\neq0} +\nu \int \hat{u}^2_{\neq0}\\
= & C_\nu\| \s{\eta}\hat{u}_{\neq0}\|^2+\nu\| \s{\eta}\hat{v}\|^2+C_\nu\int \eta^2 f^2_{\neq0} +\nu \int \chi^2\hat{u}^2_{\neq0}+\nu \int (1-\chi^2)\hat{u}^2_{\neq0}\\
\leq & C_\nu\| \s{\eta}\hat{u}_{\neq0}\|^2+\nu\| \s{\eta}\hat{v}\|^2+C_\nu\int \eta^2 f^2_{\neq0} +\nu \int  \eta^2(\chi'\hat{u}_{\neq0})^2\\
 &+\nu \int  \eta^2(\chi\p_\eta\hat{u}_{\neq0})^2    +\nu \int \eta \hat{u}^2_{\neq0}\\
\leq & 2C_\nu\| \s{\eta}\p_x\hat{u}_{\neq0}\|^2+\nu\| \s{\eta}\hat{v}\|^2+C_\nu\int \eta^2 f^2_{\neq0} +\nu C\|\s{\eta}\p_\eta\hat{u}_{\neq0}\|^2.
\end{split}\end{equation}
Here $\chi=\chi(\eta)$ is a cut-off function, and we obtain the last inequality by using the compact support of $\chi$ and hence we have $\eta^2\leq C\eta$ in a bounded interval.

Also the following estimate stands by multiplying $\p_x \hat{u}_{\neq0}$ to \eqref{boundarylayerl1th}$_1$.
\begin{equation}
\begin{split}
\label{bll1two}
&A\|\s{\eta}\p_x\hat{u}_{\neq0}\|^2_{L^2}+\f{1}{2}A\int^{2\pi}_0|\hat{v}(x,0)|^2dx\\
=&-\int \p_xf_{\neq0} \hat{u}_{\neq0}\\
\leq&\nu\|\hat{u}_{\neq0}\|^2_{L^2}+C_\nu\|\p_xf_{\neq0}\|^2_{L^2}\\
\leq&\nu\| \s{\eta}\p_x\hat{u}_{\neq0}\|^2+ \nu C\|\s{\eta}\p_\eta\hat{u}_{\neq0}\|^2+C_\nu\|\p_xf_{\neq0}\|^2_{L^2}.
\end{split}
\end{equation}
Then using the incompressibility, and similar as \eqref{bllone}, we deduce that
\begin{equation}
\begin{split}\label{bllthree}
&A\|{\eta}^{\f{3}{2}}\p_x\hat{u}_{\neq0}\|^2_{L^2}+2A\| \s{\eta}\hat{v}\|^2\\
&=A\|{\eta}^{\f{3}{2}}\p_x\hat{u}_{\neq0}\|^2_{L^2}+A\int \hat{v}\eta^2\p_x\hat{u}_{\neq0}\\
&=\int f_{\neq0} \eta^3\p_x\hat{u}_{\neq0}.+\int \eta^2\p_x\hat{u}_{\neq0}\p^2_\eta\hat{u}_{\neq0}\\
&=\int f_{\neq0} \eta^3\p_x\hat{u}_{\neq0}.-2\int \eta\p_x\hat{u}_{\neq0}\p_\eta\hat{u}_{\neq0}\\
&\leq C_\nu\int f^2_{\neq0} \eta^3+ \nu \|{\eta}^{\f{3}{2}}\p_x\hat{u}_{\neq0}\|^2_{L^2}.+C_\nu \|\s{\eta}\p_x\hat{u}_{\neq0}\|^2_{L^2}+\nu\|\s{\eta}\p_\eta\hat{u}_{\neq0}\|^2_{L^2}.
\end{split}\end{equation}

Adding \eqref{bllone}, \eqref{bll1two} and \eqref{bllthree} together, and choosing $\nu$ small enough, we can obtain that
\begin{align*}
&\|\s{\eta}\p_x\hat{u}_{\neq0}\|^2_{L^2}+\|{\eta}^{\f{3}{2}}\p_x\hat{u}_{\neq0}\|^2_{L^2}+\| \s{\eta}\hat{v}\|^2_{L^2}+ \|\s\eta \p_\eta\hat{u}_{\neq0}\|^2_{L^2}\\
&\leq C_A\lt(\| \eta f_{\neq0}\|^2_{L^2}+\|\p_xf_{\neq0}\|^2_{L^2}+\| \eta^\f{3}{2} f_{\neq0}\|^2_{L^2}\rt).
\end{align*}

{\noindent\bf Induction on the weight}

Now for $\ell\geq 1$, if we have

\begin{align}
&\|\eta^{\ell-\f{1}{2}}\p_x\hat{u}_{\neq0}\|^2_{L^2}+\|\eta^{\ell+\f{1}{2}}\p_x\hat{u}_{\neq0}\|^2_{L^2}+\|\eta^{\ell-\f{1}{2}}\hat{v}\|^2_{L^2}+ \|\eta^{\ell-\f{1}{2}}\p_\eta\hat{u}_{\neq0}\|^2_{L^2}\leq C_\ell\dl^2. \label{bll1fourteen}
\end{align}
Now we go to prove that \eqref{bll1fourteen} is still valid for $\ell+1$.

Multiplying \eqref{boundarylayerl1th}$_1$ by $ \eta^{2\ell+1}\hat{u}_{\neq0}$ and $\eta^{2\ell+2}\p_x \hat{u}_{\neq0}$ respectively, then integrating the resulted equations over $\bT\times[0,1]$, we can obtain that
\begin{align}
&-\int \eta^{2\ell+1}\hat{u}_{\neq0} \p^2_\eta\hat{u}_{\neq0}=\int f_{\neq0}\eta^{2\ell+1} \hat{u}_{\neq0}-A\int\eta^{2\ell+1}\hat{u}_{\neq0}\hat{v},\nn
\end{align}
and
\begin{align}
&A\|\eta^{\ell+\f{3}{2}}\p_x\hat{u}_{\neq0}\|^2_{L^2}-\int \eta^{2\ell+2}\p_x\hat{u}_{\neq0}\p^2_\eta\hat{u}_{\neq0}+A\int \hat{v}\eta^{2\ell+2}\p_x\hat{u}_{\neq0}=\int f_{\neq0}\eta^{2\ell+2}\p_x\hat{u}_{\neq0}.\nn
\end{align}
The the incompressibility, integration by parts and the H\"{o}lder inequality indicate that
\begin{equation}
\begin{split}\label{bll1fifteen}
&\|\eta^{\ell+\f{1}{2}} \p_\eta\hat{u}_{\neq0}\|^2_{L^2}\\
&\leq (2\ell+1)\ell\|\eta^{\ell-\f{1}{2}}\hat{u}_{\neq0}\|^2_{L^2}+\|\eta^{\ell+\f{1}{2}} f_{\neq0}\|_{L^2}\|\eta^{\ell+\f{1}{2}} \hat{u}_{\neq0}\|_{L^2}+A\|\eta^{\ell+\f{3}{2}} \hat{u}_{\neq0}\|_{L^2}\|\eta^{\ell-\f{1}{2}} \hat{v}\|_{L^2}\\
&\leq(2\ell+1)\ell\|\eta^{\ell-\f{1}{2}}\p_x\hat{u}_{\neq0}\|^2_{L^2}+\|\eta^{\ell+\f{1}{2}} f_{\neq0}\|_{L^2}\|\eta^{\ell+\f{1}{2}} \hat{u}_{\neq0}\|_{L^2}+\nu\|\eta^{\ell+\f{3}{2}} \hat{u}_{\neq0}\|_{L^2}^2+C_\nu\|\eta^{\ell-\f{1}{2}} \hat{v}\|_{L^2}^2,
\end{split}\end{equation}
and
\begin{align}
&A\|{\eta}^{\ell+\f{3}{2}}\p_x\hat{u}_{\neq0}\|^2_{L^2}+(2\ell+2)\int \eta^{2\ell+1}\p_x\hat{u}_{\neq0}\p_\eta\hat{u}_{\neq0}+\f{2\ell+2}{2}A\int \hat{v}^2\eta^{2\ell+1}\nn\\
&\leq\| \eta^{\ell+\f{3}{2}} f_{\neq0}\|_{L^2} \|\eta^{\ell+\f{1}{2}}\p_x\hat{u}_{\neq0}\|_{L^2}.\label{bll1sixteen}
\end{align}

Adding \eqref{bll1fifteen} to \eqref{bll1sixteen} together and then using the Cauchy inequality and the induction assumption \eqref{bll1fourteen}, we can get that
\begin{align*}
&\|{\eta}^{\ell+\f{3}{2}}\p_x\hat{u}_{\neq0}\|^2_{L^2}+\|\eta^{\ell+\f{1}{2}} \hat{v}\|^2_{L^2}+ \|\eta^{\ell+\f{1}{2}} \p_\eta\hat{u}_{\neq0}\|^2_{L^2}\leq C_{\ell+1}\dl^2.
\end{align*}
If we apply $\p^j_x$ ($j\geq 0$) to \eqref{boundarylayerl1th} and we can also obtain that
\begin{align}
&\|\eta^{\ell+\f{3}{2}}\p^j_x\hat{u}_{\neq0}\|^2_{L^2}+\|\eta^{\ell+\f{1}{2}}\p^j_x\hat{v}\|^2_{L^2}+\|\eta^{\ell+\f{1}{2}}\p^j_x\p_\eta\hat{u}_{\neq0}\|^2_{L^2}\leq C_{j,\ell}\dl^2. \label{bll1thirteen}
\end{align}

For higher derivatives' estimate on $\eta$, we use the equation to obtain
\bes
\p^j_x\p^2_{\eta}\hat{u}_{\neq0}=\p^j_x\lt(A\eta \p_x \hat{u}_{\neq0}+A\hat{v}-f_{\neq0}\rt),
\ees
which can be estimated by using \eqref{bll1thirteen}.

For $k\geq 1$, we have
\begin{align}
\p^j_x\p^{k+2}_{\eta}\hat{u}_{\neq0}=&\p^j_x\p^{k}_{\eta}\lt(A\eta \p_x \hat{u}_{\neq0}+A\hat{v}-f_{\neq0}\rt)\nn\\
      =&Ak\p_x\p^{k-1}_\eta\hat{u}_{\neq0}+A\p^{k}_\eta\hat{v}-\p^{k}_\eta f_{\neq0}\nn\\
      =&A(k-1)\p_x\p^{k-1}_\eta\hat{u}_{\neq0}-\p^{k}_\eta f_{\neq0},\nn
\end{align}
which can be estimated inductively.

So, at last, we obtain that
\bes
\int_{0}^{+\i}\int_0^{2\pi}\big|\partial_x^j\partial_\eta^k \big(\hat{u}_{\neq0}, \hat{v}\big)\big|^2\eta^{2\ell+\f{1}{2}}dx d\eta\leq C_{j,k,\ell}\dl^2.
\ees
Combining the above inequality and \eqref{bldecayl1se}, we can arrive at that
\bes
\int_{0}^{+\i}\int_0^{2\pi}\big|\partial_x^j\partial_\eta^k \big(\hat{u}-\hat{A}_1, \hat{v}\big)\big|^2\eta^{2\ell+\f{1}{2}}dx d\eta\leq C_{j,k,\ell}\dl^2.
\ees
By using the Hardy inequality, we immediately deduce that
$$\|\partial_x^j\partial_\eta^k\hat{u}\|^2\leq C  \|\eta^2\partial_x^j\partial_\eta^{k+1}\hat{u}\|^2\leq   C_{j,k,2}\dl^2,$$
and then \eqref{decayofbll1}.

\subsubsection{The linearized Euler system for $(u_e^{(1+\f{2}{3})},\ v_e^{(1+\f{2}{3})},\ p_e^{(1+\f{2}{3})})$ and its solvability}
\indent

Since the system \eqref{middleeuler2} is completely the same as the system \eqref{middleeuler1}, we omit the details to prove the solvability.

\subsubsection{Corrections of the expansions to the boundary layer equations and the Euler equations}
\indent

Since we want the solutions of the boundary layer equations decaying fast at infinity, we define
\bes
\t{u}^{(1)}_p:=u^{(1)}_p-A_1,\q \t{\hat{u}}^{(1)}_p:=\hat{u}^{(1)}_p-\hat{A}_1,
\ees
and use $\t{u}^{(1)}_p$ and $\t{\hat{u}}^{(1)}_p$ to replace the original ${u}^{(1)}_p$ and ${\hat{u}}^{(1)}_p$ respectively. Then we need to correct $u^{(1)}_e$ to $\t{u}^{(1)}_e$ such that it satisfis
\begin{align}
&\t{u}^{(1)}_e\big|_{y=1}=-u^{(1)}_p\big|_{\zeta=0}+A_1={u}^{(1)}_e\big|_{y=1}+A_1\nn\\
&\t{u}^{(1)}_e\big|_{y=0}=-\hat{u}^{(1)}_p\big|_{\eta=0}+\hat{A}_1={u}^{(1)}_e\big|_{y=0}+\hat{A}_1,\nn
\end{align}
and $(\t{u}^{(1)}_e,v^{(1)}_e)$ is still a solution of the system \eqref{middleeuler1}.

Let
\bes
\t{u}^{(1)}_e={u}^{(1)}_e+\phi^{(1)}(y).
\ees
Obviously, $(\t{u}^{(1)}_e, v^{(1)}_e)$ is still a solution of the system \eqref{middleeuler1}. By using the incompressibility and harmonic property of $v^{(1)}_e$, we can see that
\bes
\p_x\Dl {u}^{(1)}_e=-\p_y\Dl {v}^{(1)}_e=0.
\ees
So we choose $\phi^{(1)}$ to satisfy
\bes
\lt\{
\bali
& \Dl \t{u}^{(1)}_e=\Dl {u}^{(1)}_e+\phi^{(1)}_{yy}=0,\\
& \phi^{(1)}(1)=A_1,\q \phi^{(1)}(0)=\hat{A}_1.
\eali
\rt.
\ees

Then we have
\begin{itemize}
\item $(\t{u}^{(1)}_e,v^{(1)}_e)$ is still a solution of system \eqref{middleeuler1};
\item $\t{u}^{(1)}_e\big|_{y=1}+\t{u}^{(1)}_p\big|_{\zeta=0}=0 ,\q \t{u}^{(1)}_e\big|_{y=0}+\t{\hat{u}}^{(1)}_p\big|_{\eta=0}=0$;
\item  $\Dl \t{u}^{(1)}_e=0$.
\end{itemize}

\subsection{Equations for higher order expansions and its solvability}
\indent

After performing the above lower order extensions, we can perform the higher order expansions inductively as follows.

\subsubsection{Equations for $(u_e^{(1+\f{1+k}{3})},v_e^{(1+\f{1+k}{3})},p_e^{(1+\f{1+k}{3})})$}\label{sec431}
\indent

For $k\geq 2$, by collecting the $\epsilon-1\f{1+k}{3}$th order terms of the Euler expansions in \eqref{eulerextension}, we deduce that $(u_e^{(1+\f{1+k}{3})},v_e^{(1+\f{1+k}{3})},$ $ p_e^{(1+\f{1+k}{3})})$ satisfies the following linearized Euler equations
\be\label{middleeulerk}
\left \{
\begin{array}{ll}
u_e(y) \partial_x u_e^{(1+\f{1+k}{3})}+v_e^{(1+\f{1+k}{3})}\p_y u_e(y)+\partial_x p_e^{(1+\f{1+k}{3})}=f^{(1+\f{1+k}{3})}_e,\\[5pt]
u_e(y) \partial_x v_e^{(1+\f{1+k}{3})}+\partial_y p_e^{(1+\f{1+k}{3})}=g^{(1+\f{1+k}{3})}_e,\\[5pt]
\partial_x u_e^{(1+\f{1+k}{3})}+\partial_y v_e^{(1+\f{1+k}{3})}=0,\\
 v_e^{(1+\f{1+k}{3})}|_{y=1}=-{v}_p^{(1+\f{1+k}{3})}|_{\zeta=0},\ v_e^{(1+\f{1+k}{3})}|_{y=0}=-\hat{v}_p^{(1+\f{1+k}{3})}|_{\eta=0},\\
 v_e^{(1+\f{1+k}{3})}(x,y)=v_e^{(1+\f{1+k}{3})}(x+2\pi,y),
\end{array}
\right.
\ee
where the forced term $f^{(1+\f{1+k}{3})}_e$ is defined as
\begin{align}
f^{(1+\f{1+k}{3})}_e:=-\sum^{k}_{i=0} u_e^{(1+\f{i}{3})}\partial_x u_e^{(\f{k+1-i}{3})}-\sum^{k}_{i=0} v_e^{(1+\f{i}{3})}\partial_y u_e^{(\f{k+1-i}{3})}+\Dl u_e^{(1+\f{k+1}{3}-2)},\nn
\end{align}
and $g^{(1+\f{1+k}{3})}_e$ is defined as
\begin{align}
g^{(1+\f{1+k}{3})}_e:=-\sum^{k}_{i=0} u_e^{(1+\f{i}{3})}\partial_x v_e^{(\f{k+1-i}{3})}-\sum^{k}_{i=0} v_e^{(1+\f{i}{3})}\partial_y v_e^{(\f{k+1-i}{3})}+\Dl v_e^{(1+\f{k+1}{3}-2)}.\nn
\end{align}
Here for notation simplification, we have set
\begin{align}
u^{(\f{1}{3})}_e=u^{(\f{2}{3})}_e=u^{(\f{4}{3})}_e\equiv 0,\q
v^{(0)}_e=v^{(\f{1}{3})}_e=v^{(\f{2}{3})}_e=v^{(\f{4}{3})}_e\equiv 0. \nn
\end{align}

{\noindent\bf Here we give an explanation why we choose the Euler shear flow satisfying $u^{''}_e(y)=0$. }

When $k=2$, the system \eqref{middleeulerk} is simplified to
\be\label{middleeuler3}
\left \{
\begin{array}{ll}
u_e(y) \partial_x u_e^{(2)}+v_e^{(2)}\p_y u_e(y)+\partial_x p_e^{(2)}=f^{(2)}_e,\\[5pt]
u_e(y) \partial_x v_e^{(2)}+\partial_y p_e^{(2)}=g^{(2)}_e,\\[5pt]
\partial_x u_e^{(2)}+\partial_y v_e^{(2)}=0,\\
 v_e^{(2)}|_{y=1}=-{v}_p^{(2)}|_{\zeta=0},\ v_e^{(2)}|_{y=0}=-\hat{v}_p^{(2)}|_{\eta=0},\\
 v_e^{(2)}(x,y)=v_e^{(2)}(x+2\pi,y),
\end{array}
\right.
\ee
where
\begin{align}
f^{(2)}_e:&=- u_e^{(1)}\partial_x u_e^{(1)}-v_e^{(1)}\partial_y u_e^{(1)}+\p^2_y u_e(y),\label{bwx}\\
g^{(2)}_e:&=- u_e^{(1)}\partial_x v_e^{(1)}- v_e^{(1)}\partial_y v_e^{(1)}.\nn
\end{align}

If $(u_e(y), 0)$ is the shear flow solution of the Euler equation, by taking $x$ average of \eqref{middleeuler3}, we can see that it satisfies
\bes
\int^{2\pi}_0 f^{(2)}_e(x,y)dx\equiv 0.
\ees
Then from \eqref{bwx}, we see that
\begin{equation*}
\begin{split}
2\pi u^{''}_e(y)=&\int_0^{2\pi}v_e^{(1)}\partial_yu_e^{(1)}dx\\
=&\int_0^{2\pi}\partial_yu_e^{(1)}d(\int_0^xv_e^{(1)}(x',y)dx')\\
=&-\int_0^{2\pi}\partial_{xy}u_e^{(1)}\int_0^xv_e^{(1)}(x',y)dx'\\
=&\int_0^{2\pi}\partial_{yy}v_e^{(1)}\int_0^xv_e^{(1)}(x',y)dx'\\
=&\int_0^{2\pi}\partial_{xx}v_e^{(1)}\int_0^xv_e^{(1)}(x',y)dx'\\
=&\int_0^{2\pi}\partial_{x}v_e^{(1)}v_e^{(1)}dx=0,
\end{split}
\end{equation*}
which gives the explanation of the choice of the leading Euler equation in \eqref{equationofleadingeuler}. \qed

Now we show the solvability of the system \eqref{middleeulerk}.
\begin{proposition}\label{propeulerorderk}
There exists $\dl_0>0$ such that for any $\dl\in(0,\dl_0)$, the linearized Euler equations (\ref{middleeulerk}) has a solution $(u_e^{(1+\f{1+k}{3})}, v_e^{(1+\f{1+k}{3})}, p_e^{(1+\f{1+k}{3})})$ which satisfies
\begin{align}\label{eulerkesti}
\|\partial^j_x\partial^k_y(u_e^{(1+\f{k+1}{3})},v_e^{(1+\f{k+1}{3})})\|_2\leq C_{j,k}\dl, \ \forall\ j,k\geq 0.
\end{align}
\end{proposition}
\indent

The idea of proof is the same as that of the system \eqref{middleeuler1}. First we cancel the pressure from \eqref{middleeulerk}$_1$ and \eqref{middleeulerk}$_2$ to obtain that
\be\label{middleeulerk1}
\lt\{
\bali
&-u_e(y) \Dl v_e^{(1+\f{1+k}{3})} =\p_yf^{(1+\f{1+k}{3})}_e-\p_xg^{(1+\f{1+k}{3})}_e,\\
&v_e^{(1+\f{1+k}{3})}|_{y=1}=-{v}_p^{(1+\f{1+k}{3})}|_{\zeta=0},\ v_e^{(1+\f{1+k}{3})}|_{y=0}=-\hat{v}_p^{(1+\f{1+k}{3})}|_{\eta=0},\\
&v_e^{(1+\f{1+k}{3})}(x,y)=v_e^{(1+\f{1+k}{3})}(x+2\pi,y).
\eali
\rt.
\ee

Noting that $v_e^{(1+\f{1+k}{3})}$ is harmonic, i.e.
\bes
\p_yf^{(1+\f{1+k}{3})}_e-\p_xg^{(1+\f{1+k}{3})}_e=0.
\ees
In fact, by the procedure of the Euler expansions, we find $\Dl u_e^{(1+\f{k+1}{3}-2)}=\Dl v_e^{(1+\f{k+1}{3}-2)}=0$, therefore there hold
\begin{align}
f^{(1+\f{1+k}{3})}_e:=-\sum^{k}_{i=0} u_e^{(1+\f{i}{3})}\partial_x u_e^{(\f{k+1-i}{3})}-\sum^{k}_{i=0} v_e^{(1+\f{i}{3})}\partial_y u_e^{(\f{k+1-i}{3})},\nn\\
g^{(1+\f{1+k}{3})}_e:=-\sum^{k}_{i=0} u_e^{(1+\f{i}{3})}\partial_x v_e^{(\f{k+1-i}{3})}-\sum^{k}_{i=0} v_e^{(1+\f{i}{3})}\partial_y v_e^{(\f{k+1-i}{3})}.\nn
\end{align}
Then a direct computation indicates that
\begin{align}
&\p_yf^{(1+\f{1+k}{3})}_e-\p_xg^{(1+\f{1+k}{3})}_e\nn\\
& =\underbrace{-\sum^{k}_{i=0} \p_yu_e^{(1+\f{i}{3})}\partial_x u_e^{(\f{k+1-i}{3})}}_{I_1}\underbrace{-\sum^{k}_{i=0} u_e^{(1+\f{i}{3})}\p_y\partial_x u_e^{(\f{k+1-i}{3})}}_{I_2}\nn\\
   &\quad\underbrace{-\sum^{k}_{i=0} \p_yv_e^{(1+\f{i}{3})}\partial_y u_e^{(\f{k+1-i}{3})}}_{I_3}\underbrace{-\sum^{k}_{i=0} v_e^{(1+\f{i}{3})}\p^2_y u_e^{(\f{k+1-i}{3})}}_{I_4}\nn\\
   &\quad\underbrace{+\sum^{k}_{i=0} \p_xu_e^{(1+\f{i}{3})}\partial_x v_e^{(\f{k+1-i}{3})}}_{I_5}+\underbrace{\sum^{k}_{i=0} u_e^{(1+\f{i}{3})}\partial^2_x v_e^{(\f{k+1-i}{3})}}_{I_6}\nn\\
   &\quad+\underbrace{\sum^{k}_{i=0} \p_xv_e^{(1+\f{i}{3})}\partial_y v_e^{(\f{k+1-i}{3})}}_{I_7}+\underbrace{\sum^{k}_{i=0} v_e^{(1+\f{i}{3})}\p_x\partial_y v_e^{(\f{k+1-i}{3})}}_{I_8}.\nn
\end{align}
By using the incompressibility and harmonic property of $(u_e^{(1+\f{i}{3})},v_e^{(1+\f{i}{3})})$ and $(u_e^{(\f{k+1-i}{3})},v_e^{(\f{k+1-i}{3})})$ , we can obtain that
\bes
I_1+I_3=I_5+I_7=0,\q I_2+I_6=I_4+I_8=0,
\ees
which indicates that $  \Dl v_e^{(1+\f{1+k}{3})}=0$. Then the system \eqref{middleeulerk1} is solved by solving the following linear boundary value problem
\be\label{middleeulerk2}
\lt\{
\bali
&- \Dl v_e^{(1+\f{1+k}{3})} =0,\\
&v_e^{(1+\f{1+k}{3})}|_{y=1}=-{v}_p^{(1+\f{1+k}{3})}|_{\zeta=0},\ v_e^{(1+\f{1+k}{3})}|_{y=0}=-\hat{v}_p^{(1+\f{1+k}{3})}|_{\eta=0},\\
&v_e^{(1+\f{1+k}{3})}(x,y)=v_e^{(1+\f{1+k}{3})}(x+2\pi,y).
\eali
\rt.
\ee

Since
\begin{align}
&\int^{2\pi}_0 v_e^{(1+\f{1+k}{3})}(x,1)dx=-\int^{2\pi}_0 {v}_p^{(1+\f{1+k}{3})}(x,0)dx=0,\nn\\
&\int^{2\pi}_0 v_e^{(1+\f{1+k}{3})}(x,0)dx=-\int^{2\pi}_0 \hat{v}_p^{(1+\f{1+k}{3})}(x,0)dx=0,\nn
\end{align}
by integrating \eqref{middleeulerk2}$_1$ in $x\in [0,2\pi]$, one gets that for any $y\in[0,1]$,
\bes
\int^{2\pi}_0 v_e^{(1+\f{1+k}{3})}(x,y)dx=0.
\ees
This implies , after integrating \eqref{middleeulerk1}$_1$ in $x\in [0,2\pi]$, the compatibility condition: for any $y\in[0,1]$,
\be\label{middleeulerk3}
\int^{2\pi}_0 f^{(1+\f{1+k}{3})}_e(x,y)dx=0.
\ee

Actually by integration by parts, we have
\begin{align}
 &\sum^{k}_{i=0}\int^{2\pi}_0 u_e^{(1+\f{i}{3})}\partial_x u_e^{(\f{k+1-i}{3})}dx\nn\\
& =-\sum^{k}_{i=0}\int^{2\pi}_0 \p_xu_e^{(1+\f{i}{3})} u_e^{(\f{k+1-i}{3})}dx\nn \text{ after relabeling the index} \label{middleeulerk4}\\
 & =-\sum^{k}_{i=0}\int^{2\pi}_0u_e^{(1+\f{i}{3})}\partial_x u_e^{(\f{k+1-i}{3})}dx,\nn
\end{align}
which implies that
\begin{align}
 &\sum^{k}_{i=0}\int^{2\pi}_0 u_e^{(1+\f{i}{3})}\partial_x u_e^{(\f{k+1-i}{3})}dx=0.
\end{align}

Using the incompressibility and harmonic property, we have
\begin{align}
 &\sum^{k}_{i=0}\int^{2\pi}_0 v_e^{(1+\f{i}{3})}\partial_y u_e^{(\f{k+1-i}{3})}dx\nn\\
& =-\sum^{k}_{i=0}\int^{2\pi}_0 v_e^{(1+\f{i}{3})}\int^x_0\partial^2_y v_e^{(\f{k+1-i}{3})}(\bar{x},y)d\bar{x}dx\nn\\
 & =\sum^{k}_{i=0}\int^{2\pi}_0 v_e^{(1+\f{i}{3})}\int^x_0\partial^2_x v_e^{(\f{k+1-i}{3})}(\bar{x},y)d\bar{x}dx \label{middleeulerk5}\\
 & =\sum^{k}_{i=0}\int^{2\pi}_0 v_e^{(1+\f{i}{3})}\lt(\partial_x v_e^{(\f{k+1-i}{3})}(x,y)-\partial_x v_e^{(\f{k+1-i}{3})}(0,y)\rt)dx\nn \\
 & =0.\nn
\end{align}
In the last but second line of \eqref{middleeulerk5},  integration by parts (the same as \eqref{middleeulerk4}) implies that the first term is zero, while the fact that the second term is zero is due to
\bes
\int^{2\pi}_0 v_e^{(1+\f{i}{3})}dx=0, \text{ for } \forall\ y\in[0,1].
\ees
Combining \eqref{middleeulerk4} and \eqref{middleeulerk5}, we obtain \eqref{middleeulerk3}.

With $v^{(1+\f{1+k}{3})}_e$ from \eqref{middleeulerk2}, we define
\bes
u^{(1+\f{1+k}{3})}_e(x,y)=-\int^x_0 \p_yv^{(1+\f{1+k}{3})}_e(\bar{x},y)d\bar{x},
\ees
which satisfies
\begin{eqnarray}
\left \{
\begin {array}{ll}
\partial_x u_e^{(1+\f{1+k}{3})}+\partial_yv_e^{(1+\f{1+k}{3})}=0,\\[5pt]
u_e^{(1+\f{1+k}{3})}(x,y)=u_e^{(1+\f{1+k}{3})}(x+2\pi,y).\nonumber
\end{array}
\right.
\end{eqnarray}

After obtaining $(u_e^{(1+\f{1+k}{3})}, v_e^{(1+\f{1+k}{3})})$, we construct $p_e^{(1+\f{1+k}{3})}$ as following
\begin{align*}
p_e^{(1+\f{1+k}{3})}(x,y):=\phi(y)- Ayu_e^{(1+\f{1+k}{3})}(x,y)-A \int^x_0v_e^{(1+\f{1+k}{3})}(\bar{x},y)d\bar{x}+\int^x_0f_e^{(1+\f{1+k}{3})}(\bar{x},y)d\bar{x},
\end{align*}
where $\phi(y)$ is a function which satisfies
\begin{align*}
\phi'(y)+Ay\partial_x v_e^{(1+\f{1+k}{3})}(0,y)=g^{(1+\f{1+k}{3})}_e(0,y).
\end{align*}
Combining the equations of $(u_e^{(1+\f{1+k}{3})}, v_e^{(1+\f{1+k}{3})})$, it's direct to obtain
\begin{align*}
u_e(y) \partial_x v_e^{(1+\f{1+k}{3})}+\partial_yp_e^{(1+\f{1+k}{3})}=g^{(1+\f{1+k}{3})}_e(x,y).
\end{align*}
Hence, $(u_e^{(1+\f{1+k}{3})}, v_e^{(1+\f{1+k}{3})},p_e^{(1+\f{1+k}{3})})$ solves the equation (\ref{middleeulerk}) and satisfies (\ref{eulerkesti}).

\subsubsection{Higher order boundary expansions near the upper boundary $\{y=1\}$}

\indent

By substituting the the upper boundary layer expansion \eqref{boundaryextensionu} into (\ref{ns}) and collecting the $\epsilon-1\f{1+k}{3}$th order terms, we obtain
 the following steady boundary equations for $(u_p^{(1+\f{1+k}{3})},v_p^{(2+\f{1+k}{3})},p_p^{(2+\f{1+k}{3})})$
 \be\label{boundarylayeruk}
 \lt\{
 \begin{aligned}
 &(A+u^{(0)}_p)\p_xu^{(1+\f{1+k}{3})}_p+u^{(1+\f{1+k}{3})}_p\p_xu^{(0)}_p+(v^{(1)}_e(x,1)+v^{(1)}_p)\p_\zeta u^{(1+\f{1+k}{3})}_p-\p^2_\zeta u^{(1+\f{1+k}{3})}_p\\
 &+(v^{(2+\f{1+k}{3})}_p-v^{(2+\f{1+k}{3})}_p(x,0))\p_\zeta u^{(0)}_p=f_{\text{upper},k}(x,\zeta),\\
 &\p_\zeta p^{(2+\f{2+k}{3})}_p=g_{\text{upper},k}(x,\zeta),\\
 &\p_x u^{(1+\f{1+k}{3})}_p+ \p_\zeta v^{(2+\f{1+k}{3})}_p=0,
 \end{aligned}
 \rt.
 \ee
 where the force term $f_{\text{upper},k}(x,\zeta)$ is defined as
 \begin{align}
 f_{\text{upper},k}(x,\zeta)=&-A\zeta \p_xu^{(\f{1+k}{3})}_p-\p_xp^{(1+\f{1+k}{3})}_p+\p^2_xu^{1+(\f{k+1}{3})-2}_p,\nn\\
                             &-\sum_{i\geq 1}\lt(\f{\p^i_yv^{(1)}_e(x,1)}{i!}\zeta^i\p_\zeta u^{(1+\f{1+k}{3}-i)}_p+\f{\p^i_yv^{(2+\f{1+k}{3}-i)}_e(x,1)}{i!}\zeta^i\p_\zeta u^{(0)}_p\rt)\nn\\
                             &-\sum_{i\geq 0}\lt(\f{\p_x\p^i_yu^{(1+\f{1+k}{3}-i)}_e(x,1)}{i!}\zeta^iu^{(0)}_p-\f{\p^i_yu^{(1+\f{1+k}{3}-i)}_e(x,1)}{i!}\zeta^i\p_xu^{(0)}_p\rt)\nn\\
                             &-\sum^k_{i=0}\lt(\sum_{j\geq 0}\f{\p^j_yu^{(1+\f{i}{3}-j)}_e(x,1)}{j!}\zeta^j+u^{(1+\f{i}{3})}_p\rt)\p_x u^{(\f{k+1-i}{3})}_p\nn\\
                             &-\sum^k_{i=0}u^{(1+\f{i}{3})}_p\sum_{j\geq 0}\f{\p_x \p^j_yu^{(\f{k+1-i-j}{3})}_e}{j!}\zeta^j\nn\\
                               &-\sum^{k+1}_{i=1}\lt(\sum_{j\geq 0}\f{\p^j_yv^{(1+\f{i}{3}-j)}_e(x,1)}{j!}\zeta^j+v^{(1+\f{i}{3})}_p\rt)\p_\zeta u^{1+(\f{k+1-i}{3})}_p\nn\\
                               &-\sum^{k+1}_{i=0}v^{(1+\f{i}{3})}_p\sum_{j\geq 0}\f{\p^{j+1}_y u^{(\f{k+1-i}{3}-j)}_e(x,1)}{j!}\zeta^j,\nn
 \end{align}
 and $g_{\text{upper},k}(x,\zeta)$ is defined as
 \begin{align}
 g_{\text{upper},k}(x,\zeta)=&\p^2_xv^{1+(\f{k+1}{3})-2}_p\nn\\
                              &-\sum^{k+1}_{i=0}\lt(\sum_{j\geq 0}\f{\p^j_yu^{(\f{k+1-i}{3}-j)}_e(x,1)}{j!}\zeta^j+u^{(\f{k+1-i}{3})}_p\rt)\p_x v^{(1+\f{i}{3})}_p\nn\\
                              &-\sum^k_{i=0}u^{(\f{k+1-i}{3})}_p\p_x v^{(1+\f{i}{3})}_p\nn\\
                               &-\sum^{k+1}_{i=0}\lt(\sum_{j\geq 0}\f{\p^j_yv^{(1+\f{i}{3}-j)}_e(x,1)}{j!}\zeta^j+v^{(1+\f{i}{3})}_p\rt)\p_\zeta v^{1+(\f{k+1-i}{3})}_p\nn\\
                               &-\sum^{k+1}_{i=0}v^{(1+\f{i}{3})}_p\sum_{j\geq 0}\f{\p^{j+1}_y v^{(\f{k+1-i}{3}-j)}_e(x,1)}{j!}\zeta^j.\nn
 \end{align}
Here we understand that
\bes
u^{(-\f{1}{3})}_p=u^{(\f{1}{3})}_p=u^{(\f{2}{3})}_p=u^{(\f{4}{3})}_p\equiv 0,
\ees
and
\bes
v^{(-\f{1}{3})}_p=v^{(\f{1}{3})}_p=v^{(\f{2}{3})}_p=v^{(\f{4}{3})}_p\equiv 0.
\ees

For the system \eqref{boundarylayeruk}, we have the following proposition.

\begin{proposition}\label{propdcuk}
There exists $\dl_0>0$ such that for any $\dl\in(0,\dl_0)$, the equations (\ref{boundarylayeruk}) has a unique solution $(u_p^{(1+\f{1+k}{3})},v_p^{(2+\f{1+k}{3})})$ which satisfies
\begin{align*}
\begin{aligned}
&\int_{-\infty}^0\int_0^{2\pi}\big|\partial_x^j\partial_\zeta^k \big({u}_p^{(1+\f{1+k}{3})}-A_{1+\f{1+k}{3}},v_p^{(2+\f{1+k}{3})}\big)\big|^2\big<\zeta\big>^{2\ell}dx d\zeta\leq C_{j,k,\ell}\dl^2,\\
&\int_0^{2\pi}v_p^{(2+\f{1+k}{3})}(x,\zeta)dx=0, \ \forall\ \zeta\leq 0,
\end{aligned}
\end{align*}
where
$A_{1+\f{1+k}{3}}:=\lim\limits_{\zeta\rightarrow -\infty}u_p^{(1+\f{1+k}{3})}(x,\zeta)$ is a constant which satisfies $|A_{1+\f{1+k}{3}}|\leq C\dl.$
\end{proposition}
\pf
Set
\begin{align*}
\bar{u}:&=A+u_p^{(0)}, \ \bar{v}:=v_e^{(1)}(x,1)+v_p^{(1)},\\
u:&=u_p^{(1)}+u_e^{(1)}(x,1)\kappa(\zeta),\\
v:&=v_p^{(2)}-v_p^{(2)}(x,0)-u_e^{(1)}(x,1)\int_0^\zeta \kappa(\bar{\zeta})d\bar{\zeta}.
\end{align*}
Then, the equations (\ref{boundarylayeruk}) reduce to
\begin{eqnarray}\label{bluk1}
\left \{
\begin {array}{ll}
\bar{u}\partial_x u+\bar{v}\partial_\zeta u+u\partial_x \bar{u}+v\partial_\zeta\bar{u}-\partial^2_{\zeta}u=\tilde{f}_{\text{upper},k},\\[7pt]
\partial_x u+\partial_\zeta v=0,\\[5pt]
u(x,\zeta)=u(x+2\pi,\zeta),\ v(x,\zeta)=v(x+2\pi,\zeta)\\[5pt]
u|_{\zeta=0}=v|_{\zeta=0}=0,\  \lim\limits_{x\rightarrow -\infty}\partial_\zeta u=0,
\end{array}
\right.
\end{eqnarray}
where $\tilde{f}_{\text{upper},k}$ is $2\pi$-periodic function and decays fast as $\zeta\rightarrow -\infty$. This system \eqref{bluk1} is solved  as \eqref{boundarylayeru1} and is exactly (2.47) in \cite{FGLT}. One can refer to there for further detailed proof of Proposition \ref{propdcuk}. Here we omit the details.

Next, we construct the pressure $p_p^{(2+\f{1+k}{3})}(x,\zeta)$. Consider the equation
\begin{align*}
\partial_\zeta p_p^{(2+\f{1+k}{3})}(x,\zeta)=g_{\text{upper},k}(x,\zeta), \quad \lim_{\zeta\rightarrow -\infty}p_p^{(2+\f{1+k}{3})}(x,\zeta)=0,
\end{align*}
where
\begin{align*}
g_{\text{upper},k}(x,\zeta)=-\lt[u^{(0)}_p\p_xv^{(1)}_e(x,1)+(A+u^{(0)}_p)\p_xv^{(1)}_p+(v^{(1)}_e(x,1)+v^{(1)}_p)\p_\zeta v^{(1)}_p-\p^2_\zeta v^{(1)}_p\rt].
\end{align*}
Actually, we just take $ p_p^{(2+\f{1+k}{3})}(x, \zeta)=\int^x_{-\i}g_{\text{upper},k}(x,\bar{\zeta})d\bar{\zeta}$.

\subsubsection{Higher order boundary expansions near the upper boundary $\{y=0\}$}

\indent

By substituting the the lower boundary layer expansion \eqref{boundaryextensionl} into (\ref{ns}) and collecting the $\epsilon-1\f{3+k}{3}$th order terms, we obtain
 the following steady boundary layer equations for $(\hat{u}_p^{(1+\f{1+k}{3})},\hat{v}_p^{(1+\f{3+k}{3})},$ $\hat{p}_p^{(1+\f{3+k}{3})})$

 \be\label{boundarylayerlk}
 \lt\{
 \begin{aligned}
 &A\eta\p_x\hat{u}^{(1+\f{1+k}{3})}_p+A\lt(\hat{v}^{(1+\f{3+k}{3})}_p-{v}^{(1+\f{3+k}{3})}_e(x,0)\rt)+\p_x\hat{p}^{(1+\f{3+k}{3})}_p-\p^2_{\eta}\hat{u}^{(1+\f{1+k}{3})}_p\\
  &=f_{\text{lower},k}(x,\eta),\\
 &\p_\eta p^{(1+\f{3+k}{3})}_p=g_{\text{lower},k}(x,\eta),\\
 &\p_x \hat{u}^{(1+\f{1+k}{3})}_p+ \p_\eta \hat{v}^{(1+\f{3+k}{3})}_p=0,\\
& \lim\limits_{\eta\rightarrow +\infty} \hat{v}^{(1+\f{3+k}{3})}_p=0,
 \end{aligned}
 \rt.
 \ee
  where  $f_{\text{lower},k}(x,\eta)$ is defined as
 \begin{align*}
 f_{\text{lower},k}(x,\eta)=&-\sum^{k+1}_{i=0}\lt(\sum_{j\geq 0}\f{\p^j_y u^{(1+\f{i}{3}-j)}_e(x,0)}{j!}\eta^j+\hat{u}^{(1+\f{i}{3})}_p\rt)\p_x \hat{u}^{(1+\f{k-i}{3})}_p\nn\\
                           &-\sum^{k+1}_{i=0}\hat{u}^{(1+\f{k-i}{3})}_p\sum_{j\geq 0}\f{\p_x\p^j_y u^{(1+\f{k-i}{3}-j)}_e}{j!}\eta^j\nn\\
                               &-\sum^{k+1}_{i=1}\lt(\sum_{j\geq 0}\f{\p^j_yv^{(1+\f{i}{3}-j)}_e(x,0)}{j!}\eta^j+\hat{v}^{(1+\f{i}{3})}_p\rt)\p_\eta \hat{u}^{1+(\f{k+2-i}{3})}_p\nn\\
                                &-\sum^{k+1}_{i=0}\hat{v}^{(1+\f{i}{3})}_p\sum_{j\geq 0}\f{\p^{j+1}_y u^{(1+\f{k-i}{3}-j)}_e}{j!}\eta^j\nn\\
                               &+\p^2_x\hat{u}^{(\f{k}{3})}_p,
                               \end{align*}
 and $g_{\text{lower},k}(x,\eta)$ is defined as
                               \begin{align*}
 g_{\text{lower},k}(x,\eta)=&A\eta\hat{u}^{(1+\f{k-1}{3})}_p+\p^2_x\hat{v}^{(\f{k-1}{3})}_p+\p^2_\eta\hat{v}^{1+(\f{k+1}{3})}_p\nn\\
                             &-\sum^{k+1}_{i=0}\lt(\sum_{j\geq 0}\f{\p^j_yu^{(\f{k+1-i}{3}-j)}_e(x,0)}{j!}\eta^j+\hat{u}^{(\f{k+1-i}{3})}_p\rt)\p_x \hat{v}^{(1+\f{i}{3})}_p\nn\\
                             &-\sum^k_{i=0}\hat{u}^{(\f{k+1-i}{3})}_p\sum_{j\geq 0}\f{\p_x\p^j_y {v}^{(1+\f{i}{3}-j)}_e(x,0)}{j!}\eta^j\nn\\
                               &-\sum^{k+1}_{i=0}\lt(\sum_{j\geq 0}\f{\p^j_yv^{(1+\f{i}{3}-j)}_e(x,0)}{j!}\eta^j+\hat{v}^{(1+\f{i}{3})}_p\rt)\p_\eta \hat{v}^{1+(\f{k-i}{3})}_p\nn\\
                               &-\sum^{k+1}_{i=0}\hat{v}^{(1+\f{i}{3})}_p\sum_{j\geq 0}\f{\p^{j+1}_y v^{(\f{k+1-i}{3}-j)}_e(x,0)}{j!}\eta^j.
 \end{align*}
Here we understand that
\bes
\hat{u}^{(\f{1}{3})}_p=\hat{u}^{(\f{2}{3})}_p=\hat{u}^{(\f{4}{3})}_p\equiv 0,
\ees
and
\bes
\hat{v}^{(\f{3}{3})}_p=\hat{v}^{(\f{5}{3})}_p=\hat{v}^{(\f{6}{3})}_p\equiv 0.
\ees
Also we have used \eqref{middleeulerk} for $k$ is replaced by $k+2$.

For the system \eqref{boundarylayerlk}, we have the following proposition.
\begin{proposition}\label{propdclk}
There exists $\dl_0>0$ and a constant $\hat{A}_{1+\f{1+k}{3}}$ such that for any $\dl\in(0,\dl_0)$, the system (\ref{boundarylayerlk}) has a unique solution $(\hat{u}_p^{(1+\f{1+k}{3})},\hat{v}_p^{(1+\f{3+k}{3})},\hat{p}^{(1+\f{3+k}{3})}_p)$ which satisfies
\begin{align*}
\begin{aligned}
&\int_{0}^{+\i}\int_0^{2\pi}\big|\partial_x^j\p^k_{\eta}\big(\hat{u}_p^{(1+\f{1+k}{3})}-\hat{A}_{1+\f{1+k}{3}},\hat{v}_p^{(1+\f{3+k}{3})}\big)\big|^2\langle\eta\rangle^{2\ell} dx d\eta\leq C_{j,k,\ell}\dl^2,\\
&\int_0^{2\pi}\hat{v}_p^{(1+\f{3+k}{3})}(x,\eta)dx=0, \ \forall\ \eta\geq 0,
\end{aligned}
\end{align*}
where $\hat{A}_{1+\f{1+k}{3}}=\lim\limits_{\eta\rightarrow+\i}\hat{u}_p^{(1+\f{1+k}{3})}$.
\end{proposition}

\pf

We set
\bes
 \hat{p}_p^{(1+\f{3+k}{3})}=-\int^{+\i}_\eta g_{\text{lower},k}(x,\bar{\eta})d\bar{\eta}+A\int^x_0\hat{v}^{(1+\f{3+k}{3})}_p(\bar{x},0)d\bar{x},
\ees
and define
\bes
\hat{u}=\hat{u}^{(1+\f{k+1}{3})}_p+\kappa(\eta)u^{(1+\f{1+k}{3})}_e(x,0),\q \hat{v}=\hat{v}^{(1+\f{3+k}{3})}_p-\int^\eta_0\kappa(\bar{\eta})d\bar{\eta}\p_xu^{(1+\f{1+k}{3})}_e(x,0).
\ees
Then system \eqref{boundarylayerlk} can be reformulated into
 \be\label{boundarylayerlkf}
 \lt\{
 \begin{aligned}
 &A\eta \p_x \hat{u}+A\hat{v}-\p^2_{\eta}\hat{u}=\t{f}_{\text{lower},k},\\
 &\p_x \hat{u}+\p_\eta\hat{v}=0,\\
 &\hat{u}\big|_{\eta=0}=0\q \p_\eta\hat{u}\big|_{\eta=+\i}=\hat{v}\big|_{\eta=+\i}=0,
 \end{aligned}
 \rt.
 \ee
where $\t{f}_{\text{lower},k}(x,\eta)$ is defined as
\begin{align*}
\t{f}_{\text{lower},k}:=&{f}_{\text{lower},k}+A\p_xu^{(1+\f{1+k}{3})}_e(x,0)\lt(\eta\kappa(\eta)-\int^\eta_0\kappa(\bar{\eta})d\bar{\eta}\rt)\nn\\
&-\kappa^{''}(\eta)u^{(1)}_e(x,0)+\int^{+\i}_\eta \p_xg_{\text{lower},k}(x,\bar{\eta})d\bar{\eta},
\end{align*}
which decays fast enough at $\eta$ infinity. And $\hat{v}$ can be determined by using the divergence-free condition, for $1\leq k\leq 19$, $$\hat{v}^{(1+\f{3+k}{3})}=\int_\eta^\infty \p_x\hat{u}^{(1+\f{1+k}{3})}(x,\eta') d\eta',$$
and for $k=20$, we choose
$$\hat{v}^{(1+\f{3+k}{3})}=-\int_0^\eta \p_x\hat{u}^{(1+\f{1+k}{3})}(x,\eta') d\eta'.$$

  The system \eqref{boundarylayerlkf} is the same as that in \eqref{boundarylayerl1f} and can be obtained the solvability and decay property of the solution by the same way.
\qed
\subsubsection{Correction of the boundary layer solutions and the Euler equations of higher order}
\indent

Since we want the boundary layer to decay fast at infinity, we define for $k\geq 1$,
\bes
\t{u}^{(1+\f{1+k}{3})}_p:=u^{(1+\f{1+k}{3})}_p-A_{1+\f{1+k}{3}},\q \t{\hat{u}}^{(1+\f{1+k}{3})}_p:=\hat{u}^{(1+\f{1+k}{3})}_p-\hat{A}_{1+\f{1+k}{3}},
\ees
and use $\t{u}^{(1+\f{1+k}{3})}_p$ and $\t{\hat{u}}^{(1+\f{1+k}{3})}_p$ to replace the original ${u}^{(1+\f{1+k}{3})}_p$ and ${\hat{u}}^{(1+\f{1+k}{3})}_p$ respectively. Then we need to correct $u^{(1+\f{1+k}{3})}_e$ to $\t{u}^{(1+\f{1+k}{3})}_e$ such that
\begin{align}
&\t{u}^{(1+\f{1+k}{3})}_e\big|_{y=1}=-u^{(1+\f{1+k}{3})}_p\big|_{\zeta=0}+A_{1+\f{1+k}{3}}={u}^{(1+\f{1+k}{3})}_e\big|_{y=1}+A_{1+\f{1+k}{3}}\nn\\
&\t{u}^{(1+\f{1+k}{3})}_e\big|_{y=0}=-\hat{u}^{(1+\f{1+k}{3})}_p\big|_{\eta=0}+\hat{A}_{1+\f{1+k}{3}}={u}^{(1+\f{1+k}{3})}_e\big|_{y=0}+\hat{A}_{1+\f{1+k}{3}},\nn
\end{align}
and $(\t{u}^{(1+\f{1+k}{3})}_e,v^{(1+\f{1+k}{3})}_e)$ is still a solution of the system \eqref{middleeulerk}.

Let
\bes
\t{u}^{(1+\f{1+k}{3})}_e={u}^{(1+\f{1+k}{3})}_e+\phi^{(1+\f{1+k}{3})}(y).
\ees
Obviously, $(\t{u}^{(1+\f{1+k}{3})}_e, v^{(1+\f{1+k}{3})}_e)$ is still a solution of the system \eqref{middleeulerk}. By using the incompressibility and harmonic property of $v^{(1+\f{1+k}{3})}_e$, we can see that
\bes
\p_x\Dl {u}^{(1+\f{1+k}{3})}_e=-\p_y\Dl {v}^{(1+\f{1+k}{3})}_e=0.
\ees
So we choose $\phi^{(1+\f{1+k}{3})}$ to satisfy
\bes
\lt\{
\bali
& \Dl \t{u}^{(1+\f{1+k}{3})}_e=\Dl {u}^{(1+\f{1+k}{3})}_e+\phi^{(1+\f{1+k}{3})}_{yy}=0,\\
& \phi^{(1+\f{1+k}{3})}(1)=A_{1+\f{1+k}{3}},\q \phi^{(1+\f{1+k}{3})}(0)=\hat{A}_{1+\f{1+k}{3}}.
\eali
\rt.
\ees
Then we have
\begin{itemize}
\item $(\t{u}^{(1+\f{1+k}{3})}_e,v^{(1+\f{1+k}{3})}_e)$ is still a solution of system \eqref{middleeulerk};
\item $\t{u}^{(1+\f{1+k}{3})}_e\big|_{y=1}+\t{u}^{(1+\f{1+k}{3})}_p\big|_{\zeta=0}=0 ,\q \t{u}^{(1+\f{1+k}{3})}_e\big|_{y=0}+\t{\hat{u}}^{(1)}_p\big|_{\eta=0}=0$;
\item  $\Dl \t{u}^{(1+\f{1+k}{3})}_e=0$.
\end{itemize}

Then we use $\t{u}^{(1+\f{1+k}{3})}_p, \t{\hat{u}}^{(1+\f{1+k}{3})}_p, \t{u}^{(1+\f{1+k}{3})}_e$ to be the new ${u}^{(1+\f{1+k}{3})}_p, \hat{u}^{(1+\f{1+k}{3})}_p, {u}^{(1+\f{1+k}{3})}_e$ to enter the next round iteration.
\qed

\subsection{Approximate solutions}\label{Approximate solutions}

\indent

In this subsection, we construct an approximate solution of the Navier-Stokes equations (\ref{ns}). First let $\chi(y)\in C^\i_c([0,+\i))$ be a function satisfying
\bes
\chi(y)=\lt\{
\bali
&1,\q y\in [0,1/4],\\
&0,\q y\geq 3/4.
\eali
\rt.
\ees
Set
\begin{align*}
{u}_{p,\text{com}}^a(x,y):=& (1-\chi(y))^2 \lt[u_p^{(0)}+\epsilon {u}_p^{(1)}+\sum^{20}_{k=1}\e^{1+\f{1+k}{3}}u^{(1+\f{1+k}{3})}_p\rt]\\[5pt]
 &+ \chi^2(y) \lt[\epsilon \hat{u}_p^{(1)}+\sum^{20}_{k=1}\e^{1+\f{1+k}{3}}\hat{u}^{(1+\f{1+k}{3})}_p \rt]\\
:=& (1-\chi(y))^2 u_p^a+ \chi^2(y)\hat{u}_p^a,\\[5pt]
{v}_{p,\text{com}}^a(x,y):=&(1-\chi(y))^2\lt[\epsilon {v}_p^{(1)}+\epsilon^2 {v}_p^{(2)}+\sum^{20}_{k=1}\e^{2+\f{1+k}{3}}v^{(2+\f{1+k}{3})}_p\rt]\\
  &+\chi^2(y)\lt[\epsilon^{1+\f{2}{3}} \hat{v}_p^{(1+\f{2}{3})}+\sum_{k=1}^{20}\e^{1+\f{3+k}{3}}\hat{v}^{(1+\f{3+k}{3})}_p\rt]\\[5pt]
:=&(1-\chi(y))^2 v_p^a+\chi^2(y) \hat{v}_p^a,\\[5pt]
{p}_{p,\text{com}}^a(x,y):=&(1-\chi(y))^4\lt[\epsilon {p}_p^{(1)}+\epsilon^2 {p}_p^{(2)}+\sum^{20}_{k=1}\e^{2+\f{1+k}{3}}p^{(2+\f{1+k}{3})}_p\rt]\\
        &+\chi^4(y)\lt[\epsilon^{1+\f{2}{3}} \hat{p}_p^{(1+\f{2}{3})}+\sum_{k=1}^{20}\e^{1+\f{3+k}{3}}\hat{p}^{(1+\f{3+k}{3})}_p\rt]\\[5pt]
:=&(1-\chi(y))^4 p_p^a+\chi^4(y) \hat{p}_p^a,
\end{align*}
and
\begin{align*}
u_e^a:=&Ay+\e u^{(1)}_e+\sum^{20}_{k=1}\e^{1+\f{1+k}{3}}u^{(1+\f{1+k}{3})}_e, \\[5pt]
 v_e^a:=&\e v^{(1)}_e+\sum^{20}_{k=1}\e^{1+\f{1+k}{3}}v^{(1+\f{1+k}{3})}_e,\\[5pt]
 p_e^a:=&\e p^{(1)}_e+\sum^{20}_{k=1}\e^{1+\f{1+k}{3}}p^{(1+\f{1+k}{3})}_e.
\end{align*}

We construct an approximate solution $(u^a,v^a,p^a)$ by
\begin{align*}
u^a(x,y):&=u_e^a+{u}_{p,\text{com}}^a+\epsilon^{9}h(x,y),\\[5pt]
v^a(x,y):&=v_e^a+{v}_{p,\text{com}}^a,\\[5pt]
p^a(x,y):&=p_e^a(x,y)+{p}_{p,\text{com}}^a,
\end{align*}
where the corrector $h(x,y)$ will be given in Appendix, which satisfies
\begin{align*}
h(x,0)=h(x,1)=0, \ \|\partial_x^j\partial_y^k h\|_2\leq C_{j,k}\epsilon^{-k},
\end{align*}
and makes $(u^a, v^a)$ be divergence-free
\begin{align*}
\p_x u^a+\p_yv^a=0.
\end{align*}
Moreover, $(u^a, v^a)$  satisfies the following boundary conditions
\begin{align*}
 u^a(x+2\pi,y)&=u^a(x,y), \ v^a(x+2\pi,y)=v^a(x,y),\\[5pt]
u^a(x, 1)&=\alpha+\dl f(x), \ v^a(x,1)=0,\\[5pt]
u^a(x,0)&=0,\ v^a(x,0)=0.
\end{align*}
By collecting the estimates in Proposition \ref{propeulerorder1}, Proposition \ref{propeulerorderk} and using the Sobolev embedding,  we deduce that
\begin{align}
&\|\p^j_x\p^k_y(u^a_e-Ay)\|_\infty\leq C_{j,k}\epsilon(\dl+\epsilon), \q \|\p^j_x\p^k_yv^a_e\|_\infty\leq C_{j,k}\epsilon(\dl+\epsilon).\nn
\end{align}
By collecting the estimates in Proposition \ref{propdcu0}, Proposition \ref{propdcu1} and Proposition \ref{propdcuk}, one has
\begin{align}
\|\zeta^\ell\p^{j}_x\partial_\zeta^k u_p^a\|_\infty\leq C_{j,k,\ell} (\dl+\epsilon), \ \|\zeta^\ell\p^{j}_x\partial_\zeta^k v_p^a\|_\infty\leq C_{j,k,\ell} \epsilon(\dl+\epsilon).\nn
\end{align}
Similarly, there holds
\begin{align*}
\|\eta^\ell\p^{j}_x\partial_\eta^k \hat{u}_p^a\|_\infty\leq C_{j,k,\ell} \e(\dl+\epsilon^{2/3}),  \ \|\eta^\ell\p^{j}_x\partial_\eta^k \hat{v}_p^a\|_\infty\leq C_{j,k,\ell} \epsilon^{1+2/3}(\dl+\epsilon^{2/3}).
\end{align*}

Finally, set
 \begin{align*}
 R_u^a:&=u^a\p_xu^a+v^a\p_yu^a+\p_xp^a-\epsilon^2\Dl u^a,\\[5pt]
R_v^a:&=u^a\p_xv^a+v^a\p_yv^a+\p_y p^a-\epsilon^2\Dl v^a,
 \end{align*}
 then there holds
 \begin{align*}
 \|R_u^a\|_2+\|\partial_x R_u^a\|_2\leq C\varepsilon^9 , \  \|R_v^a\|_2+\|\partial_x R_v^a\|_2\leq C\epsilon^9,
 \end{align*}
and $(u^a, v^a, p^a)$ satisfies
\begin{eqnarray*}
\left\{
\begin{array}{lll}
u^a\p_xu^a+v^a\p_yu^a+\p_xp^a-\epsilon^2\Dl u^a=R_u^a,\ (x,y)\in \bT\times[0,1],\\[5pt]
u^a\p_xv^a+v^a\p_yv^a+\p_y p^a-\epsilon^2\Dl v^a=R_v^a,\ (x,y)\in \bT\times[0,1],\\[5pt]
 \p_xu^a+\p_yv^a=0, \ (x,y)\in \bT\times[0,1], \\[5pt]
 u^a(x+2\pi,y)=u^a(x,y), \ v^a(x+2\pi,y)=v^a(x,y), \ (x,y)\in \bT\times[0,1],\\[5pt]
u^a(x,1)=\alpha+\dl f(x),\ v^a(x,1)=0, \ x\in [0,2\pi], \\[5pt]
 u^a(x,0)=v^a(x,0)=0, \ x\in [0,2\pi].
\end{array}
\right.
\end{eqnarray*}
This is exactly the constructed approximate solution that we stated in the beginning of Section \ref{sec2}. \qed

\section{Appendix}

In this section, we give a construction of corrector $h(x,y)$ defined in section \ref{Approximate solutions}. Firstly, we give a simple lemma which is similar to Lemma 6.1 in Appendix B in \cite{FGLT}.

\begin{lemma}\label{corector equation}
Assume that $K(x,y)$ is a $2\pi$-periodic smooth function which satisfies
\beno
\int_0^{2\pi}K(x,y)dx=0, \ \forall\ y\in [0,1]; \quad K(x, 0)=K(x, 1)=0,
\eeno
then there exists a $2\pi$-periodic function $h(x,y)$ such that
\begin{align}\label{correctorh}
&\partial_x h(x,y)=K(x,y); \ h(x,0)=h(x, 1)=0;\nonumber\\
&\int_0^{2\pi}h(x,y)dx=0, \ \|\partial_x^j\partial_y^kh\|_2\leq C\|\partial_x^j\partial_y^kK\|_2.
\end{align}
\end{lemma}
\begin{proof}
By the Fourier series expansion, we have
\beno
K(x,y)=\sum_{n\neq 0}K_n(y)e^{in x}, \quad K_n(0)=K_n(1)=0.
\eeno
Set
\beno
h(x,y)=\sum_{n\neq 0}\frac{K_n(y)}{in}e^{i n\theta}.
\eeno
It's easy to justify that $h(x,y)$ satisfies (\ref{correctorh}) which completes the proof.
\end{proof}

Next, we construct the corrector $h(x,y)$ by the above lemma.
Direct computation gives
\begin{align*}
&\p_x(u^a_e+ u^a_{p,\text{com}})+\p_y(v^a_e+v^a_{p,\text{com}})\nn\\
=&2\chi(y)\chi^{\prime}(y)(v^a_p+\hat{v}^a_p)-2\chi^{\prime}(y)v^a_p,\nn\\
          =&-\e^9 2\chi(y)\chi^{\prime}(y)\lt[(y-1)^{-9}\zeta^9 v^a_p+y^{-9}\eta^9\hat{v}^a_p\rt]-\e^9 2\chi^{\prime}(y)y^{-9}\eta^9v^a_p\\
          :=&-\e^9 K(x,y).
\end{align*}

Notice that $\chi'(r)=0,\ \text{for } r\in [0,1/4]\cup [3/4,1]$ and the properties of $v^a_p,  \hat{v}^a_p$, we know that $K(x,y)$ satisfies the assumption in Lemma \ref{corector equation}, then there exists  $h(x,y)$ such that
\beno
\p_x(u^a_e+ u^a_{p,\text{com}})+\p_y(v^a_e+v^a_{p,\text{com}})=-\e^9 \p_x h(x,y),
\eeno
which indicates that
\bes
\p_x u^a+\p_yv^a=0.
\ees

\section*{Data availability statement}

Data sharing is not applicable to this article as no datasets were generated or analysed during the current study.

\section*{Conflict of interest statement}

The authors declare that they have no conflict of interest.

\section*{Acknowledgments}

 M. Fei is partially supported by NSF of China under Grant No.12271004 and  NSF of Anhui Province of China under Grant No. 2308085J10 and No.gxbjZD2022009. X. Pan is supported by NSF of China  under Grant  No.11801268 and No.12031006 and the Fundamental Research Funds for the Central Universities of China under Grant No.NS2023039. We thank Dr. Chen Gao for helpful discussions on this topic.


\begin{thebibliography}{00}

\bibitem {Batchelor:1956} {\sc G. K. Batchelor}: On steady laminar flow with closed streamlines at large Reynolds number. {\it J.Fluid Mech.}, 7(1956), no.1, 177--190.

 \bibitem {BedrossianM:2015PMI} {\sc J. Bedrossian and N. Masmoudi}: Inviscid damping and the asymptotic stability of planar shear flows in the 2D Euler equations. {\it Publ. Math. Inst. Hautes \'{E}tudes Sci.} 122(2015), 195--300.

\bibitem {BedrossianMV:2016ARMA} {\sc J. Bedrossian, N. Masmoudi and V. Vicol}: Enhanced dissipation and inviscid damping in the inviscid limit of the Navier-Stokes equations near the two dimensional Couette flow. {\it Arch. Ration. Mech. Anal.} 219(2016), no.3, 1087--1159.

 \bibitem {BedrossianVW:2018JNS} {\sc J. Bedrossian, V. Vicol and F. Wang}: The Sobolev stability threshold for 2D shear flows near Couette. {\it J. Nonlinear Sci.} 28(2018), no.6, 2051--2075.

\bibitem {BedrossianH:2020CMP} {\sc J. Bedrossian and S. He}: Inviscid damping and enhanced dissipation of the boundary layer for 2D Navier-Stokes linearized around Couette flow in a channel. {\it Comm. Math. Phys.} 379(2020), no.1, 177--226.

\bibitem {CWZ2022} {\sc Q. Chen, D. Wu and Z. Zhang}: On the $L^\infty$ stability of Prandtl expansions in Gevery class. {\it Sci. China Math.}, 65(2022), 2521--2562.

\bibitem {ChenWZ:2023SCM}  {\sc Q. Chen, D. Wu and Z. Zhang}: On the stability of shear flows of Prandtl type for the steady Navier-Stokes equations. {\it Sci. China Math.} 66 (2023), no. 4, 679--722.



\bibitem {DM:2019} {\sc A. Dalibard and N.Masmoudi}: {Separation for the stationary Prandtl equation}. {\sc Publ. Math. Inst. Hautes \'{E}tudes Sci.}, 130(2019), 187--297.


\bibitem {DLX} {\sc S. Ding, Z. Lin and F. Xie}: Verification of Prandtl boundary layer ansatz for steady electrically conducting fluids with a moving physical boundary. {\it SIAM J. Math. Anal.}, 53(2021), 4997--5059.

\bibitem {DJL} {\sc S. Ding, Z. Ji and Z. Lin}: Validity of Prandtl layer theory for steady magnetohydrodynamics over a moving plate with nonshear outer ideal MHD flows. {\it J. Differential Equations}, 278(2021), 220--293.


\bibitem {DrivasIN:2023ARXIV} {\sc T. D. Drivas, S. Iyer, T. T. Nguyen}: The Feynman-Lagerstrom criterion for boundary layers. {\it arXiv}: 2308.15447.

\bibitem{FGLT} {\sc M. Fei, C. Gao, Z. Lin and T. Tao}: Prandtl-Batchelor flows on a disk, {\it Comm. Math. Phys.} 397 (2023), no. 3, 1103--1161.

\bibitem{FeiGLT:2021ARXIV} {\sc M. Fei, C.Gao, Z. Lin and T. Tao}: Prandtl-Batchelor flows on an annulus, {\it arXiv}: 2111.07114.

\bibitem {FTZ2018} {\sc M. Fei, T. Tao and Z. Zhang}: On the zero-viscosity limit of the Navier-Stokes equations in $\mathbb{R}_{+}^3$ without analyticity. {\it J. Math. Pures Appl.}, 112 (2018), 170--229.

 \bibitem {GX2023} {\sc C. Gao and Z. Xin}: Prandtl boundary layers in an infinitely long convergent channel. {\it arXiv}: 2308.03440, 2023.

\bibitem {GaoZ:2023SCM} {\sc C. Gao and L. Zhang}: On the steady Prandtl boundary layer expansions, {\it Sci. China Math.} 66 (2023), no. 9, 1993--2020.

\bibitem {GaoZ:2021ARXIV} {\sc C. Gao and L. Zhang}: Remarks on the steady Prandtl boundary layer expansion, {\it arXiv}: 2107.08372.



\bibitem {GerardM:2019ARMA} {\sc D. G\'erard-Verat and Y. Maekawa}, Sobolev stability of Prandtl expansions for the steady Navier-Stokes equations. {\it Arch. Ration. Mech. Anal.}, 233 (2019), no. 3, 1319--1382.


\bibitem {GMM2018} {\sc D. G\'erard-Varet, Y. Maekawa and N. Masmoudi}: Gevrey stability of Prandtl expansions for 2D Navier-Stokes flows. {\it Duke Math. J.}, 167(2018), 2531--2631.


\bibitem {GM2015} {\sc  D. G\'erard-Varet and N. Masmoudi}: Well-posedness for the Prandtl system without analyticity or monotonicity. {\it Annales Scientifiques de l'Ecole Normale Superieure}, 48(2015), 1273--1325.



 \bibitem {G2000} {\sc E. Grenier}: On the nonlinear instability of Euler and Prandtl equations. {\it Comm. Pure Appl. Math.}, 53(2000), 1067--1091.

 \bibitem {GGN2016} {\sc E. Grenier, Y. Guo and T. Nguyen}: Spectral instability of characteristic boundary layer flows. {\it Duke Math. J.}, 165(2016), 3085--3146.


 \bibitem {GN2019} {\sc E. Grenier and T. Nguyen}: $L^\infty$  instability of Prandtl layers. {\it Ann. PDE}, 5(2019), 36pp.

% b

\bibitem {GuoN:2017ANNPDE} {\sc Y. Guo and T. Nguyen}:  Prandtl boundary layer expansions of steady Navier-Stokes flows over a moving plate. {\it Ann. PDE} 3 (2017), no. 1, Paper No. 10, 58 pp

\bibitem {GuoI:2023CPAM} {\sc Y. Guo and S. Iyer}: Validity of steady Prandtl layer expansions. {\it Comm. Pure Appl. Math.}, 76 (2023), no. 11, 3150--3232.

\bibitem {IyerZ:2019JDE} {\sc S. Iyer and C. Zhou}: Stationary inviscid limit to shear flows. {\it J. Differential Equations}, 267 (2019), no. 12, 7135--7153.

\bibitem {IyerZ:2019JDE1}{\sc S. Iyer and C. Zhou}: Corrigendum to  ``Stationary inviscid limit to shear flows'' [J. Differ. Equ. 267 (12) (2019) 7135--7153], {\it J. Differential Equations} 289 (2021), 279--288.

\bibitem {Iyer:2017ARMA} {\sc S. Iyer}: Steady Prandtl boundary layer expansions over a rotating disk, {\it Arch. Ration. Mech. Anal.}, 224 (2017), no.2, 421--469

\bibitem {Iyer:2019Peking} {\sc S. Iyer}: Global steady Prandtl boundary layer over a moving boundary, {\it Peking Math. J.}, I: 2 (2019), no. 2, 155--238; II: 2 (2019), no. 3-4, 353--437; III: 3 (2020), no. 1, 47--102.

\bibitem {Iyer:2019SIAM} {\sc S. Iyer}: Steady Prandtl boundary layer over a moving boundary: nonshear Euler flow. {\it SIAM, J.Math.Anal.}, 51(2019), no.3, 1657--1685

\bibitem {IyerM:2021AIA} {\sc S. Iyer and N. Masmoudi}:  Boundary layer expansions for the stationary Navier-Stokes equations. {\it Ars Inven. Anal.} 2021, Paper No. 6, 47 pp.

\bibitem {IyerM:2020ARXIV} {\sc S. Iyer and N. Masmoudi}: Global-in-$x$ stability of steady Prandtl expansions for 2D Navier-Stokes flows. {\it arXiv}: 2008.12347.

\bibitem {Kim-thesis} {\sc S.C.Kim}: {On Prandtl-Batchelor theory of steady flow at large Reynolds number}, Ph.D thesis, New York University, 1996.

\bibitem {Kim:1998SIAM} {\sc S. C. Kim}: On Prandtl-Batchelor theory of a cylindrical eddy: asymptotic study. {\it SIAM J. Appl. Math.}, 58(1998), 1394--1413.

\bibitem {KimC:2001SIAM} {\sc S. C. Kim and S. Childress}: {Vorticity selection with multiple eddies in two-dimensional steady flow at high Reynolds number}. {\it SIAM J. Appl. Math.}, 61 (2001), no. 5, 1605--1617.

% a

 \bibitem {KNVW2022} {\sc I. Kukavica, T. T. Nguyen, V. Vicol and F. Wang}: On the Euler+Prandtl expansion for the Navier-Stokes equations. {\it J. Math. Fluid Mech.}, 24(2022), 46pp.

   \bibitem {KVW2020} {\sc I. Kukavica, V. Vicol and F. Wang}: The inviscid limit for the Navier-Stokes equations with data analytic only near the boundary. {\it Arch. Ration. Mech. Anal.}, 237(2020), 779-827.



   \bibitem {LYZ2023} {\sc C. Liu, T. Yang and Z. Zhang}: Validity of Prandtl expansions for steady MHD in the Sobolev framework. {\it SIAM J. Math. Anal.} 55(2023), 2377--2410.



  \bibitem {M2014} {\sc Y. Maekawa}:  On the inviscid limit problem of the vorticity equations for viscous incompressible flows in the half-plane. {\it Comm. Pure Appl. Math.}, 67(2014),  1045--1128.



  \bibitem {NN2018} {\sc T. T. Nguyen and T. T. Nguyen}: The inviscid limit of Navier-Stokes equations for analytic data on the half-space. {\it Arch. Ration. Mech. Anal.}, 230(2018), 1103--1129.


\bibitem {OleinikS:1999} {\sc O. A. Oleinik and V. N. Samokhin}: {\it Mathematical models in boundary layer theory}. Applied Mathematics and Mathematical Computation, 15. Champan and Hall/CRC, Boca Raton, FL, 1999.


\bibitem {Prandtl:1904} {\sc L. Prandtl}: \"{U}ber Fl\"{u}ssigkeitsbewegung bei sehr kleiner Reibung. Verhandlungen des III. Internationalen Mathematiker
Kongresses, Heidelberg, 1904, pp. 484--491, Teubner, Leizig. See Gesammelte Abhandlungen II, pp. 575--584.

 \bibitem {SC1998-1} {\sc M. Sammartino and R. -E. Caflisch}: Zero viscosity limit for analytic solutions of the Navier-Stokes equation on a half-space. I. Existence for Euler and Prandtl equations, {\it Comm. Math. Phys.}, 192(1998), 433--461.

  \bibitem {SC1998-2} {\sc M. Sammartino and R.-E. Caflisch}: Zero viscosity limit for analytic solutions of the Navier-Stokes equation on a half-space. II. Construction of the Navier-Stokes solution,  {\it Comm. Math. Phys.}, 192(1998), 463--491.


\bibitem {SchlichtingG:2017} {\sc H. Schlichting and K. Gersten}: {\it Boundary-layer theory}. Ninth edition. With contributions from Egon Krause and Herbert Oertel Jr. Translated from the German by Katherine Mayes. Springer-Verlag, Berlin, 2017. xxviii+805 pp.

\bibitem {WV:2007} {\sc L.Van Wijngaarden}: Prandtl-Batchelor flows revisited. {\it Fluid Dyn. Res.}, 39 (2007), 267--278.


\bibitem {ShenWZ:2021} {\sc W. Shen, Y.Wang and Z.Zhang}: {Boundary layer separation and local behavior for the steady Prandtl equation}, {\it Adv. Math.}, 389 (2021), 107896, 25pp.



\bibitem {WangZ:2021AIHP} {\sc Y. Wang and Z. Zhang}: Global $C^{\infty}$ regularity of the steady Prandtl equation with favorable pressure gradient, {\it  Ann. Inst. H. Poincar\'{e} Anal. Non Lin\'{e}aire}, 38 (2021), no. 6, 1989--2004.

\bibitem {GuoWZ:2023ANNPDE}  {\sc Y. Guo, Y. Wang and Z. Zhang}: Dynamic stability for steady Prandtl solutions, {\it Ann. PDE} 9 (2023), no. 2, Paper No. 16, 33 pp.

\bibitem {WangZ:2023MATHANN}  {\sc Y. Wang and Z. Zhang}: Asymptotic behavior of the steady Prandtl equation, {\it Math. Ann.} 387 (2023), no. 3-4, 1289--1331.



 \bibitem {WWZ2017} {\sc C. Wang, Y. Wang and Z.-F. Zhang}: Zero-viscosity limit of the Navier-Stokes equations in the analytic setting, {\it Arch. Ration. Mech. Anal.}, 224(2017), 555--595.


\bibitem {Wood:1957JFM} {\sc W. W. Wood}: Boundary layers whose streamlines are closed, {\it J. Fluid Mech}, 01(1957), no.2, 77--87.



\bibitem {ZZ2016} {\sc P. Zhang and Z. Zhang}: Long time well-posedness of Prandtl system with small and analytic initial data, {\it J. Funct. Anal.}, 270(2016), 2591--2615.


\end{thebibliography}
\end{document}